%% file: besovtreeuq_R2.tex
\documentclass[10pt,a4paper]{article}
\usepackage{a4wide}
\usepackage{amsfonts}
\usepackage{amscd}
\usepackage{amssymb}
\usepackage{amsthm}
\usepackage{amsmath}
\usepackage{latexsym}
\usepackage{dsfont}
\usepackage{hyperref}
\usepackage[usenames,dvipsnames]{color}
\usepackage{esvect} 
\usepackage{algorithm}
\usepackage{algpseudocode}
\usepackage{enumerate}
\usepackage[dvipsnames]{xcolor}
\usepackage{bm}

\usepackage{subfigure}
\usepackage{graphicx}

\usepackage{caption}
\newcommand{\as}[1]{\textcolor{black}{#1}}
\usepackage{xcolor}
\newcommand{\rev}[1]{\textcolor{black}{#1}}

\input{macros_as.tex}
\newcommand{\mf}{\mathfrak } 

\usepackage{forest}

\begin{document}
	
	\title{
	Multilevel \MC\, FEM\, \\ for \\ Elliptic PDEs with Besov Random Tree Priors}

	\date{}
	
	\author{
		Christoph Schwab \footnote{ETH Zürich, Seminar for Applied Mathematics}
		\and
		Andreas Stein 
		\addtocounter{footnote}{0}\footnotemark[\value{footnote}]  \footnote{andreas.stein@sam.math.ethz.ch}
		\footnote {The datasets generated during and/or analysed during the current study are available from the corresponding author on reasonable request.}
	}
	
	\maketitle
	
\begin{abstract}
We develop a multilevel Monte Carlo (MLMC)-FEM algorithm for linear, elliptic diffusion problems in polytopal domain $\mathcal D\subset \mathbb R^d$, with Besov-tree random coefficients.
This is to say that the logarithms of the diffusion coefficients are sampled from 
so-called Besov-tree priors, which have recently been proposed to model data for fractal phenomena in science and engineering.
Numerical analysis of the fully discrete FEM for the elliptic PDE includes quadrature 
approximation and must account for a) nonuniform pathwise upper and lower coefficient 
bounds, and for b) low path-regularity of the Besov-tree coefficients.

Admissible non-parametric random coefficients correspond to random functions exhibiting singularities on random fractals with tunable fractal dimension, but involve no a-priori specification of the 
fractal geometry of singular supports of sample paths.
Optimal complexity and convergence rate estimates for quantities of interest and for their second moments are proved. 
A convergence analysis for MLMC-FEM is performed which yields choices of the algorithmic steering parameters for efficient implementation.
A complexity (``error vs work'') analysis of the MLMC-FEM approximations is provided.
\end{abstract}

\section{Introduction}
\label{sec:intro}
The efficient numerical solution of partial differential equations with uncertain 
inputs is key in \emph{forward uncertainty quantification}, i.e.,
the computational quantification of uncertainty of solutions to PDEs with uncertain inputs.
It is also crucial in \emph{computational inverse uncertainty quantification},
e.g. via Markov chain \MC\, methods, where numerous numerical solves of the forward model 
subject to realizations of the uncertain input are required.
Here, we consider the linear, elliptic diffusion with uncertain coefficient.
It models a wide range of phenomena such as diffusion through a medium with uncertain
or even unknown permeability, e.g. in subsurface flow, light scattering in dust clouds, 
to name but a few. 
Physical modelling of subsurface flow in particular stipulates
systems of fractures of uncertain geometry with high permeability along fractures 
(see, e.g., \cite{FracSurvey2019} and the references there).
With fracture geometry being only statistically known, 
it is natural in computational \UQ\, (UQ) to specify the geometry in a nonparametric fashion, 
rather than, for instance, through a Gaussian random field (GRF for short) with a known, 
parametric two-point correlation to be calibrated from experimental data.
This function space perspective has also become topical recently in the context of 
inverse imaging noisy signals.
Modelling with random, fractal geometries also has found applications in biology 
(roots, lungs \cite{AchdouRamif}).
There, Gaussian parametric models have been found computationally efficient 
due to the availability of padding and circulant embedding based numerics,
enabling the use of fast Fourier transform algorithms for sample path generation.
However, Gaussian models 
are perceived as inadequate for the efficient representation of edges and interfaces in imaging. 
Accordingly, non-parametric representations of inputs with fractal irregularities in 
sample paths have been proposed recently, e.g. in \cite{KLSS21,HL11}, 
and the references there. 
We also mention the so-called \emph{Besov priors} in 
Bayesian inverse problems with elliptic PDE constraints 
(e.g. \cite{LassasBesov09,HosNig17,AgaDasHel2021} and the references there).

In the present paper, we investigate so-called \emph{Besov random tree priors} \cite{KLSS21}, as stochastic log-diffusion coefficient in a linear elliptic PDE. These priors are given by a wavelet series with stochastic coefficients, and certain terms in the expansion vanishing at random, according to the law of so-called \emph{Galton-Watson trees}.
Samples of the corresponding random fields involve fractal geometries, hence the Besov random tree prior may be a viable candidate in applications, 
where models based on GRFs do not allow for sufficiently flexibility. 
We quantify the pathwise regularity of the random tree prior in terms of Hölder-regularity, and investigate the forward propagation of the uncertainty in the elliptic PDE model in a bounded domain. 
All results in the present article encompass the "standard" Besov prior from~\cite{LassasBesov09} as special case, when no terms in the wavelet series are eliminated. 
As we point out in our analysis, regularity is inherently low, both with respect to the spatial and stochastic domain of the random field. 
This is taken into account when developing efficient numerical methods for the elliptic PDE problem at hand.

We develop a \MLMC\, (MLMC) Finite Element (FE) simulation algorithm and furnish its mathematical analysis for the estimation of quantities of interest (QoI) in the forward PDE model. 
\MMLMC\, methods (\cite{giles2008multilevel, giles2015multilevel, barth2011multi}) are, by now, 
a well-established methodology in computational UQ, and are effective in regimes with comparably low regularity. 
\rev{We mention that our MLMC FE error analysis includes the case of standard low-regularity Besov priors as a special case.}
In contrast, higher-order methods that consider an equivalent parametric, deterministic PDE problem,
such as (multilevel) Quasi-\MC\, (\cite{herrmann2019multilevel}), 
generalized polynomial chaos (gPC) expansions (\cite{hoang2013sparse}), 
or multilevel Smolyak quadrature (\cite{zech2020convergence}) 
are \emph{not suitable} in the present random tree model: 
The parametrization of the prior involves discontinuities in the stochastic domain, which strongly violates the regularity requirements of the aforementioned higher-order methods.
\rev{We refer to Appendix~\ref{appendix:parametrization} for a detailed discussion.}
On the other hand, MLMC techniques merely require square-integrability of the pathwise solution 

\subsection{Contributions}
\label{sec:Contr}
For a model linear elliptic diffusion equation, 
in a polytopal domain $\cD\subset \bR^d$, 
we provide the mathematical setting and the numerical analysis
of a MLMC-FEM for diffusion in random media with log-fractal Besov random tree structure.
In particular, we establish well-posedness of the forward problem including strong measurability of
random solutions (a key ingredient in the ensuing MLMC-FE convergence analysis),
and pathwise almost sure Besov regularity of weak solutions.
Technical results of independent interest include:
(i) Bounds on exponential moments of Besov random variables in Hölder norms, 
    generalizing results in \cite{KLSS21, DHS12, LassasBesov09},
(ii) Numerical analysis of elliptic forward problems with fractal coefficient, 
     in particular bounds on the fractal scale truncation error and on the 
    finite element approximation error, as well as the impact of numerical quadrature in view
    of low (H\"older) path regularity of the random coefficients,
(iii) a complete MLMC-FE convergence analysis, for estimating the mean of non-linear functionals of the random solution field.
    
\subsection{Preliminaries and Notation}
\label{sec:Notat}
We denote by $\cV'$ the topological dual for any vector space 
$\cV$ and by $\dualpair{\cV'}{\cV}{\cdot}{\cdot}$ the associated dual pairing.
We write $\cX\hookrightarrow\cY$ for two normed spaces $(\cX, \left\|\cdot\right\|_\cX), (\cY, \left\|\cdot\right\|_\cY)$, if $\cX$ is continuously embedded in $\cY$, i.e., 
there exists $C>0$ such that $\|\varphi\|_\cY\le C\|\varphi\|_\cX$ holds for all $\varphi\in\cX$.
The Borel $\gs$-algebra of any metric space $(\cX, d_\cX)$ is generated by the open sets in $\cX$ and denoted by $\cB(\cX)$.
For any $\gs$-finite and complete measure space $(E,\cE,\mu)$, 
a Banach space $(\cX, \left\|\cdot\right\|_\cX)$, 
and 
integrability exponent $p\in[1,\infty]$,
we define the Lebesgue-Bochner spaces 
\begin{equation*}
	L^p(E; \cX):=\{\varphi:E\to\cX|\;
	\text{$\varphi$ is strongly measurable and $\|\varphi\|_{L^p(E;\cX)}<\infty$}  \},
\end{equation*}
where 
\begin{equation*}
	\|\varphi\|_{L^p(E;\cX)}:=
	\begin{cases}
		\left(\int_{E}\|\varphi(x)\|_\cX^p\mu(dx)\right)^{1/p},\quad &p\in[1,\infty) \\
		\esssup_{x\in E} \|\varphi(x)\|_\cX,\quad &p=\infty.
	\end{cases}
\end{equation*}
In case that $\cX=\bR$, we use the shorthand notation $L^p(E):=L^p(E;\bR)$.
If $E\subset\bR^d$ is a subset of \rev{the} Euclidean space, 
we assume $\cE=\cB(E)$ and $\mu$ is the Lebesgue measure, 
unless stated otherwise.
For any bounded and connected spatial domain $\cD\subset \bR^d$ we denote for $k\in\bN$ and $p\in[1,\infty]$ 
the standard Sobolev space $W^{k,p}(\cD)$ with $k$-order weak derivatives in $L^p(\cD)$.
The Sobolev-Slobodeckji space with fractional order $s\ge0$ is denoted by $W^{s,p}(\cD)$. 
Furthermore, $H^s(\cD):=W^{s,2}(\cD)$ for any $s\ge0$ and we use the identification $H^0(\cD)=L^2(\cD)$. 
Given that $\cD$ is a Lipschitz domain, we define for any $s>1/2$
\begin{equation}
	H_0^s(\cD):={\rm ker}(\gamma_0) = \{\varphi\in H^s(\cD)|\; \gamma_0(\varphi)=0 \;\mbox{on}\; \partial\cD \},
\end{equation}
Here, 
$\gamma_0\in \cL(H^s(\cD),H^{s-1/2}(\partial \cD))$ denotes the trace operator.

Let $\rC(\ol\cD)$ denote the space of all continuous functions $\varphi:\ol\cD\to\bR$.
For any $\ga\in\bN$, $\rm \rC^{\ga}(\ol\cD)$ is the space of all functions $\varphi\in\rC(\ol\cD)$ 
with $\ga$ continuous partial derivatives. 
For non-integer $\ga>0$,
we denote by $\rC^{\ga}(\ol\cD)$ the space of all 
$\varphi\in \rC^{\floor{\ga}}(\ol\cD)$ with $\ga-\floor{\ga}$-Hölder 
continuous $\floor{\ga}$-th partial derivatives.
For any positive, real $\ga>0$ we further denote by $\cC^\ga(\cD)$ the \emph{Hölder-Zygmund space} of smoothness $\ga$. 
We refer to, e.g., \cite[Section 1.2.2]{TriebelTOFS2} for a \rev{definition}.
We denote by $\rS(\bR^d)$ the Schwartz space of all smooth, rapidly decaying functions, 
and with $\rS'(\bR^d)$ its dual, the space of tempered distributions.
Moreover, for any open set $O\subseteq\bR^d$, 
$\mathrm D(O)$ denotes the space of all smooth functions 
$\varphi\in \as{\rC}^\infty(O)$ with compact support in $O$.

\rev{
For the finite element error analysis we introduce a countable set $\mfH\subset(0,\infty)$
of discretization parameters, 
and denote by $h\in \mfH$ a generic 
discretization parameter, such as in the present paper the FE meshwidth of a regular,
simplicial and quasi-uniform partition of the physical domain.
We further assume there exists a decreasing sequence 
$(h_\ell, \ell\in\bN)\subset\mfH$ such that $\lim_{\ell\to\infty} h_\ell = 0$.
}

\subsection{Layout of this paper}
\label{sec:Layout}
In Section~\ref{sec:besov-rv} we introduce the class of random fields 
taking values in the Besov spaces $B^s_{p,p}$ 
which we will use in the sequel to model
the logarithm of the diffusion coefficient function. 
Using multiresolution (``wavelet'') bases in $B^s_{p,p}$, 
in Sections~\ref{subsec:besov-rv}, \ref{subsec:besov-tree-rv}
we construct probability measures on $B^s_{p,p}$ in the spirit of
the Gaussian measure on path space for the Wiener process, in 
L\'{e}vy-Cieselski representation. The multilevel structure of
the construction will be essential in the ensuing MLMC-FE convergence analysis and its algorithmic realization.
In Section~\ref{sec:randompdes} we introduce the linear, elliptic 
divergence-form PDE with Besov-tree coefficients. We recapitulate
(mostly known) results on existence, uniqueness and on strong measurability
of random solutions.
In Section~\ref{sec:approximation} we introduce a conforming Galerkin
Finite Element discretization based on continuous, piecewise linear
approximations in the physical domain.
We account for the error due to finite truncation of the
random tree priors, and provide sharp error bounds for the 
Finite Element discretization errors, under the (generally low) 
Besov regularity of the coefficient samples.
Section~\ref{sec:mlmc} then addresses the MLMC-FE 
error analysis, also for Fr\'echet-differentiable, possibly
nonlinear functionals.
Section~\ref{sec:numerics} then illustrates the theory with several
numerical experiments, where the impact of the parameter choices
in the Besov random tree priors on the overall error convergence 
of the MLMC-FEM algorithms is studied.
Section~\ref{sec:Concl} provides a brief summary of the main results,
and \rev{indicates} several generalizations of these and directions of further
research.
Appendix~\ref{appendix:trees} collects definitions 
and key properties of Galton-Watson trees which are used in the main text. 
Appendix~\ref{appendix:fem} provides a detailed description of the 
FE implementation in the experiments reported in Section~\ref{sec:numerics}.
\section{Random Variables in Besov Spaces}
\label{sec:besov-rv}
\subsection{Wavelet representation of Besov spaces}
\label{subsec:besovspaces}
Let $\bT^d:=[0,1]^d$ denote the $d$-dimensional torus for $d\in\bN$.
We briefly recall the construction of orthonormal wavelet basis on $L^2(\bR^d)$ and $L^2(\bT^d)$ 
and the wavelet representation of the associated Besov spaces. 
For more detailed accounts 
we refer to \cite[Chapter 1]{TriebelFctSpcDom}, \cite[Chapter 1.2]{TriebelTOFS4}, 
and to \cite[Chapter 5]{daubechies1992ten} for orthonormal wavelets in multiresolution analysis (MRA).
\subsubsection{Univariate MRA}
\label{subsec:1dMRA}
Let $\phi$ and $\psi$ be compactly supported scaling and wavelet functions in $\rC^{\ga}(\bR)$, $\ga\ge 1$,
suitable for multi-resolution analysis in $L^2(\bR)$.  
We assume that $\psi$ satisfies \emph{the vanishing moment condition} 
\begin{equation}\label{eq:vanmoments}
	\int_\bR \psi(x) x^m dx = 0, \quad m\in\bN_0,\; m<\ga.
\end{equation}
One example are Daubechies wavelets with $ M := \lfloor \ga \rfloor \in\bN$ vanishing moments 
also known as ${\mathrm D}{\mathrm B}(\lfloor \ga \rfloor)$-wavelets), 
that have support $[-M+1,M]$ and are in $\as{\rC}^1(\bR)$ for $M\ge 5$ 
(see, e.g., \cite[Section 7.1]{daubechies1992ten}).
For any $j\in\bN_0$ and $k\in\bZ$, the MRA is defined by the dilated 
and translated functions
\begin{equation}
	\psi_{j,k,0}(x):=\phi(2^jx-k),
	\quad \text{and} \quad
	\psi_{j,k,1}(x):=\psi(2^jx-k),
	\quad x\in\bR.
\end{equation}
As $\|\phi\|_{L^2(\bR)}=\|\psi\|_{L^2(\bR)}=1$, it follows that 
$((\psi_{0,k,0}), k\in\bZ)\,\cup\,((2^{j/2}\psi_{j,k,1}), (j,k)\in\bN_0\times\bZ)$ 
is an orthonormal basis of $L^2(\bR)$.
\subsubsection{Multivariate MRA}
\label{subsec:dVarMRA}
A corresponding \emph{isotropic}\footnote{Anisotropic tensorizations leading upon truncation 
to so-called ``hyperbolic cross approximations'' may be considered. As such constructions tend to
inject preferred directions along the cartesian axes, we do not consider them here.}
wavelet basis that is orthormal in $L^2(\bR^d)$, $d\geq 2$ 
may be constructed by tensorization of univariate MRAs.
We define index sets $\cL_0:=\{0,1\}^d$ and $\cL_j:=\cL_0\setminus \{(0,\dots,0)\}$ for $j\in\bN$.
We note that $\cL_j$ has cardinality $|\cL_j|=2^d$ if $j=0$, and $|\cL_j|=2^d-1$ otherwise.
For any $l\in\cL_0$, we define furthermore  
\begin{equation}\label{eq:psi_scaled}
	\psi_{j,k,l}(x):=2^{dj/2}\prod_{i=1}^d\psi_{j,k_i,l(i)}(x_i),
	\quad  j\in \bN_0,\; k\in\bZ^d,\; x\in\bR^d,
\end{equation}
to obtain that $((\psi_{j,k,l}),\; j\in \bN_0,\, k\in\bZ^d,\, l\in\cL_j)$ 
is an orthonormal basis of $L^2(\bR^d)$.

Orthonormal bases consisting of locally supported, 
periodic functions on the torus $\bT^d$ can be introduced by
tensorization, as e.g. in \cite[Section 1.3]{TriebelFctSpcDom}.
Given $\phi$ and $\psi$, 
we fix a scaling factor $w\in\bN$ such that 
\begin{equation*}
	\text{supp}(\psi_{w,0,l})\subset 
	\left\{x\in\bR^d\big|\, \|x\|_2<\frac{1}{2}\right\},
	\quad l\in\cL_0.
\end{equation*}
With this choice of $w$, it follows for $j\in\bN_0$ that 
\begin{equation*}
	\text{supp}(\psi_{j+w,0,l})\subset 
	\left\{x\in\bR^d\big|\, \|x\|_2<2^{-j-1}\right\}.
\end{equation*}
Now let 
$K_{j}:=\{k\in\bZ^d|\,0\le k_1,\dots,k_d<2^{j}\}\subset 2^{j}\bT^d$ and note that $|K_{j+w}|=2^{d(j+w)}$.
\as{Define the one-periodic wavelet functions
	\begin{equation*}
		\psi_{j,k,l}^{per}(x):=\sum_{n\in\bZ^d} \psi_{j,k,l}(x-n), 
		\quad j\in\bN_0,\; k\in K_j,\;l\in\cL_0,\;x\in\bR^d,
	\end{equation*}
	and their restrictions to $\bT^d$ by 
	\begin{equation}\label{eq:torusbasis}
		\psi_{j,k}^{l}(x)
		:=
		\psi_{j,k,l}^{per}(x), \quad j\in\bN_0,\; k\in K_j,\;l\in\cL_0,\;x\in\bT^d.
\end{equation}}
We now obtain \as{for the index set $\cI_{w}:=\{j\in\bN_0,\; k\in K_{j+w},\;l\in\cL_j\}$}
\begin{equation}
	\mathbf\Psi_w:=\left((\psi_{j+w,k}^l),\;(j,k,l)\in \cI_w \right)
\end{equation}
is a $L^2(\bT^d)$-orthonormal basis, see \cite[Proposition 1.34]{TriebelFctSpcDom}. 
We next introduce Besov spaces via suitable
wavelet-characterization as developed in \cite[Chapter 1.3]{TriebelFctSpcDom}.
\rev{For this we introduce the set of one-periodic distributions on $\bR^d$ given by 
\begin{equation*}
	\rS'^{per}(\bR^d):=\{\varphi\in \rS'(\bR^d)\big|\; 
	\varphi(\cdot - k) = \varphi\quad\text{for all $k\in \bZ^d$}\},
\end{equation*}
see \cite[Eq. (1.131)]{TriebelFctSpcDom}.}
We distinguish between spaces of one-periodic functions on $\bR^d$ 
and their restrictions to $\bT^d$:
\begin{defi}\label{def:besovspace}
	~
	\begin{enumerate}
		
		\item 
	For any $p\in[1,\infty)$ and $s\in(0,\ga)$ the \emph{Besov space $B_{p,p}^{s,per}(\bR^d)$ of one-periodic functions on $\bR^d$} 
        is given by
		\begin{equation}\label{eq:besovspaceper}
			B_{p,p}^{s,per}(\bR^d)
			:=
			\left\{
			\varphi\in \rev{\rS'^{per}(\bR^d)}\bigg|\; 
			\sum_{(j,k,l)\in\cI_w} 2^{jp(s+\frac{d+w}{2}-\frac{d}{p})} |(\varphi,\psi_{j+w,k,l}^{per})_{L^2(\bT^d)}|^p
			<\infty
			\right\}.
		\end{equation}
	In case that $p=\infty$, one has
	\begin{equation}\label{eq:besovspaceinfper}
		B_{\infty,\infty}^{s,per}(\bR^d)
		:=
		\left\{
		\varphi\in \rev{\rS'^{per}(\bR^d)}\bigg|\; 
		\sup_{(j,k,l)\in\cI_w} 2^{j(s+\frac{d+w}{2})} |(\varphi,\psi_{j+w,k,l}^{per})_{L^2(\bT^d)}|
		<\infty
		\right\}.
	\end{equation}
	
		\item For any $p\in[1,\infty)$ and $s\in(0,\ga)$ 
                the \emph{Besov space $B_{p,p}^s(\bT^d)$ on $\bT^d$} is given by
		\begin{equation}\label{eq:besovspace}
			B_{p,p}^s(\bT^d)
			:=
			\left\{
			\varphi\in \mathrm D'(\bT^d) \bigg|\; 
			\sum_{(j,k,l)\in\cI_w} 2^{jp(s+\frac{d+w}{2}-\frac{d}{p})} |(\varphi ,\psi_{j+w,k}^l)_{L^2(\bT^d)}|^p
			<\infty
			\right\}.
		\end{equation}
		In case that $p=\infty$, we set 
		\begin{equation}\label{eq:besovspaceinf}
			B_{\infty,\infty}^s(\bT^d)
			:=
			\left\{
			\varphi\in \mathrm D'(\bT^d) \bigg|\; 
			\sup_{(j,k,l)\in\cI_w} 2^{j(s+\frac{d+w}{2})} |(\varphi ,\psi_{j+w,k}^l)_{L^2(\bT^d)}|
			<\infty
			\right\}.
		\end{equation}
	\end{enumerate}
\end{defi}
\begin{rem}\label{rem:extension}
Definition~\ref{def:besovspace} may be generalized to define the spaces 
$B_{p,q}^{s,per}(\bR^d)$ and $B_{p,q}^s(\bT^d)$ with $p,q\in[1,\infty]$ and $p\neq q$, 
see \cite[Chapter 1.3]{TriebelFctSpcDom}.
The random fields introduced in Subsections~\ref{subsec:besov-rv} and~\ref{subsec:besov-tree-rv} 
are naturally $B_{p,p}^s(\bT^d)$-valued by construction, thus we only treat the case $p=q$ for the sake of brevity.  	
By \cite[Theorem 1.29]{TriebelFctSpcDom}
\rev{there} exists a \emph{prolongation isomorphism}
\begin{equation}\label{eq:extension-map}
	\mathrm {prl}^{per}:B_{p,p}^s(\bT^d)\to B_{p,p}^{s,per}(\bR^d),
\end{equation}
that extends $\varphi\in B_{p,p}^s(\bT^d)$ to its (unique) one-periodic counterpart 
in $\mathrm {prl}^{per}(\varphi)\in B_{p,p}^{s,per}(\bR^d)$.
This in turn allows to identify any $\varphi\in B_{p,p}^s(\bT^d)$ 
as the restriction of a periodic function 
$\mathrm {prl}^{per}(\varphi)\in B_{p,p}^{s,per}(\bR^d)$ to $\bT^d$.
We use $\mathrm {prl}^{per}$ to define (non-periodic) 
Besov space-valued random variables on subsets $\cD\subset \bT^d$ 
by restriction in Subsection~\ref{subsec:randomtree-diffusion}.
	
Definition~\ref{def:besovspace} is based on an equivalent characterization of the spaces 
$B_{p,p}^{s,per}(\bR^d)$ and $B_{p,p}^s(\bT^d)$. 
They are often (equivalently) defined using a dyadic partition of unity 
(see e.g. \cite[Definitions 1.22 and 1.27]{TriebelFctSpcDom}): 
	Using the latter definition for 
	$B_{p,p}^{s,per}(\bR^d)$,
	\cite[Theorem 1.36(i)]{TriebelFctSpcDom} shows that 
        the spaces \eqref{eq:besovspaceper} resp. \eqref{eq:besovspaceinfper} 
        are isometrically isomorphic to $B_{p,p}^{s,per}(\bR^d)$. 
	As a consequence of the prolongation map $\mathrm{prl}^{per}$ in~\eqref{eq:extension-map}, 
        the same holds true for the spaces $B_{p,p}^s(\bT^d)$, 
        see \cite[Theorem 1.37(i)]{TriebelFctSpcDom}.
\end{rem}
\subsubsection{Besov Spaces and MRAs}
\label{sec:BesMRAs}
We define the subspace 
$V_{w+1}:=\text{span}\{\psi_{w,k}^l|\; k\in \rev{K_w},\ l\in\cL_0\}\subset L^2(\bT^d)$ 
and observe that $\dim(V_{w+1})=2^{d(w+1)}$.
By the multiresolution analysis for one-periodic, univariate functions 
in \cite[Chapter 9.3]{daubechies1992ten}, 
it follows that 
$( (\psi_{j,k}^l) ,\;j\le w,\; k\in K_j,\ l\in\cL_j)$ is another orthonormal basis of $V_{w+1}$.
Hence, 
we may replace the first $2^{d(w+1)}$ basis functions in~\eqref{eq:torusbasis} 
to obtain the (computationally more convenient) 
$L^2(\bT^d)$-orthonormal basis
\begin{equation}\label{eq:torusbasis2}
	\mathbf\Psi:=\left((\psi_{j,k}^l),\;(j,k,l)\in \cI_{\mathbf \Psi} \right),
	\quad
	\cI_{\mathbf \Psi}:=\{j\in\bN_0,\; k\in K_j,\;l\in\cL_j\}.
\end{equation}
Based on~\eqref{eq:torusbasis2}, we define for $s>0$, $p\in[1,\infty)$, 
and sufficiently regular $\varphi\in L^2(\bT^d)$
the \emph{Besov norms} 
\begin{equation}\label{eq:besovnorm}
	\|\varphi\|_{B_{p,p}^s(\bT^d)}:=
	\left(
	\sum_{(j,k,l)\in\cI_{\mathbf\Psi}} 2^{jp(s+\frac{d}{2}-\frac{d}{p})} |(\varphi ,\psi_{j,k}^l)_{L^2(\bT^d)}|^p
	\right)^{1/p},
\end{equation}
and
\begin{equation}\label{eq:besovnorminf}
	\|\varphi\|_{B_{\infty,\infty}^s(\bT^d)}:=
	\sup_{(j,k,l)\in\cI_{\mathbf\Psi}} 2^{j(s+\frac{d}{2})} |(\varphi ,\psi_{j,k}^l)_{L^2(\bT^d)}|
	<\infty.
\end{equation}
By Definition~\ref{def:besovspace}, 
it follows that 
$\varphi\in B_{p,p}^s(\bT^d)$ if and only if $\|\varphi\|_{B_{p,p}^s(\bT^d)}<\infty$.
\subsubsection{Notation}
\label{subsec:BesovNotat}
We fix some notation for Besov, H\"older and Zygmund spaces to be used in the
remainder of this paper.
As the (periodic) domain $\bT^d$ does not vary in the subsequent analysis, 
we use the abbreviations 
$B_p^s:=B_{p,p}^s(\bT^d)$, $\rC^\ga:=\rC^\ga(\bT^d)$ and $\cC^\ga:=\cC^\ga(\bT^d)$
for convenience in the following.
Furthermore, 
we will assume that the basis functions in $\mathbf\Psi_w, \mathbf\Psi\subset \rC^\ga$,
are sufficiently smooth with Hölder index $\ga>s$ for given $s>0$, 
and therefore omit the restriction $s\in(0,\ga)$ in the following.
In this case, it holds that $\mathbf \Psi_w$ (and thus $\mathbf\Psi$) is a basis of $B_p^s$ for $p<\infty$, 
see \cite[Theorem 1.37]{TriebelFctSpcDom}.

We recall that for any $s>0$ there holds $\cC^s=B^s_\infty$ (see \cite[Remark 1.28]{TriebelFctSpcDom}), as well as 
$\rC^s=\cC^s$ for $s\in(0,\infty)\setminus\bN$, and $\rC^s\subsetneq\cC^s$ for $s\in\bN$ (see \cite[Section 1.2.2]{TriebelTOFS2}).
By~\eqref{eq:besovnorm} and~\eqref{eq:besovnorminf} we further obtain the continuous embeddings
\begin{equation}\label{eq:besovembedding}
	\begin{split}
		B_p^s\hookrightarrow B_q^t
		&\quad\text{if $1\le p\le q<\infty$ and $s-\frac{d}{p}\ge t-\frac{d}{q}$}, \\
		B_p^s\hookrightarrow B^t_\infty=\cC^t
		&\quad \text{for $t\in\left(0,s-\frac{d}{p}\right]$,}
	\end{split}	
\end{equation}
with the embedding constants in~\eqref{eq:besovembedding} bounded by one (cf. \cite[Chapter 2.1]{TriebelTOFS4}).
\subsection{Besov priors}
\label{subsec:besov-rv}
To introduce Besov space-valued random variables as in \cite{LassasBesov09}, we consider a complete probability space $(\gO,\cA, P)$. 
Following the constructions in \cite{DHS12,AgaDasHel2021,KLSS21},
based on the representation in~\eqref{eq:besovspace}, 
we now define $B_p^s$-valued random variables by replacing the 
$L^2(\bT^d)$-orthogonal projection coefficients
$(\varphi ,\psi_{j,k}^l)_{L^2(\bT^d)}$ with suitable random variables. 
More precisely, 
consider for any $p\in[1,\infty)$ an independent and identically distributed (i.i.d.) sequence 
$X=((X_{j,k}^l), (j,k,l)\in\mathbf{\cI_\Psi})$ of \emph{$p$-exponential} random variables. 
That is, each $X_{j,k}^l:\gO\to\bR$ is $\cA/\cB(\bR)$-measurable with density 
\begin{equation}\label{eq:p-exponential}
	\phi_p(x):=\frac{1}{c_p}\exp\left(-\frac{|x|^p}{\gk}\right), \quad x\in\bR,
	\qquad c_p := \int_\bR \exp\left(-\frac{|x|^p}{\gk}\right) dx,
\end{equation}
where $\gk>0$ is a fixed scaling parameter. 
We recover the normal distribution with variance $\frac{\gk}{2}$ if $p=2$, 
and the Laplace distribution with scaling $\gk$ for $p=1$.
\begin{defi}\cite[Definition 9]{LassasBesov09}\label{def:besovprior}
Let $\mathbf\Psi$ be the $L^2(\bT^d)$-orthogonal wavelet basis as in~\eqref{eq:torusbasis2}, 
let $s>0$, $p\in[1,\infty)$ and let $X=((X_{j,k}^l), (j,k,l)\in\mathbf{\cI_\Psi})$ 
be an i.i.d. sequence of $p$-exponential random variables with density 
$\phi_p$ as in~\eqref{eq:p-exponential}.
Let the random field $b:\gO\to L^2(\bT^d)$ be given by
	\begin{equation}\label{eq:b}
		b(\go):= \sum_{(j,k,l)\in\mathbf{\cI_\Psi}} 
		\eta_j X_{j,k}^l(\go)\psi_{j,k}^l, \quad \go\in\gO,
		\quad\text{where}\quad
		\eta_j := 2^{-j(s+\frac{d}{2}-\frac{d}{p})}, 
		\quad j\in\bN_0.
	\end{equation}
	We call $b$ \rev{a \emph{$B_p^s$-random variable}.} 
\end{defi}

The random variables $b$ from Definition~\ref{def:besovprior} 
are also referred to as \emph{Besov priors} in the literature on inverse problems.
The following regularity results are well-known:
\begin{prop}~\label{prop:besovreg}
	\begin{enumerate}
	\item \rev{(\cite[Lemma 10]{LassasBesov09}, \cite[Proposition 1]{DHS12} ) }
	Let $b$ be a $B_p^s$-random variable for $s>0$ and $p\in[1,\infty)$. 
         Then, the following conditions are equivalent:
		\begin{enumerate}[(i)]
			\item $\|b\|_{B_p^t}<\infty$ holds $P$-a.s.;
			\item $\bE\left(\exp\left(\eps \|b\|_{B_p^t}^p\right)\right)<\infty$,
			\quad$\eps\in(0,\frac{1}{\rev{4} \gk})$;
			\item $t<s-\frac{d}{p}$.
		\end{enumerate}
		\item \cite[Theorem 2.1]{DHS12} 
                If, in addition, $\mathbf\Psi$ forms a basis of $B_p^t$ for a $t<s-\frac{d}{p}$, $t\notin\bN$, 
                then it holds  
		\begin{equation*}
			\bE\left(\exp\left(\eps \|b\|_{\rC^t}\right)\right)<\infty,\quad \eps\in(0,\ol \eps),
		\end{equation*}
		where $\ol \eps>0$ is a constant depending on $p, d, s$ and $t$.
	\end{enumerate}
\end{prop}

\begin{rem}\label{rem:Stronger}

	\rev{Note that $B_p^s$-random variables as defined above only take values in $B_p^t$ for $t<s-\frac{d}{p}$ based on the previous proposition. Nevertheless, we use the notion \emph{$B_p^s$-random variables} in the following, for a clearer emphasize on the  dependence of $\eta_j$ in~\eqref{eq:b} on $s$.}
	
	We derive a considerably stronger version of \cite[Theorem 2.1]{DHS12} in  Theorem~\ref{thm:randomtreereg} below, that implies in particular 
	\begin{equation*}
		\bE\left(\exp\left(\eps \|b\|^p_{\rC^t}\right)\right)<\infty,\quad \eps\in(0,\ol \eps),
	\end{equation*}
	for any $p\ge1$ and some $\ol \eps>0$. In the Gaussian case with $p=2$, 
        this estimate would be a consequence of Fernique's theorem, however, 
        we are not aware of a similar result for arbitrary $p\ge1$ in the literature.
		
	\item 
		We recall from \cite[Theorem 1.37]{TriebelFctSpcDom}   
     	that $\mathbf\Psi$ forms an \emph{unconditional basis} of $B_p^t$ (since $p<\infty$), 
        if the scaling and wavelet functions $\phi$ and $\psi$ 
satisfy $\phi, \psi\in \rC^\ga(\bR^d)$ for $\ga>t>0$ and the vanishing moment condition~\eqref{eq:vanmoments}.
\end{rem}

\subsection{Besov random tree priors}
\label{subsec:besov-tree-rv}
Taking the cue from \cite{KLSS21},
we introduce \emph{Besov random tree priors} 
in this subsection and derive several regularity results 
for this $B_p^s$-valued random variable.
We investigate all results for periodic functions defined on the torus $\bT^d$ in this subsection. 
For the elliptic problem in Section~\ref{sec:randompdes}, 
we will later introduce the corresponding $B_p^s(\cD)$-valued random variables on physical domains 
$\cD\subset\bR^d$ with $\cD\subseteq\bT^d$ by their restrictions from $\bT^d$ (cf. Definition~\ref{def:restrictedbesovprior}). 
The random tree structure in our prior construction 
is based on certain set-valued random variables, so-called \emph{Galton-Watson (GW) trees}. 
For the readers' convenience, 
definitions of discrete trees, GW trees, along with some other useful results, 
are listed in Appendix~\ref{appendix:trees}. 

\begin{defi}\cite[Definition 3]{KLSS21}\label{def:randomtree-prior}
	Let $\mathbf\Psi$, $s>0$, $p\in[1,\infty)$, $X=((X_{j,k}^l), (j,k,l)\in\mathbf{\cI_\Psi})$ 
        and $(\eta_j,j\in\bN_0)$ be as in Definition~\ref{def:besovprior}.
	Let $\mfT$ denote the set of all trees with no infinite node (cf. Definition~\ref{def:trees}) 
        and let $T:\gO\to \mfT$ be a GW tree (cf. Definition~\ref{def:GW-tree}) 
        with offspring distribution $\cP=\textrm{Bin}(2^d, \gb)$ for $\gb\in[0,1]$, and independent of $X$.
	Furthermore, let $\mf I_T$ be the set of wavelet indices associated to $T$ from~\eqref{eq:node-to-index-set}.
	We define the \emph{random tree index set} 
	$\cI_T(\go):=\{(j,k,l)|\; (j,k)\in \mf I_T(\go),\; l\in\cL_j\}$ and
	\begin{equation}\label{eq:randomtreeprior}
		b_T(\go):= \sum_{(j,k,l)\in \cI_T(\go)} \eta_j X_{j,k}^l(\go)\psi_{j,k}^l, \quad \go\in\gO.
	\end{equation}
	We refer to $b_T$ as a \emph{$B_p^s$-random variable with wavelet density $\gb$}.
\end{defi}
\rev{We have depicted a sample of a binomial GW tree and the associated set of wavelet indices $\mf I_T$ for a series expansion in one physical dimension ($d=1$) in Figure~\ref{fig:GWtree}. Recall that for $d=1$ there holds $\cL_j=\{1\}$ for $j\ge 1$.}

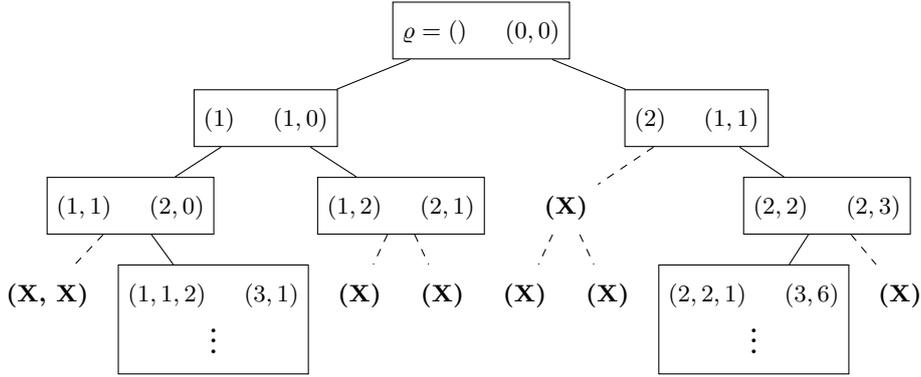
\begin{figure}
	\input{GWtree}
	\caption{\rev{Sample of a binomial GW tree $T$ with offspring distribution $\cP(2, \gb)$ with $\gb=0.5$. Each box corresponds to a node $\mf n$ of the GW tree, displayed are all nodes with length $|\mf n|\le 3$. 
	The left entry in each box is the node $\mf n\in T$, given by a finite sequence of integers. 
	The right entry is the node-to-wavelet coefficient map $\mf n\mapsto \left(|\mf n|, \mf I^2_{d,|\mf n|}\circ\mf I^1_{d,|\mf n|}(\mf n)\right)$ evaluated at $\mf n$, that determines the associated wavelet indices $(j,k)$.
		Starting from the "root node" $\varrho=()$, the two children of each node are eliminated with probability $1-\gb$, and independent of each other. Once a node is eliminated (signified by an "(X)"), all of their offspring nodes are eliminated as well.
		The remaining "surviving" nodes determine the associated random set of active wavelet indices via 
		$\mf I_T:=\left\{\left(|\mf n|, \mf I^2_{d,|\mf n|}\circ\mf I^1_{d,|\mf n|}(\mf n)\right)\big|\; \mf n\in T\right\} = \left\{ (0,0), (1,0), (1,1), (2,0), (2,1), (2,3), (3,1), (3,6)\right\}$. 	
		Note that each $\mf n\in T$ is connected to the root node $\varrho$ through their ancestors, as indicated by the solid lines.
	}}
	\label{fig:GWtree}	
\end{figure}

\begin{rem}\label{rem:Besov-random-tree}
		Definition~\ref{def:randomtree-prior} actually slightly deviates from \cite[Definition 3]{KLSS21}. 
		By definition of $\cI_T(\go)$, 
		we include the constant function $\psi_{0,0}^{(0,\dots,0)}\equiv 1\in L^2(\bT^d)$ 
                in the series expansion~\eqref{eq:randomtreeprior}.
		Of course, adding the random constant $X_{0,0}^{(0,\dots,0)}$ 
                does not affect the spatial regularity or integrability of $b_T$.
		However, in our definition, series~\eqref{eq:randomtreeprior} 
		has a natural interpretation as orthogonal expansion of a random function 
                with respect to the (deterministic, fixed) basis $\mathbf \Psi$.
	The tree structure in the "active" (i.e., with index in $\cI_T$) coefficients in the wavelet representation of $b_T$
	gives rise to random fractals on $\bT^d$, that occur whenever 
	the tree $T$ in Definition~\ref{def:randomtree-prior} does not terminate after a finite number of nodes. 
	It follows by Lemma~\ref{lem:extinction}, that the latter event occurs with positive probability if $\gb\in(2^{-d},1]$. 
        In this case the Hausdorff dimension of the fractals is $d+\log_2(\gb)\in(0,d]$, see \cite[Section 3]{KLSS21} for further details.
\end{rem}
Examples of realizations of a $B_p^s$-random variable on $\bT^2$ with varying wavelet density 
$\gb$ are shown in Figure~\ref{fig:samples}.
\begin{figure}[t]
	\centering
	\subfigure{\includegraphics[scale = 0.31]{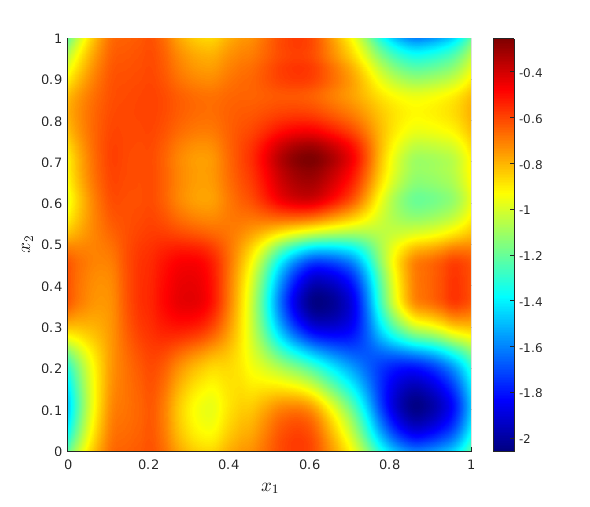}}
	\subfigure{\includegraphics[scale = 0.31]{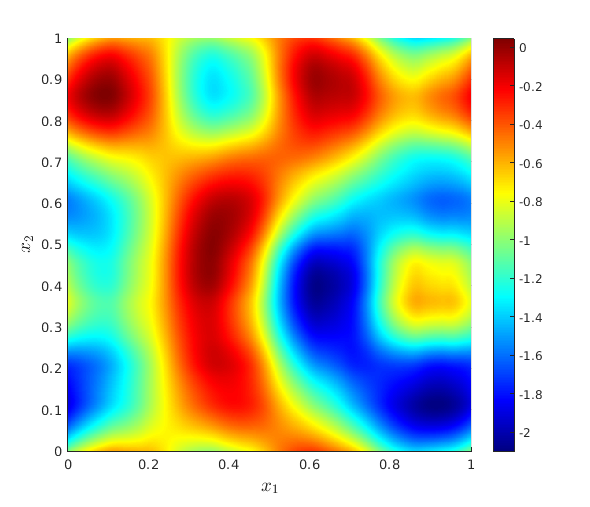}}
	\subfigure{\includegraphics[scale = 0.31]{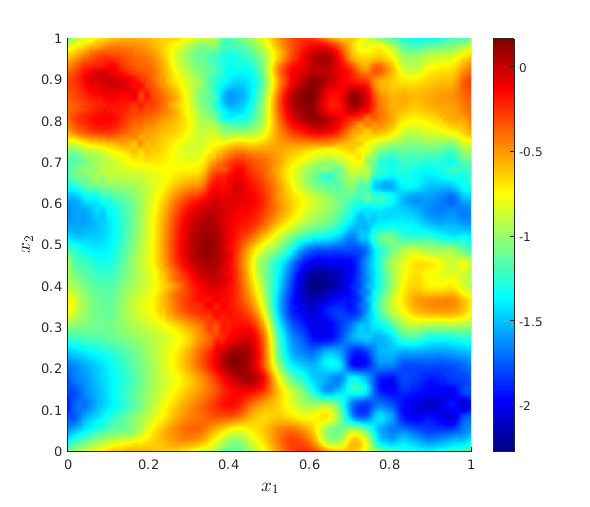}}
	\subfigure{\includegraphics[scale = 0.31]{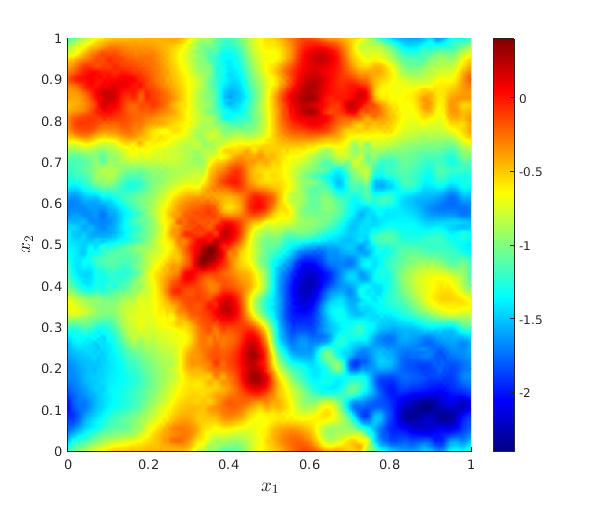}}
	\subfigure{\includegraphics[scale = 0.31]{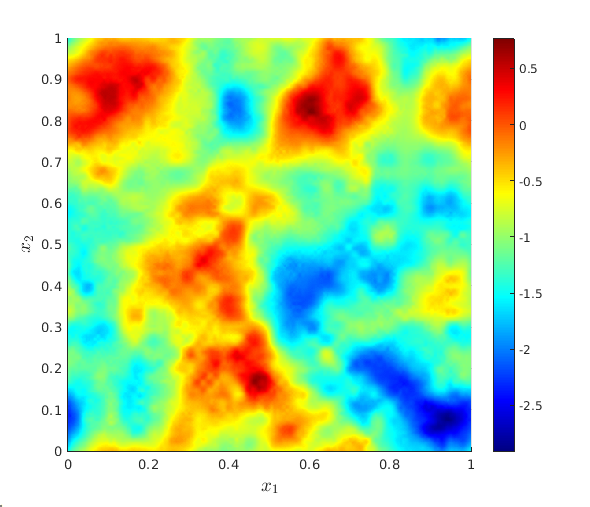}}
	\subfigure{\includegraphics[scale = 0.31]{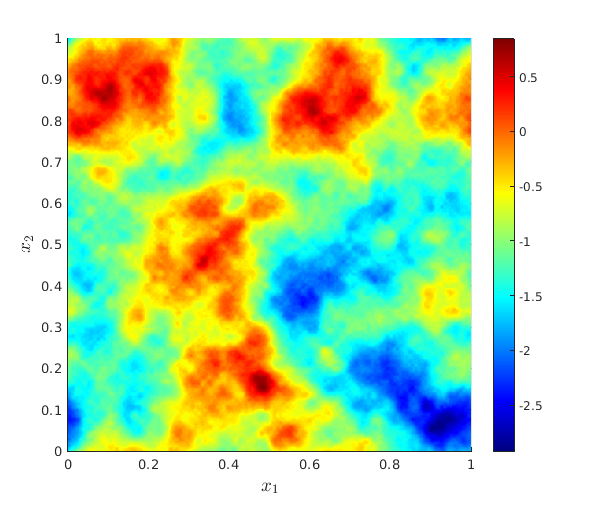}}
	\caption{
		Samples of a $B_p^s$-valued random variable on $\bT^2=[0,1]^2$ with 
		$s=p=2$ and wavelet density $\gb\in\{\frac{1}{4}, \frac{1}{3}, \frac{1}{2}\}$ 
		(top row, from left to right)
		and 
		$\gb\in\{\frac{2}{3}, \frac{3}{4}, 1\}$
		(bottom row, from left to right).
		All samples are based on \rev{DB(5)-wavelets and the same array of random numbers, that} have been sampled with a spatial resolution of $2^9\times 2^9$ equidistant grid points, 
		and the expansion in ~\eqref{eq:randomtreeprior} was truncated at 
		$N=9$ levels of dyadic subdivision (cf. Subsection~\ref{subsec:truncation}).
		By fixing the array of random numbers, the spatial grid and $N$, 
		the depicted "evolution" in the panels highlights the effect of an increasing wavelet density $\gb$. \rev{Note that the case $p=2$ and $\gb=1$ in the bottom right corresponds to a Gaussian prior.} 
	}
	\label{fig:samples}
\end{figure}
\rev{
To treat elliptic inverse problems with $b_T$ as log-diffusion coefficient, 
we detail the corresponding probability space of parameters.
Let $\bQ_0$ denote the univariate, $p$-exponential measure on $(\bR, \cB(\bR))$ 
of the random variables $X_{j,k}^l$ with Lebesgue density as in~\eqref{eq:p-exponential}.
The product-probability space of the $p$-exponentials $X$ is given by $(\gO_p, \cA_p, \bQ_p)$, 
where 
\begin{equation*}
	\gO_p:=\bR^\bN,\quad  
	\cA_p:=\bigotimes_{n\in\bN} \cB(\bR),\quad \text{and} \quad
	\bQ_p:=\bigotimes_{n\in\bN} \bQ_0.
\end{equation*}
Now let $s>0$ and $p\in[1,\infty)$ be fixed such that $s>\frac{d}{p}$. 
We define the weighted $\ell^p$-spaces 
\begin{equation}\label{eq:lp-weighted}
	\ell_s^p:=\left\{ x=\left(x_{j,k}^l, (j,k,l)\in\mathbf{\cI_\Psi}\right)\in\bR^\bN|\;
	\|x\|_{s,p}<\infty\right\},
\end{equation}
where 
\begin{equation}
	\|x\|_{s,p}
	:= \left(\sum_{(j,k,l)\in\mathbf{\cI_\Psi}} 2^{-jps}|x_{j,k}^l|^p \right)^{1/p}.
\end{equation}
Then $(\ell_s^p, \left\|\cdot\right\|_{s,p})$ is a separable Banach space, for $1\le p<\infty$. 
Moreover, we observe that for $X\sim \bQ_p$ it holds
\begin{align*}
	\bE(\|X\|_{s,p}^p)
	\le
	\sum_{(j,k,l)\in\mathbf{\cI_\Psi}} 2^{-jps}\bE(|X_{j,k}^l|^p) 
	\le 
	C \sum_{j=0}^\infty 2^{-jps} 2^{dj} (2^d-1) 
	\le 
	C \sum_{j=0}^\infty 2^{-jp(s-\frac{d}{p})}
	<\infty, 
\end{align*}
since $s>\frac{d}{p}$, thus $\bQ_p$ is concentrated on $\ell_s^p$. 
\rev{We therefore} regard $(\ell_s^p, \cB(\ell_s^p), \bQ_p)$ as the probability space of random coefficient sequences $X$ in the expansions~\eqref{eq:b} and~\eqref{eq:randomtreeprior}. 
The set-valued random variable $T$ is a GW tree, and hence takes values in the Polish space $(\mfT, \gd_{\mfT})$ of all trees with no infinite node. The metric $\gd_\mfT$ and the associated Borel $\gs$-algebra $\cB(\mfT)$ with respect to $\mfT$ are stated explicitly in Definition~\ref{defi:tree-algebra} in the Appendix. 
The image measure $\bQ_T$ of the GW tree $T$ on $(\mfT, \cB(\mfT))$ then solely depends on the parameters $\gb$ and $d$ of the offspring distribution $\cP=\textrm{Bin}(2^d, \gb)$, and is given in Equation~\eqref{eq:tree-measure} in the Appendix. Hence, the parameter probability space of GW trees is given by $(\mfT, \cB(\mfT), \bQ_T)$.
To combine the random coefficients $X$ with the GW tree $T$, we define the cartesian product 
$\gO:=\ell_s^p\times \mfT$ and equip $\gO$ with the metric 
\begin{equation*}
	d_\gO(\as{(x_1, {\bf t_1}),(x_2, {\bf t_2})}) := \|x_1-x_2\|_{s,p}+\gd_\mfT({\bf t_1}, {\bf t_2}). 
\end{equation*}
\begin{prop}
	The space $(\gO, d_\gO)$ is Polish with Borel $\gs$-algebra given by $\cB(\gO)=\cB(\ell_s^p\times \mfT)=\cB(\ell_s^p)\otimes\cB(\mfT)$.
\end{prop}
\begin{proof}
	By Lemma~\cite[Lemma 2.1]{abraham2015GW-Trees} the metric space $(\mfT, \gd_\mfT)$ with $\gd_\mfT$ given in~\eqref{eq:tree-metric} is complete and separable. Therefore, separability and completeness of $(\gO, d_\gO)$ follows by \cite[Corollary 3.39]{AB06}. 
	Moreover, $\cB(\gO)=\cB(\ell_s^p\times\as{\mfT})=\cB(\ell_s^p)\otimes\cB(\mfT)$ holds by \cite[Theorem 4.44]{AB06}	
\end{proof}
We are now ready to define the probability space associated to the $\ell_s^p\times \mfT$-valued random variable $(X, T)$: Let $(\gO, \cA, \bP)$ be the product probability given by 
\begin{equation}
	\gO:=\ell_s^p\times \mfT,\quad  
	\cA:=\cB(\ell_s^p)\otimes\cB(\mfT),\quad \text{and} \quad
	\bP:=\bQ_p \otimes \bQ_T.
\end{equation}
We remark that the product structure of the measure $\bP=\bQ_p \otimes \bQ_T$ 
is tantamount to stochastic independence of $X$ and $T$.
It still remains to identify a realization of the random variable $(X,T)$ 
with the corresponding random tree prior $b_T$. To this end, we consider the canonical mapping
\begin{equation}
	b_T:\gO \to L^2(\bT^d),\quad 
	\go\mapsto \sum_{(j,k,l)\in \cI_T(\go)} \eta_j X_{j,k}^l(\go)\psi_{j,k}^l.
\end{equation}
The map $b_T:\gO \to L^2(\bT^d)$ is indeed well-defined since $\|b_T\|_{L^2(\bT^d)}<\infty$ holds due to $s>\frac{d}{p}$. Moreover, $b_T$ is $\cA/\cB(L^2(\bT^d))$-measurable, as we show in Proposition~\ref{prob:measurability} below. 
As $b_T$ is a $L^2(\bT^d)$-valued random variable, we may define the pushforward probability measure of $b_T$ via
\begin{equation}\label{eq:beson-measure}
	b_T\#\bP(B):=\bP(b_T^{-1}(B)),\quad B\in \cB(L^2(\bT^d)).
\end{equation}
Thus, the associated probability space of $B_p^s$-random variables $b_T$ with wavelet density $\gb$ \rev{is}
$$(L^2(\bT^d),\,\cB(L^2(\bT^d)),\,b_T\#\bP).$$
\begin{rem}\label{rem:concentration}
	We know from Proposition~\ref{prop:besovreg} that $b_T\#\bP$ is concentrated on $B_p^t$ for any $t\in(0, s-\frac{d}{p})$. A more refined result that concentrates $b_T\#\bP$ on Besov spaces $B_q^t$ for $q\ge 1$ with smoothness index $t=t(s,d,p,\gb,q)$ is given in Theorem~\ref{thm:randomtreereg} below.
\end{rem}
}
\rev{
We recall at this point that we have assumed Hölder-regularity 
$\mathbf\Psi\subset \rC^{\ga}$ for some $\ga\ge1$ 
in Subsection~\ref{subsec:besovspaces}. 
For the remainder of this article, we will from now on implicitly assume that the parameter $s>0$ of a $B_p^s$-random variable satisfies $\ga\ge s>0$ for the sake of presentation.
}

\begin{prop}\label{prob:measurability}
Let $s>\frac{d}{p}$, $\gb\in[0,1]$, and 
let $b_T$ be a $B_p^s$-random variable with wavelet density $\gb$.
Then $b_T:\gO\to \rC(\bT^d)$ and $b_T$ is (strongly) $\cA/\cB(\rC(\bT^d))$-measurable.
\end{prop}
\begin{proof}
Fix a $t\in(0,s-\frac{d}{p})$ such that $t\notin\bN$. The second part of Proposition~\ref{prop:besovreg} then shows that $b_T\in {\rm C}^t $ holds $P$-a.s., and thus $b_T:\gO\mapsto \rC(\bT^d)$ follows, after possibly modifying $b_T$ on a $P$-nullset.
As in Appendix~\ref{appendix:trees}, we denote by $\cU$ the set of \rev{all finite sequences} in $\bN$ and 
introduce the subset $\cU_{\rm Bin}\subset\cU$ with entries in $\{1,\dots, 2^d\}$ 
as
	\begin{equation*}
		\cU_{\rm Bin}:=\{\mf n\in\cU|\, {\mf n}_i\in\{1,\dots, 2^d\} \text{ for $i\in\{1,\dots, |\mf n|\}$ }\}.
	\end{equation*}
Note that $T(\go)\subset \cU_{\rm Bin}$ holds $P$-a.s., since $\cP=\textrm{Bin}(2^d, \gb)$. 
Now let $\mf I_{d,j}$ be as in~\eqref{eq:node-to-index-map} and 
recall from Appendix~\ref{appendix:wavelet-trees} that 
$(j,k)\in\mf I_T(\go)$ if and only if there is $\mf n\in T(\go)$ 
such that $(j,k)=(|\mf n|, \mf I_{d,|\mf n|}(|\mf n|))$. 
Hence, we may rewrite the series expansion~\eqref{eq:randomtreeprior} as
	\begin{align*}
		b_T(\go)
		=
		\sum_{(j,k)\in\mf I_T(\go)} \eta_j \sum_{l\in\cL_j} X_{j,k}^l(\go)\psi_{j,k}^l
		=
		\sum_{\mf n\in\cU_{\rm Bin}} \indi_{\{\mf n\in T(\go)\}} \eta_{|\mf n|} 
		\sum_{l\in\cL_{|\mf n|}} X_{|\mf n|,\mf I_{d,|\mf n|}(|\mf n|)}^l(\go)\psi_{|\mf n|,\mf I_{d,|\mf n|}(|\mf n|)}^l.
	\end{align*}
As $T:\gO\to\mfT$ is $\cA/\cB(\mfT)$-measurable, 
it holds that $\indi_{\{\mf n\in T(\cdot)\}}:\gO\to\{0,1\}$ is measurable for any fixed $\mf n\in\cU_{\rm Bin}$. 
Also, the $X_{j,k}^l$ are real-valued random variables and 
$\psi_{j,k}^l\in \rC(\bT^d)$ by assumption.
Thus, $b_T:\gO\to \rC(\bT^d)$ is measurable, 
and strongly measurable as $\rC(\bT^d)$ is separable.
\end{proof}

\medskip

More insight in the pathwise regularity of Besov random tree priors, 
in particular with regard to their Hölder regularity, 
is obtained by the following result.

\begin{thm}\label{thm:randomtreereg}
	Let $b_T$ be a $B_p^s$-random variable 
        with wavelet density $\gb=2^{\gg-d}$ as in Definition~\ref{def:randomtree-prior} with $\gg\in(-\infty, d]$.
	\begin{enumerate}[1.)]
		\item It holds that $b_T\in L^q(\gO; {B^t_q})$, and hence $b_T\in B_q^t$ $P$-a.s., 
                      for all $t>0$ and $q\ge1$ such that $t < s+ \frac{d-\gg}{q} - \frac{d}{p}$.
		\item Let $s-\frac{d}{p}>0$ and $t\in(0,s-\frac{d}{p})$. Then there is a $\eps_p>0$ such that 
		\begin{equation*}
			\bE\left(\exp\left(\eps \|b\|_{\cC^t}^p\right)\right) < \infty,\quad \eps\in(0,\eps_p),
		\end{equation*}
		In particular, it holds $b_T\in L^q(\gO; \cC^t)$ for any $q\ge1$.
		\item Let $q\ge 1$ and $s-\frac{d}{p}-\frac{\min(\gg,0)}{q}>0$. 
                For any $t\in(0,s-\frac{d}{p}-\frac{\min(\gg,0)}{q})$ it holds $b_T\in L^q(\gO; \cC^t)$. 
	\end{enumerate}   
\end{thm}
\begin{proof}
	1.) For given $q\ge1$ and $t>0$ it holds by~\eqref{eq:besovnorm} that   
	\begin{align*}
		\|b_T\|^q_{B^t_q}
		=
		\sum_{(j,k,l)\in \cI_T(\go)} 
		2^{jq(t+\frac{d}{2}-\frac{d}{q})} \eta_j^q |X_{j,k}^l(\go)|^q 
		= 
		\sum_{(j,k,l)\in \cI_T(\go)} 
		2^{jq(t - \frac{d}{q} -s + \frac{d}{p})} 
		|X_{j,k}^l(\go)|^q.
	\end{align*}
	For any given $j\in\bN$, by Definition~\ref{def:randomtree-prior}, 
        the number of nodes $v(j)$ on scale $j$ in the random tree $T$ is binomial distributed (conditional on $v(j-1)$) as 
	\begin{equation*}
		v(j):=\#\{k\in K_j|(j,k,l)\in \cI_T\}
		\sim\text{Bin}(2^dv(j-1), 2^{\gg-d}),
	\end{equation*}
	with initial value $v(0)=1$. 
	Now let $(X_{j,m}, j\in\bN_0, m\in\bN)$ be an i.i.d. sequence of $p$-exponential random variables, independent of $v(j)$ for any $j\in\bN$, \rev{and recall that $l\in\cL_j$ with $|\cL_0|=2^d$ and $|\cL_j|=2^d-1$ for $\in\bN$.}
	We obtain by Fubini's theorem, Wald's identity, $\bE(v(j))=(2^d\gb)^j=2^{j\gg}$, and since  $\bE\left(|X_{j,m}|^q\right)<\infty$ for any $q>0$ that
	\begin{align*}
		\bE\left(
		\|b_T\|^q_{B^t_q}
		\right)
		&=
		\bE
		\left(
		\sum_{j=0}^\infty 
		2^{jq(t - \frac{d}{q} -s + \frac{d}{p})}\sum_{m=1}^{(2^d-1)v(j)}|X_{j,m}|^q 
		\right) 
		+
		\bE\left(|X_{0,{2^d}}|^q\right)
		\\
		&= 
		\sum_{j=0}^\infty 
		2^{jq(t - \frac{d}{q} -s + \frac{d}{p})}  
		\bE\left((2^d-1)v(j)\right)
		\bE\left(|X_{j,m}|^q\right) 
		+
		\bE\left(|X_{0,{2^d}}|^q\right)
		\\
		&\le
		2^d\bE\left(|X_{1,1}|^q\right)
		\sum_{j=0}^\infty 
		2^{jq(t - \frac{d}{q} -s + \frac{d}{p}+\frac{\gg}{q})}, 
	\end{align*}
	with $\bE\left(|X_{1,1}|^q\right)<\infty$. 
        The series converges if $t < s+ \frac{d-\gg}{q} - \frac{d}{p}$, 
        in which case $b_T\in L^q(\gO; {B^t_q})$, 
        and hence $b_T\in B^t_q$ holds $P$-a.s. 
	
	2.) Now let $q_0\ge q\ge 1$, $t_0>\frac{d}{q_0}$ and 	
           $t=t_0-\frac{d}{q_0}$,  
	   so that $B_{q_0}^{t_0}\hookrightarrow \cC^t$ holds by~\eqref{eq:besovembedding}.
	   The embedding follows by a direct comparison of the norms 
           in~\eqref{eq:besovnorm},~\eqref{eq:besovnorminf} with $t=t_0-\frac{d}{q_0}$, 
           and also shows that the corresponding embedding constant $C_0>0$ is bounded by $C_0\le 1$.

	We obtain with Hölder's inequality and analogously to the first part the estimate
	\begin{equation}
		\begin{split}\label{eq:q-moments}
			\|b_T\|^q_{L^q(\gO; \cC^t)}
			&\le 
			\bE\left( \|b_T\|^{q_0}_{\cC^t} \right)^{\frac{q}{q_0}} 
                        \\
			&\le 
			\bE\left(
			\|b_T\|^{q_0}_{B_{q_0}^{t_0}}
			\right)^{\frac{q}{q_0}} \\
			&\le 
			2^{\frac{dq}{q_0}}\bE\left(|X_{1,1}|^{q_0}\right)^{\frac{q}{q_0}} 
			\left(\sum_{j=0}^\infty 
			2^{jq_0(t_0 - \frac{d}{q_0} -s + \frac{d}{p}+\frac{\gg}{q_0})}
			\right)^{\frac{q}{q_0}} \\
			&=
			2^{\frac{dq}{q_0}}\bE\left(|X_{1,1}|^{q_0}\right)^{\frac{q}{q_0}} 
			\left(\sum_{j=0}^{\infty} 
			2^{jq_0(t - s + \frac{d}{p} + \frac{\gg}{q_0})}
			\right)^{\frac{q}{q_0}} . 		 
		\end{split}
	\end{equation}
	Now let $t<s-\frac{d}{p}$ be fixed, and let $\gg\in(0,d]$.
	\rev{For every fixed $q\ge \max(2\gg (s-\frac{d}{p}-t)^{-1}, 1)$, we choose 
	$q_0:=q$ to obtain that} 
	\begin{equation}\label{eq:moments}
		\begin{split}
			\|b_T\|^q_{L^q(\gO; \cC^t)}
			&\le
			2^{d}\bE\left(|X_{1,1}|^{q}\right) 
			\left(\sum_{j=0}^{\infty} 
			2^{jq_0(t - s + \frac{d}{p})/2}
			\right)^{\frac{q}{q_0}} 		 \\
			&\le 
			2^{d}\bE\left(|X_{1,1}|^{q}\right)
			\left(\sum_{j=0}^{\infty} 
			2^{-j\gg}
			\right)\\
			&\le 
			C\bE\left(|X_{1,1}|^{q}\right),
		\end{split}
	\end{equation}
	with a constant $C>0$ that is independent of $q$.
	We now define for given $\eps>0$, finite $n\in\bN$ and $p \in [1,\infty)$
        the random variable 
	\begin{equation*}
		E_n(\go):=\sum_{k=0}^n \frac{(\eps \|b_T(\go)\|^p_{\cC^t})^k}{k!}, \quad \go\in\gO.
	\end{equation*}
	Clearly, $E_n(\go)\to \exp(\eps \|b_T(\go)\|^p_{\cC^t})$ holds 
        $P$-.a.s as $n\to \infty$ and Inequality~\eqref{eq:moments} yields, 
        for any $n\in\bN$ and $n_\gg:=\rev{\frac{1}{p}\ceil{2\gg (s-\frac{d}{p}-t)^{-1}}}$,
        that
	\begin{equation*}
		\bE(E_n)
		=
		\sum_{k=0}^n \frac{\eps^k}{k!}
		\bE(\|b_T\|_{\cC^t}^{pk}) 
		\le
		\widetilde C
		+
		\sum_{k=n_\gg}^n \frac{\eps^k}{k!}
		C^{pk}\bE(|X_{1,1}|^{pk}) 
		=
		\widetilde C+
		\bE\left(
		\sum_{k=n_\gg}^n
		\frac{(\eps C^p|X_{1,1}|^p)^k}{k!}
		\right),
	\end{equation*}
	where $\widetilde C = \widetilde C(\gg, s,d,p,t)>0$.
	The monotone convergence theorem then shows that for sufficiently small $\eps>0$ and $t<s-\frac{d}{p}$, it holds
	\begin{equation*}
		\bE\left( \exp(\eps \|b_T\|^p_{\cC^t}) \right)
		\le
		\widetilde C
		+
		\lim_{n\to\infty}
		\bE\left(
		\sum_{k=n_\gg}^n
		\frac{(\eps C^p|X_{1,1}|^p)^k}{k!}
		\right)
		\le 
		\widetilde C +
		\bE\left( \exp(\eps C^p |X_{1,1}|^p) \right) < \infty.
	\end{equation*}
	
	3.) \rev{The claim for $\gg\in(0,d]$ follows by the previous part.}
	For $\gg\in(-\infty,0]$, $q\ge 1$ and $t\in(0,s-\frac{d}{p}-\frac{\gg}{q})$, 
	we finally use $q_0 = q$ and $t_0=t+\frac{d}{q}$ in~\eqref{eq:q-moments} 
        to obtain that 
	\begin{align*}
		\|b_T\|_{L^q(\gO; \cC^t)}
		&\le 
		C^q \bE\left(|X_{1,1}|^q\right) 
		\left(\sum_{j=0}^{\infty} 
		2^{jq(t - s + \frac{d}{p} + \frac{\gg}{q})}
		\right) < \infty.
	\end{align*}
\end{proof}

\begin{rem}\label{rem:regularity}~
	\cite[Theorems 4 and 5]{KLSS21} state that $b_T\in B_p^t$ holds $P$-a.s. for all $t\in(0, s-\gg/p)$, and that $b_T\notin B_p^{s-\gg/p}$ occurs with probability $1-p_\gb>0$ for $\gg\in(0, d]$, where $p_\gb$ is the solution to the equation $p_\gb=((1-\gb)+\gb p_\gb)^{2^d}$ (cf. Lemma~\ref{lem:extinction} in the Appendix.)
	We emphasize that Theorem~\ref{thm:randomtreereg} significantly extends these previous results, 
	as we quantify precisely the regularity of $b_T$ in terms of Besov and Hölder-Zygmund norms.
	
	Recall that we may replace the Hölder-Zygmund spaces $\cC^t$ in \rev{the second part of} Theorem~\ref{thm:randomtreereg} 
        by the "usual" Hölder spaces $\rC^t$ if $t\notin\bN$ (which is \rev{not true} for integer $t$).
	Theorem~\ref{thm:randomtreereg} shows that a wavelet density $\gb=2^{\gg-d}<1$
        improves smoothness in $B_q^t$, as the upper bound $t < s+ \frac{d-\gg}{q} - \frac{d}{p}$ is decreasing in $\gg\in(-\infty,d]$. 
	However, given that $\gg>0$ we may not expect to gain any (pathwise) Hölder regularity beyond $t<s-\frac{d}{p}$. 
This is not surprising with regard to Remark~\ref{rem:Besov-random-tree}: $b_T$ 
admits an infinite series expansion on random fractals in $\cD$ for $\gb>2^{-d}$ with positive probability. 
Hence, the local Hölder-regularity of $b_T$ on such fractals corresponds to a $B^s_p$-random variable $b$ 
as in Definition~\ref{def:besovprior} (with full wavelet density $\gb=1$).
In case that $\gg\le 0$, the series expansion of $b_T$ terminates almost surely after a finite number of terms. \rev{As $\mathbf\Psi\subset \rC^{\ga}$, this shows that $b_T\in\rC^\ga$ almost surely if $\gg\le 0$. 
Furthermore, we may increase the smoothness exponent $t$ for $b_T\in L^q(\gO; \cC^t)$ 
to the admissible range $t<s-\frac{d}{p}-\frac{\min(\gg,0)}{q}$.}
For large $q$, we see that essentially the restriction $t<s-\frac{d}{p}$ applies as for $\gg>0$. 
This in turn indicates that the bound
\begin{equation*}
	\bE\left(\exp\left(\eps \|b\|_{\cC^t}^p\right)\right)<\infty,\quad \eps\in(0,\eps_p),
\end{equation*}
from part 2.) of Theorem~\ref{thm:randomtreereg} can not be improved to Hölder indices $t\ge s-\frac{d}{p}$, 
even if $\gg\le0$.
\end{rem}

\section{Linear Elliptic PDEs with Besov Random Coefficients}
\label{sec:randompdes}
In this section, we first recall well-posedness and regularity results for 
linear, second order elliptic diffusion problems with random coefficient. 
Thereafter, we transfer the results to a setting with Besov tree random diffusion 
coefficient by exploiting the results from Section~\ref{sec:besov-rv}.
\subsection{Well-posedness and regularity}
\label{subsec:well-posedness}

Let $\cD\subset\bR^d$, $d\in\{1,2,3\}$, be a convex polygonal domain \rev{for $d=2,3$, 
and a finite interval for $d=1$.}
with the boundary $\partial \cD$ consisting of a finite number of line or plane segments.
We consider the random elliptic problem to find $u(\go):\cD\to \bR$ for given $\go\in\gO$ such that 
\begin{equation}\label{eq:ellipticpde}
	\begin{alignedat}{2}
		-\nabla\cdot(a(\go)\nabla u(\go)) &= f\quad &&\text{in $\cD$}, \quad 
		u(\go) = 0 \quad &&\text{on $\partial\cD$}.
	\end{alignedat}
\end{equation}
The diffusion coefficient $a$ in Problem~\eqref{eq:ellipticpde} admits positive paths on $\cD$, i.e., $a(\go):\cD\to \bR_{>0}$.
Moreover, $a$ is a random variable $a:\gO\to\cX$, taking values in a suitable Banach space $\cX\subset L^\infty(\cD)$.
The source term $f:\cD\to\bR$ is assumed to be a deterministic function for the sake of simplicity,
but may as well be modeled by a random function $f:\cD\times\gO\to\bR$. 
For the variational formulation of Problem~\eqref{eq:ellipticpde} 
we define $H:=L^2(\cD)$, $V:=H_0^1(\cD)$ and recall that 
$\left\|\cdot\right\|_V:V\to\bR_{\ge 0},\: v\mapsto \|\nabla v \|_H$ defines a norm on $V$ by Poincare's inequality.
The weak formulation of Problem~\eqref{eq:ellipticpde} for fixed $\go\in\gO$ is to find $u(\go)\in V$ such that for any $v\in V$ it holds 
\begin{equation}\label{eq:ellipticpdeweak}
	\int_\cD a(\go)\nabla u(\go)\cdot\nabla v dx = \dualpair{V'}{V}{f}{v}.
\end{equation}
  
\begin{defi}
	The map $\omega\mapsto u(\go) \in V$ with $u(\go)$ the solution of~\eqref{eq:ellipticpdeweak} 
        is the \textit{pathwise weak solution}.
\end{defi}

Existence and uniqueness of pathwise weak solutions are ensured by the following theorem.

\begin{thm}\label{thm:well-posed}
	Let $a:\gO\to L^\infty(\cD)$ be strongly $\cA/\cB(L^\infty(\cD))$-measurable such that  
	\begin{equation}\label{eq:coervity}
		a_-(\go):=\essinf_{x\in\cD} \: a(x,\go)>0,\quad \text{$P$-a.s.,}
	\end{equation}
	and let $f\in V'$. 
        Then, there exists \emph{for all $\go\in\gO$} 
        a unique weak solution $u(\go)\in V$ to Problem~\eqref{eq:ellipticpde}.
        The map $u:\gO\to V$ is strongly $\cA/\cB(V)$-measurable.  
\end{thm}

\begin{proof}
	By the completeness of $(\gO,\cA,P)$,
	we may assume without loss of generality that $a_-(\go)>0$ and 
	$a(\go)\in L^\infty(\cD)$ holds for \textit{all} $\go\in\gO$\footnote{If this holds only $P$-.a.s., we may modify 
	$a:\gO\to L^\infty(\cD)$ on a $P$-nullset to obtain a strongly measurable modification 
	$\widetilde a:\gO\to L^\infty(\cD)$ of $a$, 
	so that  $\essinf_{x\in\cD} \widetilde a(x, \go)>0$ and $\widetilde a\in L^\infty(\cD)$ holds for \textit{all} $\go\in\gO$.
	In fact, let 
	\begin{equation*}
		A_0:=\{\go\in\gO|\;a_-(\go)\le 0 \text{ or } a(\go)\notin L^\infty(\cD)\}.
	\end{equation*}
	Then $P(A_0)=0$ by assumption, and hence $A_0\in\cA$ by completeness of $(\gO,\cA,P)$.
	Thus, we may consider, for instance, the modification  
	$\widetilde a(\go):=a(\go)\indi_{\{\go\notin A_0\}} + \indi_{\{\go\in A_0\}}$.
	It is readily verified that $\widetilde a$ is strongly $\cA/\cB(L^\infty(\cD))$-measurable, 
        and for all $\go\in\gO$ it holds $\essinf_{x\in\cD} \widetilde a(x, \go)>0$ and $\widetilde a\in L^\infty(\cD)$.}.
		
	Existence and uniqueness of a pathwise solution $u(\go)$ now follows for all $\go\in\gO$ by the Lax-Milgram Lemma. 
        To \rev{show strong} measurability of $u$, \rev{we use Lipschitz dependence of the coefficient-to-solution map}:
        consider two diffusion coefficients $a_1,a_2:\gO\to L^\infty(\cD)$
        that satisfy the assumption of the theorem with lower bounds $a_{1,-}, a_{2,-}>0$ as in~\eqref{eq:coervity} 
        and denote by $u_1,u_2:\gO\to V$ the associated unique weak solutions. 
	Equation~\eqref{eq:ellipticpde} together with 
	$\|v\|_V^2 = \|\nabla v\|_H^2$
	and Hölder's inequality yields for any fixed coefficients $a_1,a_2 \in L^\infty(\cD)$ such that 
	$a_{i,-} := \essinf_{x\in \cD} a_{i}(x) > 0$ that 
	\begin{equation}\label{eq:perturbationest}
			\begin{split}
				\|u_1-u_2\|_V 
				&\le \frac{\|u_2\|_{V}}{a_{1,-}}\|a_1-a_2\|_{L^\infty(\cD)} 
				 \le 
				\frac{\|f\|_{V'}}{a_{1,-} a_{2,-}}\|a_1-a_2\|_{L^\infty(\cD)}.
			\end{split}
	\end{equation}
	Therefore, the data-to-solution map 
	$U: S \to V,\; a \mapsto u$ is (Lipschitz) continuous on the set
	$S:=\{ a\in L^\infty(\cD)|\, \essinf_{x\in \cD} a(x) > 0 \}$.
	Since the pathwise weak solution $u:\gO\to V$ of 
	\eqref{eq:ellipticpde} may be written as $u=U\circ a$, the claim follows with the strong $\cA/\cB(L^\infty(\cD))$-measurability of $a:\gO\to L^\infty(\cD)$. 
\end{proof}
Lipschitz continuity \eqref{eq:perturbationest} of the data-to-solution map will be essential in deriving
error estimates in Section~\ref{sec:approximation} ahead, and also implies 
strong measurability of random solutions.
\begin{prop}\label{prop:perturbation}
Let $a_1,a_2:\gO\to L^\infty(\cD)$ be strongly $\cA/\cB(L^\infty(\cD))$-measurable 
such that  
	\begin{equation}\label{eq:coercivity}
		a_{i,-}(\go):=\essinf_{x\in\cD} \: a_i(x,\go)>0,\quad \text{$P$-a.s. for $i\in\{1,2\}$.}
	\end{equation}
Then, for every $f\in V'$ exists for $i\in\{1,2\}$ 
and for all $\go\in\gO$ a unique weak solution $u_i(\go)\in V$ 
to Problem~\eqref{eq:ellipticpde} (with $a$ in place of $a_i$). 
There holds the continuous-dependence estimate 
	\begin{equation*}
		\begin{split}
			\|u_1-u_2\|_V 
			\le 
			\frac{\|f\|_{V'}}{a_{1,-} a_{2,-}}\|a_1-a_2\|_{L^\infty(\cD)}.
		\end{split}
	\end{equation*}
\end{prop}
\begin{proof}
This follows immediately with Theorem~\ref{thm:well-posed} and~\eqref{eq:perturbationest}.
\end{proof}

From the regularity analysis of deterministic linear elliptic problems 
it is well known that $H^s(\cD)$-regularity of $u$ may be derived for certain $s>1$, 
provided that $a$ is Hölder continuous. 
The corresponding estimates usually do not reveal the explicit dependence of constants on $a(\go)$ 
or bounds on the Hölder norm $\|a(\go)\|_{\rC^t}$.  
For the stochastic problem and the ensuing numerical analysis in 
Sections~\ref{sec:approximation} and \ref{sec:mlmc}, however, 
we need the explicit dependence for given $\go$ to ensure that all pathwise estimates 
also hold in in $L^q(\gO; H^s(\cD))$ for suitable $q\ge1$.
To obtain explicit estimates, we follow the approach from~\cite[Chapter 3.3]{DNSZ22} 
for parametric elliptic PDEs, where regularity estimates 
are derived via the K-method of function space interpolation.
\footnote{Recall the $K$-method of interpolation of two Banach spaces
          $(A_0, \left\|\cdot\right\|_{A_0})$ and $(A_1, \left\|\cdot\right\|_{A_1})$ 
          with continuous embedding $A_1\hookrightarrow A_0$:
          their \emph{K-functional} is defined by
\begin{equation*}
	K(a, z; A_0, A_1):=\inf_{a_1\in A_1} 
	\{\|a-a_1\|_{A_0} + z\|a_1\|_{A_1}\}, \quad a\in A_0,\; z>0. 
\end{equation*}
For any $r\in(0,1)$ and $q\in[1,\infty]$ the \emph{interpolation space of order $r$ with fine index $q$} 
is 
\begin{equation*}
	[A_0, A_1]_{r,q} = \left\{a\in A_0|\, \|a\|_{[A_0, A_1]_{r,q}}<\infty \right\},
\end{equation*}
where 
\begin{equation*}
	\|a\|_{[A_0, A_1]_{r,q}}
	:=
	\begin{cases}
		\left(\int_0^\infty z^{-rq} K(a,z;A_0, A_1)^q\frac{1}{z}dz\right)^{\frac{1}{q}},
		&\quad q\in[1,\infty), \\
		\sup_{z>0} z^{-r}K(a,z;A_0, A_1),&\quad q=\infty.
	\end{cases}
\end{equation*}
The set $[A_0, A_1]_{r,q}$ 
is a (generally non-separable) Banach space. 
}
One obtains in particular
Hölder spaces $\rC^r(\ol\cD)$ by interpolation (\cite[Lemma 7.36]{adams2003sobolev}):
\begin{equation*}
	\rC^r(\ol\cD)=[L^\infty(\cD), W^{1,\infty}(\cD)]_{r,\infty}, \quad r\in(0,1).
\end{equation*}
To investigate spatial regularity of solutions to~\eqref{eq:ellipticpde}, we introduce the normed space 
\begin{equation}
	W:=\{v\in V|\; \gD v \in H\}, \quad \|v\|_W:=\|\gD v\|_H.
\end{equation}
Note that $v=0\Leftrightarrow \|v\|_W=0$ follows by the maximum principle, since $v\in V=H^1_0(\cD)$ has vanishing trace. 
We formulate regularity results in terms of the interpolation space 
\begin{equation}\label{eq:interp-space}
	W^r:=[V, W]_{r,\infty}, \quad r\in(0,1).
\end{equation}
For a concise notation, we further set $W^1:=W$ in the following.
\begin{lem}\cite[Propositions 3.2 and 3.5]{DNSZ22}
	\label{lem:ellipticreg}
	Let $a:\gO\to \rC^r(\ol\cD)\subset L^\infty(\cD)$ be strongly measurable for some $r\in(0,1]$ 
        such that $a_-(\go)>0$ holds $P$-a.s. and let $f\in H$.
	Then, there is a constant $C=C(r,\cD)$, such that it holds 
	\begin{equation}\label{eq:ellipticregularity}
		\|u(\go)\|_{W^r}
		\le 
		\frac{C}{a_-(\go)}
		\left(1+\left(\frac{\|a(\go)\|_{\rC^r(\ol\cD)}}{a_-(\go)}\right)^{1/r}\right)\|f\|_H.
	\end{equation}
\end{lem}

All results from this subsection so far hold under the considerably weaker assumption that $\cD\subset\bR^d$ is a bounded Lipschitz domain. However, since $\cD$ is assumed convex, we are able to embed $W^r$ in (fractional) Sobolev spaces.
This is made precise in the following Lemma, which is in required for the finite element error analysis in Section~\ref{subsec:fem}.  

\begin{lem}\label{lem:interp-sobolev}
	Let $\cD$ be convex, $W^r:=[V, W]_{r,\infty}$ for $r\in(0,1)$ and let $W^1:=W$.
	Then, it holds that $W=W^1\hookrightarrow H^2(\cD)$. Moreover, $W^r\hookrightarrow H^{1+r_0}(\cD)$ for any $r_0\in(0,r)$.
\end{lem}

\begin{proof}
	By convexity of $\cD$, we have that $\|v\|_{H^2(\cD)}\le C_\cD \|v\|_W$ holds for all $v\in W$,  where $C_\cD$ only depends on the diameter of $\cD$, see, e.g., \cite[Theorem 3.2.1.2]{grisvard2011elliptic}. Thus, $W\hookrightarrow H^2(\cD)\cap V$ follows. 
	
	For the case $r\in(0,1)$, we recall that there is $C_V>0$, such that $\|v\|_{H^1(\cD)}\le C_{V}\|v\|_V$ holds for all $v\in V$ by Poincar\'e's inequality. Moreover, we have $\|w\|_{H^2(\cD)}\le C_{\cD}\|w\|_W$ for any $w\in W$, and hence $W\subset H^2\cap V$ from the first part. 
	For $v\in V\subset H^1$ this yields 
	\begin{align*}
		\|v\|_{[H^1(\cD),H^2(\cD)]_{r,\infty}}
		&= 
		\sup_{z>0} z^{-r} \inf_{w\in H^2}\{\|v-w\|_{H^1(\cD)} + \rev{z}\|w\|_{H^2(\cD)}\} \\
		&\le 
		\sup_{z>0} z^{-r} \inf_{w\in W}\{\|v-w\|_{H^1(\cD)} + \rev{z}\|w\|_{H^2(\cD)}\} \\
		&\le 
		\sup_{z>0} z^{-r} \inf_{w\in W}\{C_V\|v-w\|_{V} + C_\cD\rev{z}\|w\|_{W}\} \\
		&\le 
		\max(C_V, C_\cD) \|v\|_{W^r}.
	\end{align*}
	Hence, $W^r\hookrightarrow [H^1(\cD),H^2(\cD)]_{r,\infty}$.
	The claim now follows since for any $\eps\in(0,1+r)$ there holds
	\begin{equation*}
		[H^1(\cD),H^2(\cD)]_{r,\infty}
		=[H^{1+r-\eps}(\cD),H^{1+r+\eps}(\cD)]_{\frac{1}{2},\infty}\hookrightarrow H^{1+r-\eps}(\cD),
	\end{equation*} 
	see \cite[Section 7.32]{adams2003sobolev}.	
\end{proof}

\subsection{Besov random tree priors as log-diffusion coefficient}
\label{subsec:randomtree-diffusion}
To formulate Problem~\eqref{eq:ellipticpde} with a Besov random tree coefficient, we assume 
that $\cD\subseteq\bT^d$. We follow \cite[Section 2]{TriebelFctSpcDom} and define, 
for given $\go\in\gO$, the random element $b_T(\go):\cD\to\bR$ as the \emph{restriction} of a periodic function in $B_{p,p}^{s,per}$ to the domain $\cD$.
The restriction $\varphi|_{\cD}$ of any $\varphi\in\rS'(\bR^d)$ to $\cD$ is in turn given by the element $\varphi|_{\cD}\in \rD'(\cD)$ such that 
\begin{equation*}
	\dualpair{\rD'(\cD)}{\mathrm D(\cD)}{\varphi|_{\cD}}{v}
	=
	\dualpair{\rS'(\bR^d)}{\rS(\bR^d)}{\varphi}{v_0},
	\quad 
	v\in\mathrm D(\cD),
\end{equation*}
where $v_0\in\mathrm D(\bR^d)\subset \rS(\bR^d)$ denotes the zero-extension of any $v\in \mathrm D(\cD)$.

\begin{defi}\label{def:restrictedbesovprior}
	Let $\cD\subseteq\bT^d$ be a bounded, connected domain.
	Let $b_T$ be given in 
	Definition~\ref{def:randomtree-prior} for $p\in[1,\infty)$, $s>0$ and $\gb=2^{\gg-d}\in[0,1]$, 
	and let $\mathrm {prl}^{per}:B_{p,p}^s(\bT^d)\to B_{p,p}^{s,per}(\bR^d)$ 
        denote the isomorphic extension operator from~\eqref{eq:extension-map}.
	Then we define for any $\go\in\gO$ 
	\begin{equation*}
		b_{T,\cD}(\go)
		:=
		(\mathrm{prl}^{per}b_T(\go))|_{\cD},
	\end{equation*}
	and call $b_{T,\cD}$ a \textit{$B_p^s(\cD)$-valued random variable}.
\end{defi}

\begin{rem}
	In case that $\cD=\bT^d$, we may readily use the identification $b_{T,\cD}=b_T$. Note that $b_{T,\cD}$ is periodic in this case, in the sense that there exists an extension $\mathrm{prl}^{per}b_{T,\cD}\in B_{p,p}^{s,per}(\bR^d)$. 
	If $\cD\subsetneq\bT^d$, however, $b_{T,\cD}$ is not (necessarily) periodic, but \emph{merely the restriction of a periodic function} from the torus $\bT^d$. 
\end{rem}

\begin{rem}
\label{rmk:Domain}

The same procedure could be applied for general bounded domains $\cD\not\subset\bT^d$, 
by extending Definition~\eqref{def:randomtree-prior} from the torus $\bT^d$ 
to a sufficiently large (periodic) domain $[-L,L]^d$ for $L>1$. 
This would increase the index-set $K_j$ of wavelet coefficients by at most 
a constant factor on each dyadic scale $j$.
However, all regularity proofs from Section~\ref{sec:besov-rv} are carried out similar in this setting, 
with minor changes to absolute constants. For instance, the admissible range of $\eps$ in 
Proposition~\ref{prop:besovreg} may \rev{become} smaller if $L>1$, 
but the smoothness parameter $t\in(0,s-\frac{d}{p})$ is unaffected.
Therefore, assuming $\cD\subseteq\bT^d$ for the sake of brevity does not have any 
substantial impact on the following results.
\end{rem}
We consider Problem~\eqref{eq:ellipticpde}, 
resp. its weak formulation~\eqref{eq:ellipticpdeweak}, 
with $a(\go):=\exp\left(b_T(\go)\right)$, 
where $b_T$ is a $B_p^s$-random variable with wavelet density $\gb$.
That is, we model the log-diffusion by a Besov random tree prior to incorporate fractal structures. 
With this preparation, we are able to derive well-posedness and regularity 
of the corresponding pathwise weak solution. 
\begin{thm}\label{thm:solution-regularity}
	Let $a:=\exp\left(b_{T,\cD}\right)$ with $b_{T,\cD}$ given in 
        Definition~\ref{def:restrictedbesovprior} for $p\in[1,\infty)$, $s>0$ and $\gb=2^{\gg-d}\in[0,1]$, so that $sp>d$.
	Furthermore, let $f\in V'$.
	
	\begin{enumerate}[1.)]
		\item 	 Then, there exists almost surely a unique weak solution $u(\go)\in V$ to~\eqref{eq:ellipticpde} and $u:\gO\to V$ is strongly measurable.
		
		\item For sufficiently small $\gk>0$ in~\eqref{eq:p-exponential}, 
                there are constants $\ol q\in(1,\infty)$ and $C>0$ such that
		\begin{equation*}
			\|u\|_{L^q(\gO; V)}\le C \|f\|_{V'}  <\infty
			\quad
			\begin{cases}
				&\text{for $q\in[1,\ol q)$ if $p=1$, and} \\
				&\text{for any $q\in [1,\infty)$ if $p>1$}.
			\end{cases}
		\end{equation*}
	
		\item 
		Let $r\in (0,s-\frac{d}{p})\cap(0,1]$ , $f\in H$ and $W^r$ as in~\eqref{eq:interp-space}.	
		For sufficiently small $\gk>0$ in~\eqref{eq:p-exponential},  
                there are constants $\ol q\in(1,\infty)$ and $C>0$ such that	
		\begin{equation*}
			\|u\|_{L^q(\gO; W^r)} \le C \|f\|_{H} <\infty
			\quad
			\begin{cases}
				&\text{for $q\in[1,\ol q)$ if $p=1$ and} \\
				&\text{for any $q\in [1,\infty)$ if $p>1$}.
			\end{cases}
		\end{equation*}
	\end{enumerate}

\end{thm} 

\begin{proof}
1.) As $sp>d$, Theorem~\ref{thm:randomtreereg}	shows that $b_T\in\rC(\bT^d)$ holds $P$-a.s.
Moreover, $b_T:\gO\to \rC(\bT^d)$ is strongly measurable by Proposition~\ref{prob:measurability}, 
and thus in particular strongly $\cA/\cB(L^\infty(\bT^d))$-measurable, since $\cB(\rC(\bT^d))\subset\cB(L^\infty(\bT^d))$.
As $b_{T,\cD}$ in Definition~\ref{def:restrictedbesovprior} is the restriction of ${\rm ext}^{per}b_T$  to $\cD\subset\bT^d$, and $a=\exp\circ\,b_{T,\cD}$, it follows 
that, $a:\gO\to\rC(\ol\cD)$, and $a$ is strongly $\cA/\cB(L^\infty(\cD))$-measurable such that $a_->0$ holds $P$-a.s.
Theorem~\ref{thm:well-posed} then guarantees the $P$-a.s. existence of a unique pathwise weak solution $u(\go)$. Moreover, $u:\gO\to V$ is strongly measurable and Equation~\eqref{eq:ellipticpdeweak} shows that
\begin{equation*}
	\|u(\go)\|_V\le \frac{\|f\|_{V'}}{a_-(\go)}.
\end{equation*}

2.) To show the second part, we fix $t\in(0,s-\frac{d}{p})$ and $q\ge1$ to see that 
\begin{align*}
	\|u\|_{L^q(\gO; V)}^q
	&\le \bE\left(a_-^{-q}\right)\|f\|_{V'}^q \\
	&= \bE\left( \left(\essinf_{x\in\cD}\; \exp(-b_{T,\cD}(x))\right)^q \right)\|f\|_{V'}^q \\
	&= \bE\left( \exp\left(\essinf_{x\in\cD}\;  -qb_{T,\cD}(x)\right) \right)\|f\|_{V'}^q \\
	&\le  \bE\left( \exp\left(\esssup_{x\in\bT^d}\;  qb_T(x)\right) \right)\|f\|_{V'}^q \\
	&= \bE\left(\exp\left(q\|b_T\|_{L^\infty(\bT^d)}\right)\right)\|f\|_{V'}^q \\	
	&\le \bE\left(\exp\left(q\|b_T\|_{\cC^t}\right)\right)\|f\|_{V'}^q.
\end{align*}
We have used that $\exp(\cdot)$ is strictly increasing for the second equality, 
and that $b_T(x)$ is a centered random variable such that $b_T$ and $-b_T$ are equal in distribution. 
\rev{ This implies in turn that $a_-$ and $\|a\|_{L^\infty(\cD)}$ are equal in distribution, which we used in the second inequality.}
The last estimate is due to $\|b_T\|_{L^\infty(\bT^d)}\le \|b_T\|_{\cC^t}$ for any $t>0$.
For $p=1$, we note that $\eps_p$ in the second part of Theorem~\ref{thm:randomtreereg} 
may be chosen as $\eps_p=(\gk C)^{-1}$, where $C>0$ is the constant in~\eqref{eq:moments}. 
\rev{Hence}, for sufficiently small $\gk>0$, we may set $\ol q := \eps_p>1$ in the claim. 
In case that $p>1$, Young's inequality shows that for any $q\ge1$ there is an arbitrary small $\eps>0$ and a constant $C_\eps=\rev{C_\eps(p,q)}\in(0,\infty)$ such that
\begin{equation*}
	q\|b_T\|_{\cC^t} \le \eps \|b_T\|_{\cC^t}^p + C_\eps.
\end{equation*}
Thus, we have no restrictions on $q\in[1,\infty)$, which proves the second part of the claim.

3.) \rev{
Let $\left\|\cdot\right\|$ denote the Euclidean norm on $\cD$. Observe that for any fixed $r\in(0,1)$ we obtain by Taylor expansion and since $\exp(\cdot)$ is strictly increasing that 
\begin{equation}\label{eq:exp-hoelder}
	\begin{split}
		\|\exp(b_{T,\cD})\|_{\rC^r(\ol\cD)}
		&=
		\sup_{x,y\in\cD,\, x\neq y}
		\frac{|\exp(b_{T,\cD}(x))-\exp(b_{T,\cD}(y))|}{\|x-y\|^r} + \|\exp(b_{T,\cD})\|_{L^\infty(\cD)} \\
		&\le 
		\|\exp(b_{T,\cD})\|_{L^\infty(\cD)}\left(
		\sup_{x,y\in\cD,\, x\neq y} \frac{|b_{T,\cD}(x)-b_{T,\cD}(y)|}{\|x-y\|^r} + 1\right) \\
		&\le 
		\exp(\|b_{T,\cD}\|_{L^\infty(\cD)})
		\left( \|b_{T,\cD}\|_{\rC^r(\ol\cD)} + 1\right).
	\end{split}
\end{equation}
We obtain further for $r=1$ that 
\begin{equation} \label{eq:exp-hoelder1}
	\begin{split}
		\|\exp(b_{T,\cD})\|_{\rC^1(\ol\cD)}
		&\le 
		\exp(\|b_{T,\cD}\|_{L^\infty(\cD)})
		\left( \|b_{T,\cD}\|_{\rC^1(\ol\cD)} + 1\right).
	\end{split}
\end{equation}
For any fixed $r\in(0,s-\frac{d}{p})\cap(0,1]$ Lemma~\ref{lem:ellipticreg} now shows that
\begin{equation}\label{eq:sobolevreg}
	\begin{split}
		\|u\|_{L^q(\gO; W^r)}^q	
		&\le 
		C^q\bE\left[a_-^{-q}
		\left(1+a_-^{-\frac{1}{r}}\|\exp(b_{T,\cD})\|_{\rC^r(\ol\cD)}^{\frac{1}{r}}\right)^q
		\right]\|f\|_{H}^q \\
		&\le 
		C^q\bE\left[a_-^{-q}
		\left(1+a_-^{-\frac{1}{r}}
		\exp(\|b_{T,\cD}\|_{L^\infty(\cD)})^{\frac{1}{r}}
		\left( \|b_{T,\cD}\|_{\rC^r(\ol\cD)} + 1\right)^{\frac{1}{r}}\right)^q
		\right]\|f\|_{H}^q \\
		&\le 
		C^q\bE\left[a_-^{-q}
		2^{q-1}\left(1+a_-^{-\frac{q}{r}}
		\exp(\|b_{T,\cD}\|_{L^\infty(\cD)})^{\frac{q}{r}}
		2^{q-1}\left( \|b_{T,\cD}\|_{\rC^r(\ol\cD)}^{\frac{q}{r}} + 1\right)\right)
		\right]\|f\|_{H}^q \\
		&\le 
		C\bE\left[
		\exp(\|b_{T,\cD}\|_{L^\infty(\cD)})^{q+\frac{2q}{r}}
		\left( \|b_{T,\cD}\|_{\rC^r(\ol\cD)}^{\frac{q}{r}} + 1\right)
		\right]\|f\|_{H}^q, 
	\end{split}
\end{equation}
where we have used \eqref{eq:exp-hoelder}, \eqref{eq:exp-hoelder1} in the second step, applied Jensen's inequality twice in the third step, and used again that $b_T$ and  $-b_T$ are equal in distribution together with
$\exp(\|b_{T,\cD}\|_{L^\infty(\cD)}) \ge 1$ to derive the third estimate.
We may further assume without loss of generality that $\|b_{T,\cD}\|_{\rC^r(\ol\cD)}\ge 1$ to obtain with~\eqref{eq:sobolevreg} and Hölder's inequality for $q_1, q_2>1$ such that $\frac{1}{q_1}+\frac{1}{q_2}=1$
\begin{equation}\label{eq:sobolevreg2}
	\begin{split}
		\|u\|_{L^q(\gO; W^r)}^q	
		&\le 
		C\bE\left[
		\exp\left(q_1\left(q+\frac{2q}{r}\right)\|b_T\|_{\rC^r}\right)
		\right]^{\frac{1}{q_1}}
		\bE\left[
		\|b_T\|_{\rC^r}^{\frac{q_2q}{r}}
		\right]^{\frac{1}{q_2}}\|f\|_{H}^q,
	\end{split}
\end{equation}
}
where we have also used that $\|b_{T,\cD}\|_{L^\infty(\cD)}\le \|b_{\rev{T}}\|_{\rC^r}$ and that $\|b_{\rev{T,\cD}}\|_{\rC^r(\ol\cD)} \le \|b_T\|_{\rC^r}$ for any $r>0$.
To bound the Hölder-norm $\|b_T\|_{\rC^r}$ in~\eqref{eq:sobolevreg2}, we first consider the case $r<1$. Then, we recall from Subsection~\ref{subsec:besovspaces} that $\rC^r=\cC^r$ with equivalent norms, thus $\|b_T\|_{\rC^r}\le C\|b_T\|_{\cC^r}$.
If $r=1$, then $s-\frac{d}{p}>1$, and we use the same argument to derive the bound $\|b_T\|_{\rC^r}\le \|b_T\|_{\rC^{r+\eps}}\le C\|b_T\|_{\cC^{r+\eps}}$ for any $\eps\in(0,s-\frac{d}{p}-1)$.

For $p=1$, Theorem~\ref{thm:randomtreereg} now shows again that for sufficiently small $\gk>0$, there are admissible choices $q,\rev{q_1,q_2}\in[1,\infty)$, dependent on $r$, such that the right hand side in~\eqref{eq:sobolevreg2} is finite.
The proof is concluded by noting that $q\in[1,\infty)$ may again be arbitrary large in~\eqref{eq:sobolevreg2} if $p>1$, independent of $r$.
\end{proof}
\section{Pathwise Finite Element Approximation}
\label{sec:approximation}
\subsection{Dimension truncation}\label{subsec:truncation}

To obtain a tractable approximation of $b_T$ in~\eqref{eq:randomtreeprior}, we truncate the wavelet series expansion after $N\in\bN$ scales to obtain the $\emph{truncated random tree Besov prior}$ 
\begin{equation}\label{eq:randomtreepriortrunc}
	b_{T,N}(\go):= \sum_{\substack{(j,k,l)\in \cI_T(\go) \\ j\le N}} \eta_jX_{j,k}^l(\go)\psi_{j,k}^l, \quad \go\in\gO.
\end{equation}

The corresponding diffusion problem in weak form with truncated coefficient for fixed $\go\in\gO$ is to find  $u_N(\go)\in V$ such that for all $v\in V$
\begin{equation}\label{eq:ellipticpdetrunc}
	\int_\cD a_N(\go)\nabla u_N(\go)\cdot\nabla v dx = \dualpair{V'}{V}{f}{v},
\end{equation}
where 
\begin{equation}\label{eq:diffusioncoefftrunc}
	a_N:\gO\to L^\infty(\cD), \quad \go\mapsto \exp(b_{T,N}(\go)|_{\cD}).
\end{equation}
Existence, uniqueness, and regularity of $u_N$ follows analogously as for $u$ in the previous section.

\begin{cor}\label{cor:truncated-regularity}
	Let $N\in\bN$, $a_N=\exp\left(b_{T,N}|_{\cD}\right)$ with 
        $b_{T,N}$ be given as \rev{in~\eqref{eq:randomtreepriortrunc}} for 
        $p\in[1,\infty)$, $s>0$ and $\gb=2^{\gg-d}\in[0,1]$, so that $sp>d$.
	Furthermore, let $f\in V'$. Then the following holds.
	
	\begin{enumerate}[1.)]
		\item 	 
                There exists almost surely a unique weak solution $u_N(\go)\in V$ to the \emph{truncated} 
                Problem~\eqref{eq:ellipticpdetrunc} and $u_N:\gO\to V$ is strongly measurable.
		
		\item 
                For sufficiently small $\gk>0$ in~\eqref{eq:p-exponential}, 
                there are constants $\ol q\in(1,\infty)$ and $C>0$ such that for \emph{any $N\in\bN$} 
		\begin{equation*}
			\|u_N\|_{L^q(\gO; V)}\le C \|f\|_{V'}  <\infty
			\quad
			\begin{cases}
				&\text{for $q\in[1,\ol q)$ if $p=1$, and} \\
				&\text{for any $q\in [1,\infty)$ if $p>1$}.
			\end{cases}
		\end{equation*}
		
		\item 
		Let $r\in (0,s-\frac{d}{p})\cap(0,1]$ and $f\in H$.	
		There are constants $\ol q\in(1,\infty)$ and $C>0$ such that for \emph{any $N\in\bN$} 	
		\begin{equation*}
			\|u_N\|_{L^q(\gO; W^r)}\le C \|f\|_{H} < \infty
			\quad
			\begin{cases}
				&\text{for $q\in[1,\ol q)$ if $p=1$ and} \\
				&\text{for any $q\in [1,\infty)$ if $p>1$}.
			\end{cases}
		\end{equation*}
	\end{enumerate}
\end{cor} 

\begin{proof}
	The result follows analogously to Theorems \ref{thm:randomtreereg} and~\ref{thm:solution-regularity}, 
    upon observing that $\|b_{T,N}(\go)\|_{B^t_q}\le \|b_{T}(\go)\|_{B^t_q}$
        holds $P$-a.s. for any $t>0$, $q\in [1,\infty]$, and $N\in\bN$.
\end{proof}

The important observation from Corollary~\ref{cor:truncated-regularity} is that the bounds are independent of $N$, which is crucial when estimating the finite element discretization error of $u_N$ in the next subsection. 
We bound the truncation errors $a-a_N$ and $u-u_N$ in the remainder of this section.

\begin{prop}\label{prop:truncation}
	Let $a:=\exp\left(b_{T,\cD}\right)$ with 
        $b_{T,\cD}$ as given in Definition~\ref{def:restrictedbesovprior}
        with $p\in(1,\infty)$, $s>0$, $\gb=2^{\gg-d}\in[0,1]$ and such that $sp>d+\min(\gg,0)$.
	Let $b_{T,N}$ and $a_N$ be the approximations of $b_T$ and $a$ 
        for given $N\in\bN$ as \rev{in~\eqref{eq:randomtreepriortrunc}} and \eqref{eq:diffusioncoefftrunc}, respectively.
	\begin{enumerate}[1.)]
		\item 
                For any $q\ge 1$ and $t\in(0,s-\frac{d}{p}-\frac{\min(\gg,0)}{q})$ there is a 
                constant $C>0$ such that for every $N\in\bN$ it holds
		\begin{align*}
			\| b_{T,\cD}-b_{T,N}|_\cD \|_{L^q(\gO; \cC^t(\ol\cD))}
			\le 
			C 2^{N(t-s+\frac{d}{p}+\frac{\min(\gg,0)}{q})}.
		\end{align*}
		\item 	Moreover, for any $q\ge 1$, $\eps>0$ and $t\in(0,s-\frac{d}{p}-\frac{\min(\gg,0)}{q})$ there is a $C>0$ such that for every $N\in\bN$ it holds
		\begin{align*}
			\|a-a_N\|_{L^q(\gO; \cC^t(\ol\cD))}
			\le 
			C 2^{N(t-s+\frac{d}{p}+\frac{\min(\gg+\eps,0)}{q})}.
		\end{align*}
	\end{enumerate}
\end{prop} 

\begin{proof}
	1.) Let $q_0\ge q$, $t_0>\frac{d}{q_0}$ and $t = t_0-\frac{d}{q_0}$, 
        so that $B_{q_0}^{t_0}\hookrightarrow \cC^t$.	
	For any fixed $N\in\bN$, we obtain with Hölder's inequality 
        analogously to the proof of Theorem~\ref{thm:randomtreereg} 
        the estimate
	\begin{equation}
		\begin{split}\label{eq:truncsum}
			\| b_{T,\cD}-b_{T,N}|_\cD \|_{L^q(\gO; \cC^t(\ol\cD))}
			& \le\|b_T-b_{T,N}\|_{L^q(\gO; \cC^t)} 
                        \\
			&\le 
			\bE\left(
			\|b_T-b_{T,N}\|^{q_0}_{\cC^t}
			\right)^{\frac{1}{q_0}} 
                        \\
			&\le 
			\bE\left(
			\|b_T-b_{T,N}\|^{q_0}_{B_{q_0}^{t_0}}
			\right)^{\frac{1}{q_0}} \\
			&\le 
			\left(\sum_{j=N+1}^\infty 
			2^{jq_0(t_0 - \frac{d}{q_0} -s + \frac{d}{p}+\frac{\gg}{q_0})}
			\right)^{\frac{1}{q_0}} \\
			&=
			2^{N(t -s + \frac{d}{p}+\frac{\gg}{q_0})}
			\left(\sum_{j=1}^{\infty} 
			2^{jq_0(t - s + \frac{d}{p} + \frac{\gg}{q_0})}
			\right)^{\frac{1}{q_0}}. 		 
		\end{split}
	\end{equation}
	Now let $t<s-\frac{d}{p}$ and 
        $\gg\in(0,d]$ in~\eqref{eq:truncsum}, and choose 
        $q_0=\max\left(\gg N, 2\gg (s-\frac{d}{p}-t)^{-1}, q\right)$
        (for sufficiently large, given $N$) to obtain that
	\begin{equation}\label{eq:truncationest1}
		\begin{split}
			\|b_{T,\cD}-b_{T,N}|_\cD \|_{L^q(\gO; \cC^t(\ol\cD))}
			&\le 2^{N(t -s + \frac{d}{p}) + 1}
			\left(\sum_{j=1}^{\infty} 
			2^{jq_0(t - s + \frac{d}{p})\frac{1}{2}}
			\right)^{\frac{1}{q_0}} 		 \\
			&\le 2^{N(t -s + \frac{d}{p})}
			2\left(\sum_{j=1}^{\infty} 
			2^{j(t - s + \frac{d}{p})\frac{1}{2}}
			\right).
		\end{split}
	\end{equation}
	The final bound in~\eqref{eq:truncationest1} is independent of $q_0=q_0(N)$,
	which shows the first part of the claim for $\gg\in(0,d]$.
	For $\gg\in(-\infty,0]$, we use $q_0 = q$ in~\eqref{eq:truncsum} 
        to obtain for any $t\in(0,s-\frac{d}{p}-\frac{\gg}{q})$ that 
	\begin{equation}\label{eq:truncationest2}
		\| b_{T,\cD}-b_{T,N}|_\cD \|_{L^q(\gO; \cC^t(\ol\cD))}
		\le 2^{N(t - s + \frac{d}{p} + \frac{\gg}{q})}
		\left(\sum_{j=1}^{\infty} 
		2^{jq(t - s + \frac{d}{p} + \frac{\gg}{q})}
		\right)^{\frac{1}{q}}  		 
		\le \rev{C} 2^{N(t - s + \frac{d}{p} + \frac{\gg}{q})}.
	\end{equation}

	2.) \rev{
	To prove the second part, we observe that for any $t\in(0, s-\frac{d}{p}-\frac{\min(\gg,0)}{q})$ there holds
	\begin{align*}
		\|a-a_N\|_{L^q(\gO; \cC^t(\ol\cD))}
		&\le
		\|e^{b_{T,N}|_\cD}(e^{b_{T,\cD} - b_{T,N}|_\cD}-1)\|_{L^q(\gO; \cC^t(\ol\cD))}
		\\ &\le
		\|e^{b_{T,N}|_\cD}\|_{L^q(\gO; \cC^t(\ol\cD))}
		\|e^{b_{T,\cD} - b_{T,N}|_\cD}-1\|_{L^q(\gO; \cC^t(\ol\cD))},
	\end{align*}
	where the last equation follows by independence of $b_T - b_{T,N}$ and $b_{T,N}$.
	The first factor in this expression 
        is bounded by analogously to $\|e^{b_{T,\cD}}\|_{L^q(\gO; \cC^t(\ol\cD))}$ 
        by using the estimates ~\eqref{eq:exp-hoelder} (resp.~\eqref{eq:exp-hoelder1}),  
        Theorem~\ref{thm:randomtreereg} and Hölder's inequality
	\begin{equation} \label{eq:exp-N}
		\|e^{b_{T,N}|_\cD}\|_{L^q(\gO; \cC^t(\ol\cD))}
		\le 
		\|e^{b_{T,N}|_\cD}\|_{L^{q_1q}(\gO; L^\infty(\ol\cD))}
		\left(\|b_{T,N}|_\cD\|_{L^{q_2q}(\gO; \cC^t(\ol\cD))} + 1\right)<\infty,
	\end{equation}
	where $q_1,q_2>1$ are such that $\frac{1}{q_1}+\frac{1}{q_2}=1$ 
        and the bound holds uniformly in $N$.
	In addition, Taylor expansion yields 
	\begin{align*}
		\|e^{b_{T,\cD} - b_{T,N}|_\cD}-1\|_{\cC^t(\ol\cD)}
		&=
		\sup_{x,y\in\cD,\; x\neq y} 
		\frac{|e^{(b_{T,\cD} - b_{T,N}|_\cD)(x)}-e^{(b_{T,\cD} - b_{T,N}|_\cD)(y)}|}{\|x-y\|^r}
		+ 
		\|e^{b_{T,\cD} - b_{T,N}|_\cD}-1\|_{L^\infty(\cD)} \\
		&\le
		 \|e^{b_{T,\cD} - b_{T,N}|_\cD}\|_{L^\infty(\cD)}\sup_{x,y\in\cD,\; x\neq y} 
		 \frac{|(b_{T,\cD} - b_{T,N}|_\cD)(x)-(b_{T,\cD} - b_{T,N}|_\cD)(y)|}{\|x-y\|^r} \\
		 &\quad+ 
		 \|e^{b_{T,\cD} - b_{T,N}|_\cD}\|_{L^\infty(\cD)}
		 \|b_{T,\cD} - b_{T,N}|_\cD\|_{L^\infty(\cD)} \\
		 &\le 
		 \|e^{b_{T,\cD} - b_{T,N}|_\cD}\|_{L^\infty(\cD)}
		 \|b_{T,\cD} - b_{T,N}|_\cD\|_{\cC^t(\ol\cD)}.
	\end{align*}
	From the proof of the second part of Theorem~\ref{thm:solution-regularity} it follows that 
	$\|e^{b_{T,\cD} - b_{T,N}|_\cD}\|_{L^q(\gO; L^\infty(\cD))}<\infty$
	is bounded uniformly with respect to $N$ for all $q\ge 1$.
	Hölder's inequality for $p_1,p_2>1$ such that $\frac{1}{p_1}+\frac{1}{p_2}=1$
	thus shows together with the truncation error in~\eqref{eq:truncationest2} that 
	\begin{align*}
		\|a-a_N\|_{L^q(\gO; \cC^t(\ol\cD))}
		&\le
		\|e^{b_{T,N}|_\cD}(e^{b_{T,\cD} - b_{T,N}|_\cD}-1)\|_{L^q(\gO; \cC^t(\ol\cD))}
		\\ &\le 
		C \|e^{b_{T,\cD} - b_{T,N}|_\cD}\|_{L^{p_1q}(\gO; L^\infty(\cD))}
		\|b_{T,\cD} - b_{T,N}|_\cD\|_{L^{p_2q}(\gO; \cC^t(\ol\cD))}
		\\ &\le C 2^{N(t - s + \frac{d}{p} + \frac{\min(\gg,0)}{p_2q})}
	\end{align*}
	}
	The claim follows for any $\eps>0$ by choosing 
        $p_2>1$ so small that $\min(\gg, 0) \le p_2\min(\gg+\eps, 0)$.
\end{proof}

\begin{rem}\label{rem:truncation}
We emphasize that all estimates in Proposition~\ref{prop:truncation} are independent of $\cD\subset\bT^d$, 
as all uniform error bounds are derived with respect to $\bT^d$.
Proposition~\ref{prop:truncation} shows in particular that for any 
$q\ge 1$ and $t\in(0, s-\frac{d}{p}-\frac{\min(\gg,0)}{q})$ there is a $C>0$ such that 
for any $N\in\bN$ it holds
	\begin{align*}
		\|a-a_N\|_{L^q(\gO; L^\infty(\cD))}
		\le 
		C 2^{-Nt}.
	\end{align*}
	This estimate is essential to bound the truncation error $u-u_N$ of the approximated elliptic problem in~\eqref{eq:ellipticpdetrunc}, see Theorem~\ref{thm:u-truncation} below.
	In the borderline case $p=1$ with sufficiently small $\gk>0$ and $sp>d$, we still recover the slightly weaker estimates
	\begin{align}\label{eq:p1-truncation}
		\|a-a_N\|_{L^q(\gO; \cC^t(\ol\cD))}
		\le 
		C 2^{N(t-s+\frac{d}{p})},
		\quad 
		\|a-a_N\|_{L^q(\gO; L^\infty(\cD))}
		\le 
		C 2^{-tN}
	\end{align}
	for \emph{sufficiently small} $q\ge 1$ (depending on $\gk$) and $t\in(0,s-\frac{d}{p})$, \emph{independently} of $\gamma$. This may be seen from by letting $p_1\to1$ and $p_2\to \infty$ in the last part of the proof for Proposition~\ref{prop:truncation}.
\end{rem}

\begin{thm}\label{thm:u-truncation}
	Let $u$ be as in~\eqref{eq:ellipticpde} with $a=\exp\left(b_{T,\cD}\right)$ 
        and let $u_N$ be as in~\eqref{eq:ellipticpde} with $a_N=\exp\left(b_{T,N}|_\cD \right)$ 
        given by \eqref{eq:diffusioncoefftrunc}. 
	Furthermore, let $b_{T,\cD}$ be such that $p\in(1,\infty)$, $s>0$, $\gb=2^{\gg-d}\in[0,1]$, and $sp >d\ge d + \min(\gg,0)$.
	Then, for any $q\ge 1$ and $t\in(0, s-\frac{d}{p}-\frac{\min(\gg,0)}{q})$ 
        there is a $C>0$ such that for every $N\in\bN$ and it holds
	\begin{align*}
		\|u-u_N\|_{L^q(\gO; V)}\le C 2^{-Nt}.
	\end{align*}
\end{thm}

\begin{proof}
	For fixed $\go\in\gO$ and $N\in\bN$, we obtain by Proposition~\ref{prop:perturbation}
	\begin{equation*}
		\begin{split}
			\|u(\go)-u_N(\go)\|_V 
			\le 
			\frac{\|f\|_{V'}}{a_{-}(\go)a_{N,-}(\go)}\|a(\go)-a_N(\go)\|_{L^\infty(\cD)},
		\end{split}
	\end{equation*}
	where $a_{N,-}(\go):=\essinf_{x\in\cD} a_N(\go,x)$. 
	Taking expectations yields with Hölder's inequality 
	\begin{equation}\label{eq:utruncation}
		\begin{split}
			\|u-u_N\|_{L^q(\gO;V)} 
			\le 
			\|f\|_{V'}
		\|a_-^{-1}\|_{L^{q_1}(\gO)}
		\|a_{N,-}^{-1}\|_{L^{q_2}(\gO)}
		\|a-a_N\|_{L^{q_3}(\gO;L^\infty(\cD))},
		\end{split}
	\end{equation}
	where $q_1,q_2,q_3>1$ are such that $\frac{1}{q}=\sum_{i=1}^3\frac{1}{q_i}$ and $\|f\|_{V'}<\infty$.
	As in the proof of part 2.) in Theorem~\ref{thm:well-posed}, 
        we conclude for any $q_1\in[1,\infty)$ and $t\in(0,s-\frac{d}{p})$ with Theorem~\ref{thm:randomtreereg} 
        that
	\begin{equation*}
		\|a_-^{-1}\|_{L^{q_1}(\gO)} 
		\le \|\exp(\| b_{T,\cD }\|_{L^\infty(\cD)})\|_{L^{q_1}(\gO)}
		\le \|\exp(\|b_{T}\|_{\cC^t})\|_{L^{q_1}(\gO)}
		<\infty.
	\end{equation*}
	Similarly, it follows for all $q_2\in[1,\infty)$ that
	\begin{equation*}
		\|a_{N,-}^{-1}\|_{L^{q_2}(\gO)} 
		\le \|\exp(\|b_{T,N}\|_{\cC^t})\|_{L^{q_2}(\gO)}
		\le \|\exp(\|b_T\|_{\cC^t})\|_{L^{q_2}(\gO)}
		<\infty,
	\end{equation*}
	where we emphasize that the last bound is uniform with respect to $N$.
	Proposition~\ref{prop:truncation} and Remark~\ref{rem:truncation} show for $q_3\in[1,\infty)$ and $t\in(0,s-\frac{d}{p}-\frac{\min(\gg,0)}{q_3})$ that
	\begin{align*}
		\|a-a_N\|_{L^{q_3}(\gO; L^\infty(\cD))}
		\le 
		C 2^{-Nt}.
	\end{align*}
	This, together with~\eqref{eq:utruncation}, shows the claim, as $q_3>q$ may be chosen arbitrary close to $q$, and 
	\begin{equation*}
		\|a_-^{-1}\|_{L^{q_1}(\gO)}+\|a_{N,-}^{-1}\|_{L^{q_2}(\gO)}\le C <\infty
	\end{equation*}
	holds for all $q_1,q_2\in[1,\infty)$ with $C=C(q_1,q_2)>0$, and uniform with respect to $N$.
\end{proof}

\begin{rem}\label{rem:u-truncation}
	In view of Remark~\ref{rem:truncation}, we note that for $p=1$ with sufficiently small $\gk>0$ and $sp>d$ there holds the slightly weaker estimate 
	\begin{align*}
		\|u-u_N\|_{L^q(\gO; V)}
		\le 
		C 2^{-Nt}.
	\end{align*}
	for \emph{sufficiently small} $q\ge 1$ (depending on $\gk$) and $t\in(0,s-\frac{d}{p})$, \emph{independently} of $\gamma$. This may also be seen by letting $q_1,q_2\to \frac{1}{2q}$ and $q_3\to\infty$ in the proof of Theorem~\ref{thm:u-truncation}.
\end{rem}
\subsection{Finite element discretization}
\label{subsec:fem}
The solution $u_N:\gO\to V$ to Problem~\eqref{eq:ellipticpdetrunc} 
with truncated diffusion coefficient is still not fully tractable, 
as it takes values in the infinite-dimensional Hilbert space $V$. 
Thus, we consider Galerkin-finite element approximations of $u_N$ in a finite-dimensional subspace of $V$. 
Corollary~\ref{cor:truncated-regularity} provides the necessary regularity of $u_N$, 
independent of the truncation index $N$, therefore we fix $N\in\bN$ for the remainder of this section. 

We partition the convex, polytopal 
domain $\cD\subset \bT^{d}$, $d\in\{1,2,3\}$ by a 
sequence of simplices (intervals/triangles/tetrahedra) 
or parallelotopes (intervals/parallelograms/parallelepipeds), 
denoted by $(\cK_h)_{h\in \mfH}$. 
The refinement parameter $h>0$ takes values in a countable index set $\mfH \subset (0,\infty)$ 
and corresponds to the longest edge of a simplex/parallelotope $K\in\cK_h$. 
We impose the following assumptions on $(\cK_h)_{h\in \mfH}$ 
to obtain a sequence of "well-behaved" triangulations.
\begin{assumption}\label{ass:triangulation}
	The sequence $(\cK_h)_{h\in \mfH}$ satisfies:
	\begin{enumerate}
	\item \emph{Admissibility:} 
        For each $h\in \mfH$, $\cK_h$ consists of open, 
        non-empty simplices/parallelotopes $K$ 
        such that 
	\begin{itemize}
		\item $\ol \cD = \bigcup_{K\in\cK_h} \ol K$, 
		\item $K_1\cap K_2=\emptyset$ for any two $K_1, K_2\in \cK_h$ such that $K_1\neq K_2$, and 
		\item the intersection $\ol K_1\cap \ol K_2$ for $K_1\neq K_2$ is 
                      either empty, a common edge, a common vertex, 
                      or (in space dimension $d=3$) a common face of $K_1$ and $K_2$.
	\end{itemize}
	\item \emph{Shape-regularity:} 
         Let $\rho_{K,in}$ and $\rho_{K,out}$ denote the 
         radius of the largest in- and circumscribed circle, respectively, for a given $K\in\cK_h$.  
         There is a constant $\rho > 0$ such that 
	\begin{equation*}
		\rho : = \sup_{h\in \mfH} \sup_{K\in\cK_h} \frac{\rho_{K,out}}{\rho_{K,in}} < \infty.
	\end{equation*}
	\end{enumerate}
\end{assumption} 

Based on a given tesselation $\cK_h$, we define the space of piecewise (multi-)linear finite elements 
\begin{equation*}
	V_h:=
	\begin{cases}
		\{v\in V|\; \text{$v|_T$ is linear for all $K\in\cK_h$}\},
		&\quad\text{if $\cK_h$ consists of simplices}, \\
		\{v\in V|\; \text{$v|_T$ is $d$-linear for all $K\in\cK_h$}\},
		&\quad\text{if $\cK_h$ consists of parallelotopes}.
	\end{cases}
\end{equation*}
Clearly, $V_h\subset V$ is a finite-dimensional space and we define $n_h:=\dim(V_h)\in \bN$.
This yields for fixed $\go\in\gO$ the \emph{fully discrete problem} to find $u_{N,h}(\go)\in V_h$ such that for all $v_h\in V_h$
\begin{equation}\label{eq:ellipticpdediscrete}
	\int_\cD a_N(\go)\nabla u_{N,h}(\go)\cdot\nabla v_h dx = \dualpair{V'}{V}{f}{v_h}.
\end{equation}

\begin{thm}\label{thm:H1error}
	Let $(\cK_h)_{h\in \mfH}$ be a sequence of triangulations satisfying Assumption~\ref{ass:triangulation}, 
        and let $u_N$ and $u_{N,h}$ be the pathwise weak solutions to 
        \eqref{eq:ellipticpdetrunc} and \eqref{eq:ellipticpdediscrete}.
	Furthermore, let $N\in\bN$, $a_N$ be given as in~\eqref{eq:diffusioncoefftrunc} for $p\in[1,\infty)$ and $s>0$, 
        such that $sp>d$, and with $\gb=2^{\gg-d}\in[0,1]$.
	
	For any $f\in H$, sufficiently small $\gk>0$ in~\eqref{eq:p-exponential} and 
            any $r\in (0,s-\frac{d}{p})\cap (0,1]$, 
        there are constants $\ol q\in(1,\infty)$ and $C>0$ such that for any $N\in\bN$ and $h\in \mfH$ 
        there holds 
	\begin{equation*}
		\|u_N-u_{N,h}\|_{L^q(\gO; V)}\le C h^r 
		\quad
		\begin{cases}
			&\text{for $q\in[1,\ol q)$ if $p=1$, and} \\
			&\text{for any $q\in [1,\infty)$ if $p>1$}.
		\end{cases}
	\end{equation*}
\end{thm}

\begin{proof}
	We recall that $a_{N,-}(\go):=\essinf_{x\in\cD} a_{N,-}(\go)>0$ and obtain by Cea's Lemma
	\begin{equation}\label{eq:cea}
		\|u_N(\go)-u_{h,N}(\go)\|_V
                \le 
                 \frac{\|a_N(\go)\|_{L^\infty(\cD)}}{a_{N,-}(\go)} \|f\|_{V'} 
                 \inf_{v_h\in V_h} \|u_N(\go) - v_h\|_{V}.
	\end{equation}
	
	Now first suppose that $p>1$. Since $f\in H$, it holds by 
        Corollary~\ref{cor:truncated-regularity} for any $q\ge 1$ 
        that $u_N\in L^q(\gO; W^r)$ for $r\in(0, s-\frac{d}{p})\cap (0,1]$.
	For $0<s-\frac{d}{p}\le 1$, we have $r\in(0,s-\frac{d}{p})$, and 
	Lemma~\ref{lem:interp-sobolev}, 
    shows $u_N\in L^q(\gO; H^{1+r_0}(\cD))$ for any $r_0\in(0,r)$.
	It hence follows for $r_0\in (0,r)$ that 
	\begin{equation}\label{eq:fem-error}
		\inf_{v_h\in V_h} \|u_N(\go) - v_h\|_{V}\le C \|u_N(\go)\|_{H^{1+ r_0}(\cD)} h^{r_0}.
	\end{equation}
	This is a standard result for first order Lagrangian FEM, 
	see, e.g., \cite[Theorems 8.62/8.69]{H17} or \cite[Theorem 4.4.20]{BS08}. 
	The constant $C>0$ in~\eqref{eq:fem-error} depends on the shape-regularity parameter 
	$\rho$ and on $\cD$, but is independent of $u_N$ and $h$.
	Combining \eqref{eq:cea} and \eqref{eq:fem-error} 
     shows with Hölder's inequality
	\begin{equation}\label{eq:fem-estimate}
		\begin{split}
			\|u_N(\go)-u_{h,N}(\go)\|_{L^q(\gO; V)}
			&\le
			C  \|f\|_{V'} \|a_N\|_{L^{3q}(\gO; L^\infty(\cD))}
			\|a_{N,-}^{-1}\|_{L^{3q}(\gO)}
			\|u_N\|_{L^{3q}(\gO; H^{1+r_0}(\cD))}
			h^{r_0} 
                        \\
			&\le
			C  \|a_N\|_{L^{3q}(\gO; L^\infty(\cD))}^2
			\|u_N\|_{L^{3q}(\gO; H^{1+ r_0 }(\cD))}
			h^{r_0} \\
			&\le 
			C  \|a_N\|_{L^{3q}(\gO; L^\infty(\cD))}^2
			\|u_N\|_{L^{3q}(\gO; W^r)}
			h^{r_0} 
                        \\
			&\le C h^{r_0}.
		\end{split}
	\end{equation}
	We have used that $a_{N,-}$ and $\|a_N\|_{L^\infty(\cD)}$ are equal in distribution for the second estimate,
	and Lemma~\ref{lem:interp-sobolev} in the third line.
	The last step follows for any $q\in[1,\infty)$ by Corollary~\ref{cor:truncated-regularity} and Proposition~\ref{prop:truncation} since $p>1$.
	Moreover, as a further consequence of Corollary~\ref{cor:truncated-regularity} and Proposition~\ref{prop:truncation}, the constant $C>0$ in the final estimate in~\eqref{eq:fem-estimate} bounded independently of $N$ and $h$. 
	Since $0<s-\frac{d}{p}\le 1$, we may choose $r_0<r<s-\frac{d}{p}$ arbitrary close to $s-\frac{d}{p}$. 
	
	On the other hand, if $s-\frac{d}{p}>1$ and $r=1$, Lemma~\ref{lem:interp-sobolev} implies that $u_N\in L^q(\gO; H^{2}(\cD))$.
	Estimates~\eqref{eq:fem-error} and~\eqref{eq:fem-estimate} then hold for $r_0=r=1$, which proves the claim in case that $p>1$.

	For $p=1$ and given $q\ge1$, 
        we need to assume in addition that $\gk>0$ be sufficiently small 
        such that Corollary~\ref{cor:truncated-regularity} and~\eqref{eq:p1-truncation} 
        in Remark~\ref{rem:truncation} hold with $q$ replaced $3q$. 
        In this case, the claim for $p=1$ follows analogously as for $p>1$.
\end{proof}

\rev{In the proof of Theorem~\ref{thm:H1error}, we obtain exponential moments of power $3q$ by Hölder's inequality. For the case $p=1$, we therefore need approximately that $\gk < \frac{1}{3q}$ (up to summation constants) to counter-balance this exponent and obtain $a_N \in L^{3q}(\gO; L^\infty(\cD))$ and $u_N \in L^{3q}(\gO; W^r)$.}

\begin{thm}\label{thm:L2error}
	Let the assumptions of Theorem~\ref{thm:H1error} hold.
	For any $f\in H$, sufficiently small $\gk>0$ in~\eqref{eq:p-exponential} 
        and 
        for any $r\in (0,s-\frac{d}{p})\cap (0,1]$, 
        there are constants $\ol q\in(1,\infty)$ and $C>0$ 
        such that for any $N\in\bN$ and $h\in\mfH$ there holds 
	\begin{equation*}
		\|u_N-u_{N,h}\|_{L^q(\gO; H)}\le C h^{2r} 
		\quad
		\begin{cases}
			&\text{for $q\in[1,\ol q)$ if $p=1$, and} \\
			&\text{for any $q\in [1,\infty)$ if $p>1$}.
		\end{cases}
	\end{equation*}
\end{thm}

\begin{proof}
	The proof uses the well-known \emph{Aubin-Nitsche duality} argument. Let $e_{N,h}:=u_N-u_{N,h}$ and consider for fixed $\go\in\gO$ the dual problem to find $\varphi(\go)\in V$ such that for all $v\in V$ it holds
	\begin{equation}\label{eq:dualproblem}
		\int_\cD a_N(\go)\nabla \varphi(\go)\cdot\nabla v dx = \dualpair{V'}{V}{e_{N,h}(\go)}{v}.
	\end{equation}
	We need to investigate the regularity and integrability of $\varphi$ as a first step. Lemma~\ref{lem:ellipticreg} shows that
	\begin{equation}\label{eq:dualregularity1}
		\|\varphi(\go)\|_{W^r}
		\le 
		\frac{C}{a_-(\go)}
		\left(1+\left(\frac{\|a(\go)\|_{\rC^r(\ol\cD)}}{a_-(\go)}\right)^{1/r}\right)
		\|e_{N,h}(\go)\|_H.
	\end{equation}
	Let $t\in(0, s-\frac{d}{p})$ be fixed. 
        We integrate both sides of~\eqref{eq:dualregularity1} and use Hölder's inequality 
        as in the third part of Theorem~\ref{thm:solution-regularity} \rev{(cf. Inequality~\eqref{eq:sobolevreg2})} 
        to obtain for $q_0\ge 1$ and 
        \rev{
        $q_1,\dots,q_3\in[1,\infty)$ such that $1=\sum_{i=1}^3\frac{1}{q_i}$ 
	\begin{equation}\label{eq:dualregularity2}
		\begin{split}
			\|\varphi\|_{L^{q_0}(\gO; W^r)}^{q_0}	\le
			&C\,\bE\left[
			\exp\left(q_1\left(q_0+\frac{2q_0}{r}\right)\|b_T\|_{\rC^r}\right)
			\right]^{\frac{1}{q_1}}
			\bE\left[
			\|b_T\|_{\rC^r}^{\frac{q_2q_0}{r}}
			\right]^{\frac{1}{q_2}}
			\bE\left[\|e_{N,h}\|_{H}^{q_0q_3}\right]^{1/q_3}
		\end{split}
	\end{equation}
	Note that we have again assumed that $\|b_{T,\cD}\|_{\rC^r(\ol\cD)}\ge 1$ without loss of generality to derive \eqref{eq:dualregularity2}.
	}

	By Theorems~\ref{thm:randomtreereg},~\ref{thm:H1error} and Proposition~\ref{prop:truncation}, we now conlude that the right hand side in~\eqref{eq:dualregularity2} is finite and bounded uniformly in $N$ for any $q_0\ge 1$ if $p>1$, as the Hölder conjugates $q_1,\dots,q_{\rev 3}\in[1,\infty)$ may be arbitrary large. 
	For $p=1$,  we further need that $\gk>0$ in~\eqref{eq:p-exponential} is sufficiently small, 
        so that $\eps_p>q_0\max(q_1\rev{(1+\frac{1}{r})}, \frac{q_2}{r})$ in Theorem~\ref{thm:randomtreereg} 
        and that $\ol q\ge q_0q_{\rev 3}$ in Theorem~\ref{thm:H1error}.
	Given that $\gk>0$ is sufficiently small, there is for any $p\ge 1$ a $q_0\ge 1$ such that $\varphi\in L^{q_0}(\gO; W^r)$.
	\vskip 2pt
	For the next step, we combine Equations \eqref{eq:ellipticpdetrunc} and \eqref{eq:ellipticpdediscrete} to show the Galerkin orthogonality
	\begin{equation}\label{eq:galerkinorth}
		\int_\cD a_N(\go)\nabla e_{N,h}(\go)\cdot\nabla v_h dx = 0,
		\quad 
		v_h\in V_h.
	\end{equation}
	Let $P_h:V\to V_h$ denote the $V$-orthogonal projection onto $V_h$. 
	Testing with $v=e_{N,h}(\go)\in V$ in \eqref{eq:dualproblem} 
         then shows together with $v_h=P_h\varphi(\go)$ in \eqref{eq:galerkinorth} that
	\begin{equation}\label{eq:dualestimate}
		\begin{split}
			\|e_{N,h}(\go)\|_H^2
			&= \int_\cD a_N(\go)\nabla \varphi(\go)\cdot \nabla e_{N,h}(\go) dx \\
			&\le \|a_N(\go)\|_{L^\infty(\cD)} \|e_{N,h}(\go)\|_V \|(I-P_h)\varphi(\go)\|_V.
		\end{split}		
	\end{equation}
	Estimate~\eqref{eq:dualestimate} then yields for $q\in[1,\infty)$ with Hölder's inequality
	\begin{equation*}\label{eq:L2error}
		\begin{split}
			\|e_{N,h}\|_{L^q(\gO; H)}
			\le \|a_N\|_{L^{\frac{3q}{2}}(\gO; L^\infty(\cD))} 
			\|e_{N,h}\|_{L^{\frac{3q}{2}}(\gO; V)} 
			\|(I-P_h)\varphi\|_{L^{\frac{3q}{2}}(\gO; V)}.
		\end{split}		
	\end{equation*}
	\vskip 2pt
	First, suppose again that $p>1$, where $\varphi\in L^{q_0}(\gO; W^r)$ holds for any $q_0\ge 1$.
	Proposition~\ref{prop:truncation} and Theorem~\ref{thm:H1error} yield
	\begin{equation*}
		\begin{split}
			\|e_{N,h}\|_{L^q(\gO; H)}
			\le C h^{r}\|(I-P_h)\varphi\|_{L^{\frac{3q}{2}}(\gO; V)},
		\end{split}		
	\end{equation*}
	where $C>0$ is independent of $N$ and $h$.
	Lemma~\ref{lem:interp-sobolev}, $\varphi\in L^{q_0}(\gO; W^r)$, and \eqref{eq:fem-error} 
        further show that
	\begin{equation}\label{eq:L2error2}
		\begin{split}
			\|e_{N,h}\|_{L^q(\gO; H)}
			\le C h^{r+r_0},
			\quad
			r_0\in(0,r)\cup\{\floor r\}.
		\end{split}		
	\end{equation} 
	The claim follows as in Theorem~\ref{thm:H1error}, 
        since $r=r_0=1$ if $s-\frac{d}{p}>1$, and $r_0<r<s-\frac{d}{p}$ 
        may be arbitrary close to $s-\frac{d}{p}$ otherwise.  
	\vskip 2pt
	For $p=1$ and given $q\ge 1$ on the other hand, we need to assume that $\gk>0$ is sufficiently small
	so that $\varphi\in L^{q_0}(\gO; W^r(\cD))$ for $q_0=\frac{3q}{2}$, 
        and that \eqref{eq:p1-truncation} and Theorem~\ref{thm:H1error} hold with $q$ replaced by $\frac{3q}{2}$. 
        The claim then follows as for $p>1$ from~\eqref{eq:L2error}.
\end{proof}

Bounds on the overall approximation errors with respect to $V$ and $H$ now follow 
as an immediate consequence of Theorems~\ref{thm:u-truncation},~\ref{thm:H1error},~\ref{thm:L2error} 
and Remark~\ref{rem:u-truncation}.

\begin{cor}\label{cor:fullerror}
	Let the assumptions of Theorem~\ref{thm:H1error} hold, 
        let $f\in H$, let $t\in(0,s-\frac{d}{p})$ and 
        assume given $r\in (0,s-\frac{d}{p})\cap (0,1]$.
        Then there holds:
	\begin{enumerate}[1.)]
		\item For $p=1$ and sufficiently small $\gk>0$ in~\eqref{eq:p-exponential}, there are constants $\ol q=\ol q(\gk)\in(1,\infty)$ and $C>0$ such that for any $q\in[1,\ol q)$, $N\in\bN$ and $h\in\mfH$ there holds 
		\begin{align*}
			\|u-u_{N,h}\|_{L^q(\gO; V)}&\le C(2^{-tN}+h^{r}), \\
			\|u-u_{N,h}\|_{L^q(\gO; H)}&\le C(2^{-tN}+h^{2r}). 
		\end{align*}
		
		\item For $p\in(1,\infty)$ and any $q\in[1,\infty)$ there is a constant $C>0$ such that for any $N\in\bN$ and $h\in\mfH$ there holds 
		\begin{align*}
			\|u-u_{N,h}\|_{L^q(\gO; V)}&\le C(2^{N(-t+\frac{\min(\gg,0)}{q})}+h^{r}), \\
			\|u-u_{N,h}\|_{L^q(\gO; H)}&\le C(2^{N(-t+\frac{\min(\gg,0)}{q})}+h^{2r}). 
		\end{align*}
	\end{enumerate}
\end{cor}

\section{Multilevel Monte Carlo Estimation}\label{sec:mlmc}

We consider \MC\, estimation of $\bE(\Psi(u))$ for a given functional $\Psi$ and $u$ as solution to~\eqref{eq:ellipticpdeweak} 
with Besov random tree coefficient $a$. 
We replace $u$ by a tractable approximation $u_{N,h}$ to evaluate $\Psi(u_{N,h})\approx\Psi(u)$ 
and bound the overall error consisting of the pathwise discretization from 
Section~\ref{sec:approximation} and the statistical error of the \MC\, approximation.
\begin{assumption}\label{ass:functional}~
	\begin{enumerate}[1.)]
		\item 
		Let $\gt\in [0,1]$, let $\Psi:H^\gt(\cD)\to \bR$ be 
                Fréchet-differentiable on $H^\gt(\cD)$
                and denote by 
		\begin{equation*}
			\Psi':H^\gt(\cD)\to \cL(H^\gt(\cD); \bR)=(H^{\gt}(\cD))'
		\end{equation*}
		the Fréchet-derivative of $\Psi$. 
		There are constants $C>0$, $\rho_1,\rho_2\ge 0$ such that for all $v\in H^\gt(\cD)$  
		\begin{equation}\label{eq:functionalgrowth}
			|\Psi(v)|\le C(1+\|v\|_{H^\gt(\cD)}^{\rho_1}), \quad
			\|\Psi'(v)\|_{\cL(H^\gt(\cD); \bR)}\le C(1+\|v\|_{H^\gt(\cD)}^{\rho_2}).
		\end{equation}
		\item \label{item:Lqregularity} For $q:=2\max(\rho_1,\rho_2+1)$, there holds $u\in L^{q}(\gO;V)$.
		\item $(\cK_h)_{h\in\mfH}$ is a collection of triangulations
                      satisfying Assumption~\ref{ass:triangulation}. 
		\item \label{item:error}
                There are constants $t>0$, $r\in(0,1]$ and $C>0$ 
                such that for $q=2\max(\rho_1,\rho_2+1)$ and any $N\in\bN$ and $h\in\mfH$ 
                it holds 
		\begin{align*}
			\|u-u_{N,h}\|_{L^q(\gO; V)}\le C(2^{-tN}+h^{r}), & \quad 
			\|u-u_{N,h}\|_{L^q(\gO; H)}\le C(2^{-tN}+h^{2r}). 
		\end{align*}
	\end{enumerate}
\end{assumption}

\begin{rem}
	Assumption~\ref{ass:functional} is natural, and includes in particular bounded linear functions $\Psi$, where $\rho_1=1$ and $\rho_2=0$. Item~\ref{item:Lqregularity} follows by Theorem~\ref{thm:solution-regularity} and Item~\ref{item:error} by Corollary~\ref{cor:fullerror}, with no further restrictions whenever $p>1$. 
	Only in case that $p=1$, $\gk>0$ needs to be sufficiently small to ensure that 
        all bounds hold for $q=2\max(\rho_1,\rho_2 + 1)\ge 2$.
\end{rem}
\subsection{Singlelevel \MC}\label{subsec:slmc}
We use Monte Carlo (MC) methods to approximate $\bE(\Psi(u))$ for a given functional $\Psi$.
To this end, we first consider the standard MC estimator for (general) real-valued random variables. 
\begin{defi}
	Let $Y:\gO\to\bR$ be a random variable and let $(Y^{(i)},i\in\bN)$ be a sequence of i.i.d. copies of $Y$. For $M\in\bN$ we define \emph{Monte Carlo estimator} $E_M(Y):\gO\to\bR$ as
	\begin{equation}\label{eq:mcestimator}
		E_M(Y):=\frac{1}{M}\sum_{i=1}^M Y^{(i)}.
	\end{equation}
\end{defi}
As we are not able to sample directly from the distribution of $u$, 
we rely on i.i.d. copies $(u^{(i)}_{N,h}, i\in\bN)$  
of the pathwise approximation $u_{N,h}$ from Section~\ref{sec:approximation}. 
Thereby, 
in addition to the statistical MC error of order $\cO(M^{-1/2})$, 
we introduce a sampling bias that depends on $N$ and $h$. 
\begin{thm}\label{thm:mcerror}
Let $M\in\bN$, let $E_M(\Psi(u_{N,h}))$ be the MC estimator as in~\eqref{eq:mcestimator}, 
and let Assumption~\ref{ass:functional} hold.
Then, there is a constant $C>0$, such that for any $M, N\in\bN$ and $h\in\mfH$ it holds 
\begin{equation*}
	\|\bE(\Psi(u))-E_M(\Psi(u_{N,h}))\|_{L^2(\gO)} \le C\left(M^{-1/2} + 2^{-tN} + h^{(2-\gt)r} \right)
\end{equation*}
	
\end{thm}

\begin{proof}
	We split the overall error via 
	\begin{align*}
		\|\bE(\Psi(u))-E_M(\Psi(u_{N,h}))\|_{L^2(\gO)} 
		&\le 
		\|\bE(\Psi(u))-E_M(\Psi(u))\|_{L^2(\gO)} 
		\\
		&\quad+\|E_M(\Psi(u))-E_M(\Psi(u_{N,h}))\|_{L^2(\gO)} 
		\\
		&:= I+II.
	\end{align*}
	To bound $I$, we use independence of $\Psi(u)^{(i)}$ and $\Psi(u)^{(j)}$ for $i\neq j$ to see that 
	\begin{align*}
		I^2 
		&= 
		\bE\left(\left(\bE(\Psi(u)) - \frac{1}{M}\sum_{i=1}^M \Psi(u)^{(i)}\right)^2\right) \\
		&=
		\bE(\Psi(u))^2
		-\frac{2}{M} \sum_{i=1}^M\bE(\Psi(u))^2 
		+\frac{1}{M^2}\sum_{i,j=1}^M \bE\left(\Psi(u)^{(i)}\Psi(u)^{(j)}\right)\\
		&
		=\frac{\var(\Psi(u))}{M}.
	\end{align*}	
	Assumption~\ref{ass:functional} further shows that 	 
	\begin{equation*}
		I\le \frac{\|\Psi(u)\|_{L^2(\gO)}}{M^{1/2}}  
		\le C\frac{1+\|u\|_{L^{2\rho_1}(\gO; H^\theta(\cD))}}{M^{1/2}} 
		\le C\frac{1+\|u\|_{L^{2\rho_1}(\gO; V)}}{M^{1/2}} 
		\le CM^{-1/2}, 
	\end{equation*}
	where we have used that $\gt \le 1$ and $u\in L^{2\rho_1}(\gO; V)$.

	To bound $II$, we use Equation~\eqref{eq:functionalgrowth} and derive the pathwise estimate
	\begin{equation}\label{eq:Psierror}
		\begin{split}
			|\Psi(u)-\Psi(u_{N,h})|
			&= 
			\left|\int_0^1 \Psi'(u+z(u_{N,h}-u))(u-u_{N,h})dz\right| \\
			&\le
			\int_0^1 \|\Psi'(u+z(u_{N,h}-u))\|_{\cL(H^\gt(\cD);\bR)}\|u-u_{N,h}\|_{H^\gt(\cD)}dz \\
			&\le 
			C\left(1+\|u\|_{H^\gt(\cD)}^{\rho_2}+\|u-u_{N,h}\|_{H^\gt(\cD)}^{\rho_2}\right)\|u-u_{N,h}\|_{H^\gt(\cD)} \\
			&\le 
			C\left(1+\|u\|_{H^\gt(\cD)}^{\rho_2}+\|u_{N,h}\|_{H^\gt(\cD)}^{\rho_2}\right)
			\|u-u_{N,h}\|_{H^\gt(\cD)}.
		\end{split}
	\end{equation}
	By Assumption~\ref{ass:functional}, 
        there is a $C>0$ such that for every $N$ and every $0<h\leq 1$ it holds
	\begin{equation}\label{eq:Lqbound1}
		\|u_{N,h}\|_{L^{2(\rho_2+1)}(\gO;H^\gt(\cD))}
		\le C\|u\|_{L^{2(\rho_2+1)}(\gO;H^\gt(\cD))}
		\le C \|u\|_{L^{2(\rho_2+1)}(\gO;V)}
		<\infty.
	\end{equation}
	Furthermore, as $\gt\in[0,1]$, 
        we have by the Gagliardo-Nirenberg interpolation inequality 
	\begin{equation}\label{eq:Lqbound2}
		\|u-u_{N,h}\|_{L^{2(\rho_2+1)}(\gO;H^\gt(\cD))}
		\le 
		\|u-u_{N,h}\|_{L^{2(\rho_2+1)}(\gO;H)}^{1-\gt}
		\|u-u_{N,h}\|_{L^{2(\rho_2+1)}(\gO;V)}^{\gt}
		\le 
		C (2^{-tN}+h^{(2-\gt)r}).
	\end{equation} 
	Thus, Hölder's inequality with conjugate exponents 
        $q_1=\frac{\rho_2+1}{\rho_2}$, $q_2=\rho_2+1$ (and $q_1=\infty$, $q_2=1$ for $\rho_2=0$) 
        shows that
	\begin{equation}\label{eq:mcdiscretebias}
		\begin{split}
			II
			&\le 
			\|\Psi(u)-\Psi(u_{N,h})\|_{L^2(\gO)} \\
			&\stackrel{\eqref{eq:Psierror}}{\le} 		
			C(1+\|u\|_{L^{q_12\rho_2}(\gO;H^\gt(\cD))}+\|u_{N,h}\|_{L^{q_12\rho_2}(\gO;H^\gt(\cD))})
			\|u-u_{N,h}\|_{L^{q_22}(\gO;H^\gt(\cD))} \\
			&\stackrel{\eqref{eq:Lqbound1}}{\le} 		
			C(1+\|u\|_{L^{2(\rho_2+1)}(\gO;V)})
			\|u-u_{N,h}\|_{L^{2(\rho_2+1)}(\gO;H^\gt(\cD))}\\
			&\stackrel{\eqref{eq:Lqbound2}}{\le} 
			C (2^{-tN}+h^{(2-\gt)r}).
		\end{split}
	\end{equation} 
\end{proof}
The error contributions in Theorem~\ref{thm:mcerror} are balanced by choosing
\begin{equation}\label{eq:balancederror}
	M\approx 2^{2tN} \approx h^{-2(2-\gt)r}.
\end{equation}
With this choice, achieving the target accuracy $\|\bE(\Psi(u))-E_M(\Psi(u_{N,h}))\|_{L^2(\gO)} =\cO(\eps)$
requires sampling $M=\cO(\eps^{-2})$ high-fidelity approximations $u_{N,h}$ 
with $N=\cO(\frac{\log(\eps)}{t})$ scales and mesh refinement $h=\cO(\eps^{\frac{1}{(2-\gt)r}})$.
This is computationally challenging in dimension $d\ge 2$ and for low-regularity problems, i.e., when $t,r>0$ are small.
To alleviate the computational burden, we propose a \MLMC\, extension of the estimator $E_M$ in the next subsection.
\subsection{Multilevel \MC}\label{subsec:mlmc}
The \MLMC\,(MLMC) algorithm was invented by Heinrich~\cite{heinrich2001multilevel} to compute parametric integrals, then rediscovered and popularized by Giles~\cite{giles2008multilevel, giles2015multilevel}, and has since then found various applications in uncertainty quantification and beyond.

To apply this methodology to our model problem we fix a maximum refinement level $L\in\bN$ 
and consider a \emph{sequence of approximated solutions} $u_{N_\ell, h_\ell}$ with $(N_\ell, h_\ell)\in\bN\times \mfH$ for $\ell\in\{1,\dots,L\}$.
We assume that $N_1<\dots<N_L$ and $h_1>\dots>h_L$, 
so that the error $u-u_{N_\ell,h_\ell}$ decreases with respect to the level $\ell$. 
For notational convenience, we define $\Psi_\ell:=\Psi(u_{N_\ell,h_\ell})$ as the 
approximation of the quantity of interest $\Psi(u)$ on level $\ell$, and set $\Psi_0:=0$. 
The basic idea of the MLMC method for estimating $\bE(\Psi(u))$ is to exploit the telescopic expansion 
\begin{equation}
	\bE(\Psi(u))\approx\bE(\Psi_L)=\bE(\Psi_L)-\bE(\Psi_0)=\sum_{\ell=1}^L \bE(\Psi_\ell-\Psi_{\ell-1})
\end{equation}
of the high-fidelity approximation $\Psi_L$. 
On each level $\ell$, the correction $\bE(\Psi_\ell-\Psi_{\ell-1})$ is estimated by (standard) MC estimator with $M_\ell$ samples. 
This yields the \emph{\MLMC\, estimator} 
\begin{equation}\label{eq:mlmcestimator}
	E^L(\Psi_L):=\sum_{\ell=1}^L E_{M_\ell}(\Psi_\ell-\Psi_{\ell-1}),
\end{equation} 
with level-dependent numbers of samples $M_1,\dots,M_L\in\bN$. 
We assume that the estimators $E_{M_\ell}(\Psi_\ell-\Psi_{\ell-1})$ are \emph{jointly independent across the levels} $\ell=1,\dots,L$.

Provided that $\var(\Psi_\ell-\Psi_{\ell-1})$ decays sufficiently fast in $\ell$, 
we choose $M_1>\dots>M_L$ such that the majority of samples are generated cheaply on low levels $\ell$, 
while only a few expensive samples for large $\ell$ are necessary. 
This entails massive computational savings compared to a \emph{singlelevel \MC\,}(SLMC) estimator as in~\eqref{eq:mcestimator}, 
that requires a large number of expensive samples on level $L$, and does not exploit the level hierarchy whatsoever. 
The computational gain of the MLMC method is precisely quantified under certain assumptions in 
Giles' complexity theorem (\cite[Theorem 3.1]{giles2008multilevel}). 
Given some $\eps>0$, 
Giles derives the optimal number of refinement levels $L$ and associated numbers of samples 
$M_1,\dots, M_L$ that guarantee $\|\bE(\Psi(u))-E^L(\Psi(u))\|_{L^2(\gO)}\le \eps$. 
The latter requires exact knowledge of all constants in Assumption~\ref{ass:functional}, and, furthermore, 
exact knowledge of the cost for sampling one instance of $\Psi_\ell$.
As this is not feasible a-priori, 
we choose a slightly different approach to determine the MLMC parameters.
We retain the optimal order of complexity as in \cite[Theorem 3.1]{giles2008multilevel}. 
\begin{assumption}\label{ass:refinement}
	Let $(\cK_h)_{h\in\mfH}$ be a sequence of triangulations
satisfying Assumption~\ref{ass:triangulation}, and 
assume that $h_\ell\in\mfH$ for any $\ell\in\bN$. 
Furthermore, in view of the multilevel convergence analysis, 
we assume that there are $0<\ul c_\cK\le \ol c_\cK<1$ and $h_0>0$ 
such that
	\begin{equation}
		\label{eq:refinement-bounds}
		\ul c_\cK^\ell h_0\le h_\ell \le \ol c_\cK^\ell h_0,\quad \ell\in\bN. 
	\end{equation} 
One sample of
$\Psi_\ell=\Psi(u_{N_\ell,h_\ell})$ with $u_{N_\ell,h_\ell}\in V_{h_\ell}$ 
and $n_\ell:=\dim(V_{h_\ell})=\cO(h_\ell^{-d})$ is realized in $\cO(n_\ell)$ work and memory. 
\end{assumption}

\begin{rem}
Assumption~\ref{ass:refinement} is natural and holds, for instance, with 
$\ul c_\cK, \ol c_\cK \approx \frac{1}{2}$ if the mesh $\cK_{h_\ell}$ 
is obtained from $\cK_{h_{\ell-1}}$ by bisection of the longest edge of each $K\in\cK_{h_{\ell-1}}$.
We remark that in general, it may be hard to achieve $\ul c_\cK = \ol c_\cK$, 
which is why we imposed an upper and lower bound in~\eqref{eq:refinement-bounds}.
Simulating $\Psi_\ell$ requires $\cO(n_\ell)=\cO(h_\ell^{-d})$ 
floating point operations per sample when using multilevel solvers for continuous piecewise linear or multi-linear elements.
We also refer to Lemma~\ref{lem:cost} in Appendix~\ref{appendix:coefficient}, 
where we show that the expected cost of sampling $b_{N,T}$ on the associated grid is of order 
$\cO(h_\ell^{-d})$ if $\gb<1$.
\end{rem}

\begin{thm}\label{thm:mlmc}
	Let Assumptions~\ref{ass:functional} and~\ref{ass:refinement} hold.
	For $t, r$ and $\gt$ as in Assumption~\ref{ass:functional}, let \rev{$\eps\in(0, h_0^{(2-\gt)r})$ and}
        select the MLMC parameters in~\eqref{eq:mlmcestimator} for $\ell\in\{1,\dots,L\}$ as
	\begin{equation}\label{eq:mlmcparameters}
		L:=\left\lceil 
		\frac{\log(\eps)}{(2-\gt)r\log(\ul c_\cK)} -\frac{\log(h_0)}{\log(\ul c_\cK)}
		\right\rceil,\quad 
		M_\ell:=\left\lceil \left(\frac{h_\ell}{h_L}\right)^{2(2-\gt)r}w_\ell \right\rceil,\quad 
		N_\ell:=\left\lceil - \frac{\log(h_\ell)(2-\gt)r}{\log(2)t} \right\rceil.
	\end{equation}
	For given $\rev{L\in\bN}$, choose the weights $w_\ell>0$ to determine $M_\ell$ 
        such that $\sum_{\ell=1}^L w_\ell^{-1} \le C_w < \infty$, for sufficiently large, fixed $C_w > 0$ independent of $L$.
	 
	Then, there is a $C>0$,
        such that for any \rev{$\eps\in(0, h_0^{(2-\gt)r})$} it holds
	\begin{equation*}
		\|\bE(\Psi(u))-E^L(\Psi_L)\|_{L^2(\gO)} \le C\eps.
	\end{equation*}
\end{thm}

\begin{proof}
	We use the error splitting 
	\begin{align*}
		\|\bE(\Psi(u))-E^L(\Psi_L)\|_{L^2(\gO)} 
		\le 
		\|\bE(\Psi(u))-\bE(\Psi_L))\|_{L^2(\gO)} 
		+\|\bE(\Psi_L)-E^L(\Psi_L)\|_{L^2(\gO)}.
	\end{align*}
	We obtain in the same fashion as for the term $II$ in the proof of Theorem~\ref{thm:mcerror} that
	\begin{align*}
		\|\bE(\Psi(u))-\bE(\Psi_L)\|_{L^2(\gO)} 
		\le
		\|\Psi(u)-\Psi(u_{N_L,h_L})\|_{L^2(\gO)}
		\le C (2^{-tN_L}+h_L^{(2-\gt)r})
		\le C h_L^{(2-\gt)r}, 
	\end{align*}
	where we have used $2^{-tN_L}\le h_L^{(2-\gt)r}$ by~\eqref{eq:mlmcparameters} in the last step.
	To bound the second term, we expand $\bE(\Psi_L)$ in a telescopic sum to obtain with~\eqref{eq:mlmcestimator} 
	\begin{align*}
		\|\bE(\Psi_L)-E^L(\Psi_L)\|_{L^2(\gO)}^2
		&=
		\sum_{\ell=1}^L 
		\|\bE(\Psi_\ell-\Psi_{\ell-1})-E_{M_\ell}(\Psi_\ell-\Psi_{\ell-1})\|_{L^2(\gO)}^2 
                \\
		&=
		\sum_{\ell=1}^L M_\ell^{-1} \|\Psi_\ell-\Psi_{\ell-1}\|_{L^2(\gO)}^2.
	\end{align*}
	The first equality holds since the MC estimators $E_{M_1}(\Psi_1), \dots, E_{M_L}(\Psi_L-\Psi_{L-1})$ 
        are jointly independent and unbiased in the sense that 
        $\bE(\Psi_\ell-\Psi_{\ell-1})=\bE(E_{M_\ell}(\Psi_\ell-\Psi_{\ell-1}))$.
	The triangle inequality and the estimate~\eqref{eq:mcdiscretebias} (from the proof of Theorem~\ref{thm:mcerror}) 
        then further yield
	\begin{align*}
		\|\bE(\Psi_L)-E^L(\Psi_L)\|_{L^2(\gO)}^2
		&\le 
		2\sum_{\ell=1}^L M_\ell^{-1}\left[ \|\Psi_\ell-\Psi(u)\|_{L^2(\gO)}^2+\|\Psi(u)-\Psi_{\ell-1}\|_{L^2(\gO)}^2\right]
                \\
		&\le 
		C\sum_{\ell=1}^L M_\ell^{-1}\left[ 2^{-2tN_\ell}+h_\ell^{2(2-\gt)r}+2^{-2tN_{\ell-1}}+h_{\ell-1}^{2(2-\gt)r}\right]
                \\
		&\le C\sum_{\ell=1}^L w_\ell^{-1} h_L^{2(2-\gt)r} ,
	\end{align*}
	where we have used Assumption~\ref{ass:refinement} and 
        the choices for $M_\ell$ and $N_\ell$ in~\eqref{eq:mlmcparameters} in the last step.
	As $\sum_{\ell=1}^L w_\ell^{-1}\le C_w<\infty$ is bounded independently of $L$, 
	we conclude with~\eqref{eq:refinement-bounds}, 
        $L$ as in \eqref{eq:mlmcparameters}, and $0<\ul c_\cK\le \ol c_\cK<1$ 
        that
		\begin{equation*}
			\|\bE(\Psi(u))-E^L(\Psi_L)\|_{L^2(\gO)}
			\le C h_L^{(2-\gt)r}
			\le C (\ol c_\cK^L h_0)^{(2-\gt)r}\
			\le C\eps^{\frac{\log(\ol c_\cK)}{\log(\ul c_\cK)}}
			h_0^{\left(1-\frac{\log(\ol c_\cK)}{\log(\ul c_\cK)}\right)(2-\gt)r}
	        \le C\eps.
		\end{equation*}
\end{proof}

The computational advantages of the MLMC method are precisely quantified in the next statement. Therein, the choice of $w_\ell$ plays a key role and depends on the relation of variance decay and cost of sampling on each level.

\begin{thm}\label{thm:work}
	Let Assumptions~\ref{ass:functional} and~\ref{ass:refinement} hold, and let $\eps>0$.
	Given $t, r$ and $\gt$ and $\eps>0$, set $L, M_\ell$ and $N_\ell$ as in Theorem~\ref{thm:mlmc} and choose the weight functions
	\begin{equation*}
		w_\ell = 
		\begin{cases}
			\ell^{1+\iota} \quad&\text{if $2(2-\gt)r>d$} \\
			L \quad&\text{if $2(2-\gt)r=d$} \\
			\ul c_\cK^{(2(2-\gt)r-d)\rev{(L-\ell)/2}} \quad&\text{if $2(2-\gt)r<d$} \\
		\end{cases},
	\quad \ell\in\{1,\dots, L\},
	\end{equation*}
	where \rev{$\iota>0$} is an arbitrary small constant.
	Then, the MLMC estimator satisfies
	\begin{equation*}
		\|\bE(\Psi(u))-E^L(\Psi_L)\|_{L^2(\gO)} \le C\eps,
	\end{equation*}
	with computational cost $\cC_{MLMC}$ for $\eps\to 0$ of order
	\begin{equation*}
		\cC_{MLMC} =
		\begin{cases}
			\cO(\eps^{-2}) \quad&\text{if $2(2-\gt)r>d$} \\
			\cO(\eps^{-2}\log(\eps)^2)  \quad&\text{if $2(2-\gt)r=d$} \\
			\cO(\eps^{-2-\frac{d-2(2-\gt)r}{(2-\gt)r}}) \quad&\text{if $2(2-\gt)r<d$}.
		\end{cases}
	\end{equation*}
\end{thm}

\begin{proof}
	Since $\ul c_\cK\in(0,1)$, it holds in each scenario
	$\sum_{\ell=1}^L w_\ell^{-1}\le C$ for a constant $C>0$, and uniform with respect to $L$.
	Therefore, we conclude by Theorem~\ref{thm:mlmc} that 
	\begin{equation*}
		\|\bE(\Psi(u))-E^L(\Psi_L)\|_{L^2(\gO)} \le C\eps,
	\end{equation*}
	and it remains to derive the complexity in terms of $\eps$.
	
	Assumption~\ref{ass:refinement} implies that $h_L \le h_\ell \ul c_\cK^{(L-\ell)} h_0$. 
        We obtain with~\eqref{eq:mlmcparameters} that 
	\begin{equation}\label{eq:levelsamples}
		M_\ell=\left\lceil \ul {c_\cK}^{(\ell-L)2(2-\gt)r}w_\ell \right\rceil.
	\end{equation}
	Let $\cC_\ell$ denote the work required to generate on sample of $\Psi_\ell$.
	As $h_\ell \ge\ul c_\cK^{\ell} h_0$, it holds that 
	\begin{equation}\label{eq:levelcost}
		\cC_\ell=\cO(\dim(V_{h_\ell}))=\cO(h_\ell^{-d})\le C {\ul c_\cK}^{-d\ell}
	\end{equation}
	Since we generate $M_\ell$ independent $\Psi_\ell-\Psi_{\ell-1}$ on each level (and also generate independent samples across all levels), the overall cost of the MLMC estimator is with~\eqref{eq:levelsamples} and~\eqref{eq:levelcost} bounded by  
	\begin{align*}
		\cC_{MLMC} 
		&:= 
		\sum_{\ell=1}^L M_\ell(\cC_\ell+\cC_{\ell-1}) \\
		&\le
		C \sum_{\ell=1}^L \ul c_\cK^{(\ell-L)2(2-\gt)r}w_\ell (\ul c_\cK^{-d\ell}+\ul c_\cK^{-d(\ell-1)}) \\
		&\le  
		C \ul c_\cK^{-L2(2-\gt)r} 
		\sum_{\ell=1}^L \left(\ul c_\cK^{2(2-\gt)r-d}\right)^{\ell} w_\ell.
	\end{align*}
	Now first suppose that $2(2-\gt)r-d>0$. Since $\ul c_\cK\in(0,1)$, the ratio test for sum convergence shows that for any $\iota>0$ we obtain the uniform bound (with respect $L\in\bN$) 
	\begin{equation*}
		\sum_{\ell=1}^L \left(\ul c_\cK^{2(2-\gt)r-d}\right)^{\ell} w_\ell
		\le 
		\sum_{\ell\in\bN} \left(\ul c_\cK^{2(2-\gt)r-d}\right)^{\ell} \ell^{1+\iota} 
		<\infty.
	\end{equation*}
	On the \rev{other hand, if} $2(2-\gt)r-d=0$, there holds with $w_\ell=L$
	\begin{equation*}
		\sum_{\ell=1}^L \left(\ul c_\cK^{2(2-\gt)r-d}\right)^{\ell} w_\ell
		= L^2.
	\end{equation*}
	Finally, for $2(2-\gt)r-d<0$, it follows that there is a $C>0$ such that
	\begin{equation*}
		\sum_{\ell=1}^L \left(\ul c_\cK^{2(2-\gt)r-d}\right)^{\ell} w_\ell
		=
		\rev{
		\ul c_\cK^{L(2(2-\gt)r-d)/2}
		\sum_{\ell=1}^L \ul c_\cK^{\ell(2(2-\gt)r-d)/2}
	}
		\le C \ul c_\cK^{L(2(2-\gt)r-d)}.
	\end{equation*}
	Altogether, we obtain that there exists a constant $C>0$ independent of $L$ such that
	\begin{equation*}
		\cC_{MLMC} \le  C
		\begin{cases}
			\ul c_\cK^{-L2(2-\gt)r} \quad&\text{if $2(2-\gt)r>d$}, \\
			\ul c_\cK^{-L2(2-\gt)r} L^2 \quad&\text{if $2(2-\gt)r=d$}, \\
			\ul c_\cK^{-L2(2-\gt)r}\ul c_\cK^{L(2(2-\gt)r-d)} \quad&\text{if $2(2-\gt)r<d$}.
		\end{cases}
	\end{equation*}
	With $L$ from~\eqref{eq:mlmcparameters} it follows that 
	$\ul c_\cK^L=\cO(\eps^{\frac{1}{(2-\gt)r}}h_0^{-1})$ for $\eps\to0$. 
        This shows the following asymptotics for the $\eps$-cost bounds as $\eps\to0$
	\begin{equation*}
		\cC_{MLMC}(\eps) =
		\begin{cases}
			\cO(\eps^{-2}) \quad&\text{if $2(2-\gt)r>d$}, \\
			\cO(\eps^{-2}\log(\eps)^2)  \quad&\text{if $2(2-\gt)r=d$}, \\
			\cO(\eps^{-2-\frac{d-2(2-\gt)r}{(2-\gt)r}}) \quad&\text{if $2(2-\gt)r<d$}.
		\end{cases}
	\end{equation*}
\end{proof}

\begin{rem}
	The asymptotic complexity bounds for $\eps\to0$ are of the same magnitude as in 
\cite[Theorem 3.1]{giles2008multilevel} and \cite[Theorem 1]{cliffe2011multilevel}, 
but only require knowledge of the parameters $r,t$ and $\gt$, but not on further absolute constants.
From Theorem~\ref{thm:mcerror}, the SLMC estimator requires for given $\eps>0$ 
a total of $M\approx\eps^{-2}$ samples with refinement parameters satisfying $2^{-tN}\approx h^{(2-\gt)r}\approx\eps$.
Hence,  $h=\cO(\eps^{\frac{1}{(2-\gt)r}})$ and, 
assuming availability of a linear complexity solver such as multigrid, 
the computational cost per sample 
is bounded asymptotically as
$\cO(\dim(V_h))=\cO(h^{-d})=\cO(\eps^{-\frac{d}{(2-\gt)r}})$.  
The total cost of the SLMC estimator to achieve $\eps$-accuracy is therefore
	\begin{equation}\label{eq:slmcwork}
		\cC_{SLMC}(\eps) = \cO(\eps^{-2-\frac{d}{(2-\gt)r}})>\cC_{MLMC}(\eps),\quad 
			\text{as $\eps\to 0$}.
	\end{equation}
Consequently, under the stated assumptions,
MLMC-FEM achieves a considerable reduction in (asymptotic) $\eps$-complexity, 
even in low-regularity regimes with $2(2-\gt)r<d$.

In case that $2(2-\gt)r>d$, the assumption that 
$E_{M_1}(\Psi_1), \dots, E_{M_L}(\Psi_L-\Psi_{L-1})$ 
are independent MC estimators is not required to derive the 
optimal complexity $\cC_{MLMC}=\cO(\eps^{-2})$. 
Instead, setting $w_\ell=\ell^{2(1+\iota)}$ is sufficient to compensate for the dependence 
across discretiation levels. 
This could be exploited in a simulation 
to "recycle" samples from coarser discretization levels in order
to further increase efficiency. 
We refer, e.g., to the discussion in \cite[Section 5.2]{BS18b}.
\end{rem}

\section{Numerical Experiments}
\label{sec:numerics}
We consider numerical experiments in the rectangular domain $\cD:=\bT^2=(0,1)^2$ 
and use the constant source function $f\equiv 1$. 
For the spatial discretization we use bilinear finite elements that 
may be efficiently computed by exploiting their tensor product-structure, 
see Appendix~\ref{appendix:bilinearFE} for details. 
The initial mesh width is given by $h_0=\frac{1}{2}$ and 
we use dyadic refinements with factor $\ul c_\cK=\ol c_\cK = \frac{1}{2}$ 
to obtain a sequence of tesselations $(\cK_h,\; h=2^{-\ell} h_0,\; \ell\in\bN)$ 
that satisfies Assumption~\ref{ass:triangulation} for the MLMC algorithm. 
We further use midpoint quadrature to assemble the stiffness matrix for each realization of the diffusion coefficient. 
The resulting quadrature error does not dominate the FE error convergence from Theorems~\ref{thm:H1error} and \ref{thm:L2error}, 
as we show in Lemma~\ref{lem:quadrature} in the Appendix. 
For given $N$ and a rectangular mesh $\cK_h$, we evaluate $b_{T,N}$ at the midpoint of each $K\in\cK_h$ 
as described in Appendix~\ref{appendix:coefficient}.

We investigate different parameter regimes of varying smoothness for the diffusion coefficient, 
the values and resulting pathwise approximation rates $r$ and $t$ as in Corollary~\ref{cor:fullerror} 
are collected in Table~\ref{table:parameters}.
In all experiments, we build the random field $b_T$ resp.  $b_{T,N}$ 
based on Daubechies wavelets with five vanishing moments ("$\mathrm{DB}(5)$-wavelets"), 
with smoothness $\phi, \psi \in C^{1.177}(\bR)$ (see \cite[Section 7.1.2]{daubechies1992ten}).
We consider the $L^2(\cD)$-norm of the gradient as quantity of interest (QoI), with associated functional given by  
\begin{align*}
	&\Psi:H^1(\cD)\to [0,\infty),
	\quad u\mapsto\left(\int_\cD |\nabla u|^2 dx\right)^{1/2}, 
\end{align*}
so that Assumption~\ref{ass:functional} holds with $\theta=1$.

\begin{table}[h]
	\centering
	\begin{tabular}{ |c||c|c|c|c||c|c|| c| } 
		\hline
		Parameter values & $s$ & $p$ & $\gk$ & $\gb$ & $t$ & $\rev{s-\frac{d}{p}}$
		& MLMC complexity for $\gt=1$ \\ 
		\hline \hline
		"smooth Gaussian" & 2 & 2 & 1 & $\frac{1}{2}$ & 1 & 1 & 
		$
		\cO(\eps^{-2}|\log(\eps)|^2)
		$\\
		\hline
		"rough Gaussian" & $\frac{3}{2}$ & 2 & 1 & $\frac{1}{2}$ & $\frac{1}{2}$ & $\frac{1}{2}$ & 
		$
		\cO(\eps^{-4})
		$\\
		\hline
		"$p$-exponential" & 2 & $\frac{8}{5}$ & 1 & $\frac{3}{4}$ & $\frac{3}{4}$ & $\frac{3}{4}$ & 
		$
		\cO(\eps^{-\frac{8}{3}})
		$\\
		\hline
	\end{tabular}
	\caption{\small Parameters values for the random tree Besov priors in the numerical experiments.}
	\label{table:parameters}
\end{table}

Given $\theta$, $r$ and $t$, we prescribes target accuracies 
$\eps = 2^{-r\xi}$, $\xi\in\{5,\dots, 9\}$ 
and select, for given $\eps>0$, 
the MLMC parameters as in Theorem~\ref{thm:mlmc}.
The maximum refinement level is denoted by $L_\eps$ and the 
corresponding estimators by $E^{L_\eps}(\Psi_{L_\eps})$.
We sample $n_{ML}=2^8$ realizations of $E^{L_\eps}(\Psi_{L_\eps})$ for every $\eps$. 
As reference solution, we use $n_{ref}=2^4$ realizations of  $E^{L_{ref}}(\Psi_{L_{ref}})$ with parameters 
adjusted to achieve $\eps_{ref}:=2^{-r11}$.
We report for prescribed $\eps$ the \emph{realized empirical RMSE} 
\begin{equation*}
	\text{RMSE}(\eps)=
	\left(
	\frac{1}{n_{ML}}\sum_{i=1}^{n_{ML}} 
	\left( 
	E^{L_\eps}(\Psi_{L_\eps})(\go_i) 
	- 
	\left(
	\frac{1}{n_{ref}}\sum_{j=1}^{n_{ref}}E^{L_{ref}}(\Psi_{L_{ref}})(\go_j)
	\right)
	\right)^2
	\right)^{\frac{1}{2}}.
\end{equation*}
All computations are realized using \textsc{MATLAB} on a workstation with 
two  Intel Xeon E5-2697 CPUs with 2.7 GHz, a total of 12 cores, and 256 Gigabyte RAM.

\bigskip

We start with the "smooth Gaussian" case from Table~\ref{table:parameters}. A sample of the diffusion coefficient and the associated bilinear FE approximation of $u$ is given in Figure~\ref{fig:smoothsample}, where we also plot \rev{the (average) CPU times against the realized RMSE}. As we see, the realized error is very close to the prescribed accuracy $\eps$, which corresponds to the error estimate from Theorem~\ref{thm:mlmc}.
Moreover, the empirical results are in line with the work estimates from Theorem~\ref{thm:work}, as is seen  in the right plot of Figure~\ref{fig:smoothsample}, since the computational complexity is (asymptotically) of order $\cO(\eps^{-2}|\log(\eps)|^2)$.

Next, we decrease the parameter $s$ to obtain the "rough Gaussian" scenario from Table~\ref{table:parameters}. A sample of the diffusion coefficient and the associated bilinear FE approximation of $u$ is given in Figure~\ref{fig:roughsample}. Compared to Figure~\ref{fig:smoothsample}, we now see more detailed, sharp features in the diffusion coefficient, due to the slower decay factor of the wavelet basis.	
Average CPU times vs. realized RMSE are given in Figure~\ref{fig:roughsample}. 
Again, the realized error is of order $\cO(\eps)$, and the computational times are asymptotically of order $\cO(\eps^{-4})$, as expected from Theorem~\ref{thm:work}.

Finally, we investigate the "$p$-exponential" scenario from Table~\ref{table:parameters}, where we use a heavier-tailed distribution of $X_{j,k}^{l}$ and increase the wavelet density to $\gb=\frac{3}{4}$.
We use a standard Acceptance-Rejection algorithm to sample from the $p$-exponential density for $p=1.6$. 
A sample of the diffusion coefficient and the associated bilinear FE approximation of $u$ is given in Figure~\ref{fig:pexponentialsample}.
We observe that the variance of coefficient and solution is increased, compared to the previous two examples. This is indicated by the larger bars of the confidence intervals in the right plot of Figure~\ref{fig:pexponentialsample}. The a-priori accuracy has been scaled by a factor of three in this plot, for a better visual comparison of realized and prescribed RMSE.
Although absolute magnitude and variance of the realized RMSE have increased, we still recover in line with Theorem~\ref{thm:work} the asymptotic error decay of order $\cO(\eps)$, together with CPU times of order $\cO(\eps^{-\frac{8}{3}})$.

\rev{Our experiments show that sharper features in the prior model are achieved by decreasing either $s$ or $p$. We emphasize that the asymptotic complexity depends on the differential dimension $s-\frac{d}{p}$ as indicated in Table~\ref{table:parameters},the latter essentially corresponds to the parameter $r$ in Theorem~\ref{thm:mlmc}.  Consequently, the effect of lowering $p$ on the overall regularity is more pronounced in as the physical dimension $d$ increases.}

\begin{figure}
	\centering
	\subfigure{\includegraphics[scale = 0.32]{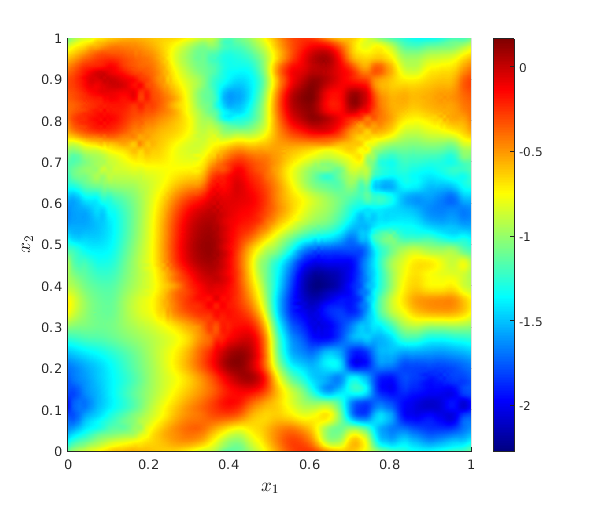}}
	\subfigure{\includegraphics[scale = 0.32]{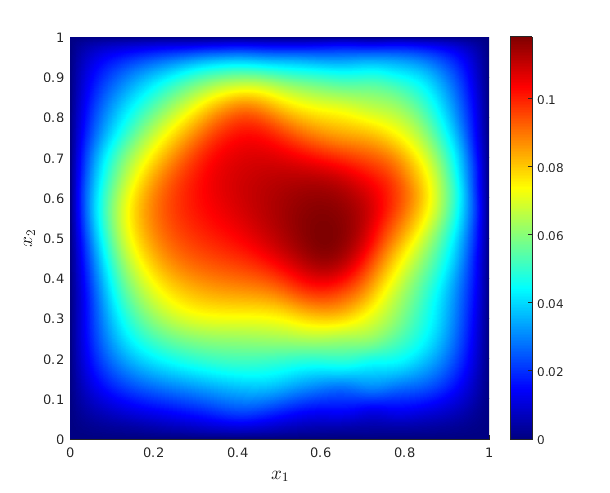}}
	\subfigure{\includegraphics[scale = 0.32]{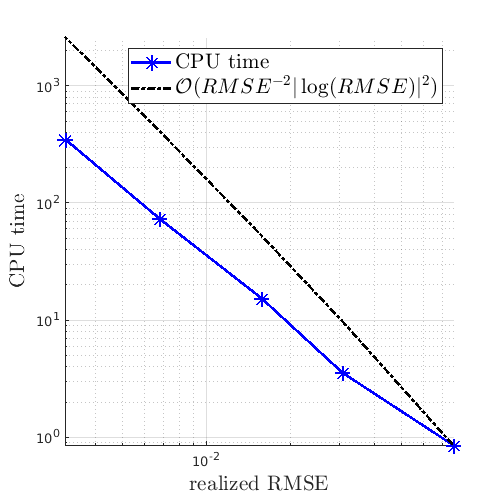}}
	\caption{Sample of a Besov random tree prior/log-diffusion coefficient $b_{T}$ (left) and the corresponding finite element approximation of $u$ (middle) for the ``smooth Gaussian'' case in Table~\ref{table:parameters}.
		Coefficient and solution in the figures have been sampled on a grid with $2^9\times 2^9$ equidistant points, 
                and wavelet series truncation was after $N=9$ scales.
		Fractal structures in the log-diffusion caused by the wavelet density $\gb=\frac{1}{2}$ are clearly visible in the left plot.
The realized RMSE (blue) resp. predicted RMSE (red) vs. computational complexity is depicted in the right plot. 
Both curves exhibit the predicted asymptotic behavior of $\cO(\eps^{-2}|\log(\eps)|^2)$, as indicated by the dashed line. 
The support line $\cO(\eps^{-2.35})$ has been added to show that the complexity is indeed non-linear in the log-scale.
The bars indicate 95\%-confidence intervals of the realized RMSE. 
}
	\label{fig:smoothsample}

	\centering
	\subfigure{\includegraphics[scale = 0.32]{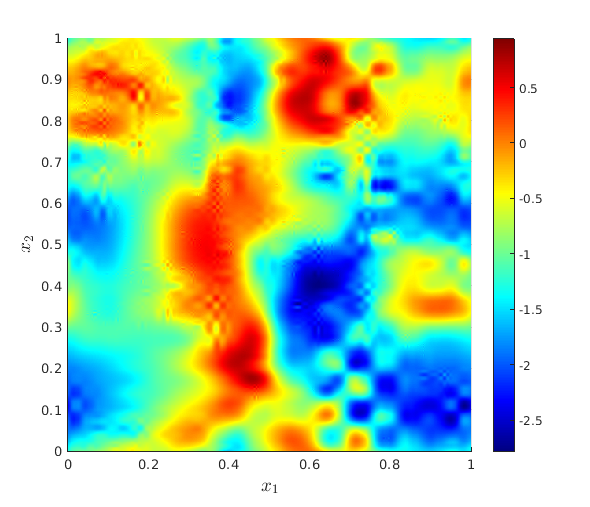}}
	\subfigure{\includegraphics[scale = 0.32]{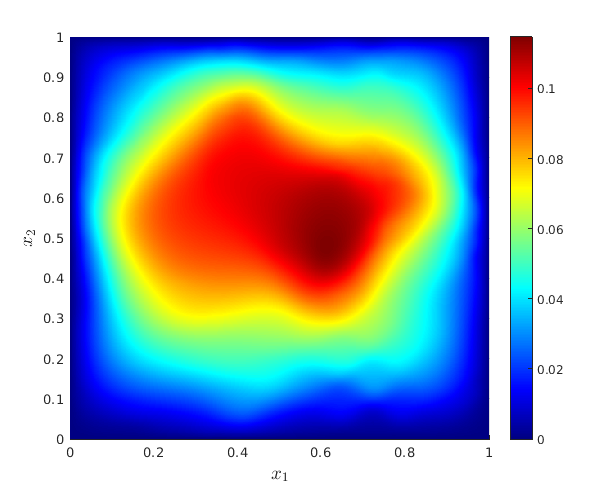}}
	\subfigure{\includegraphics[scale = 0.32]{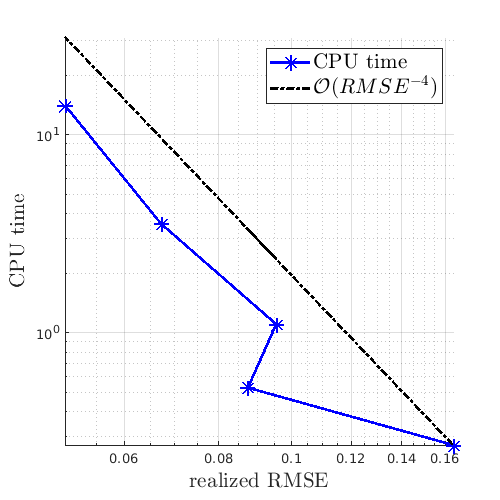}}
	\caption{Sample of a Besov random tree prior/log-diffusion coefficient $b_{T}$ (left) and the corresponding finite element approximation of $u$ (middle) for the ``rough Gaussian'' case in Table~\ref{table:parameters}.
		Coefficient and solution in the figures have been sampled on a grid with 
                $2^9\times 2^9$ equidistant points, and wavelet series truncation was after $N=9$ scales.
		The diffusion coefficient exhibits sharper features, 
                as compared to the smooth case Figure~\ref{fig:smoothsample}. 
		The realized RMSE (blue) resp. predicted RMSE (red) vs. 
                computational complexity is depicted in the right plot, 
                both curves show the predicted asymptotic behavior of $\cO(\eps^{-4})$. 
The bars indicate 95\%-confidence intervals of the realized RMSE.}
	\label{fig:roughsample}

	\centering
	\subfigure{\includegraphics[scale = 0.32]{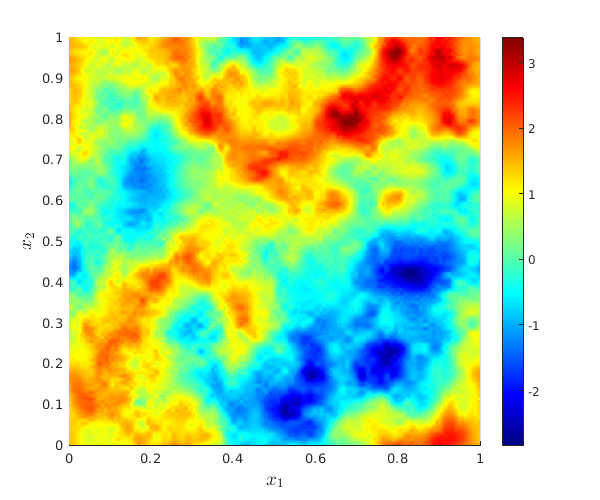}}
	\subfigure{\includegraphics[scale = 0.32]{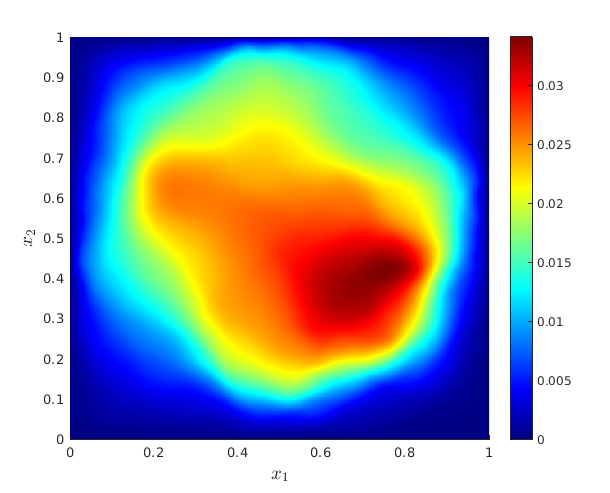}}
	\subfigure{\includegraphics[scale = 0.32]{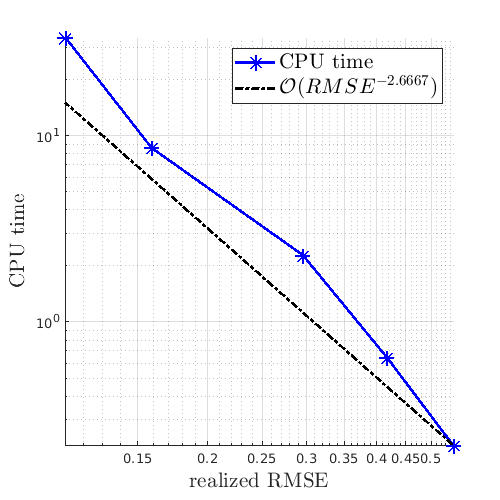}}
	\caption{Sample of a Besov random tree prior/log-diffusion coefficient $b_{T}$ (left) and the corresponding finite element approximation of $u$ (middle) for the ``p-exponential'' case in Table~\ref{table:parameters}.
	Coefficient and solution in the figures have been sampled on a grid with $2^9\times 2^9$ equidistant points, 
        and wavelet series truncation after $N=9$ scales.
	The variance is significantly increased compared to the previous examples, as indicated by the bars 95\%-confidence intervals of the realized RMSE (right plot). The a-priori fixed $\eps$ (red curve) has been scaled by a factor of three in this plot.	
	We still recover the predicted asymptotic error of order $\cO(\eps)$ with computational work of order $\cO(\eps^{-\frac{8}{3}})$. }
	\label{fig:pexponentialsample}
\end{figure}

\section{Conclusions}
\label{sec:Concl}
We have developed a computational framework for the efficient discretization of 
linear, elliptic PDEs with log-Besov random field coefficients which are modelled
by a multiresolution in the physical domain whose coefficients
are \rev{$p$-exponential} with random choices of active coefficients according to 
GW-trees. 
The corresponding pathwise diffusion coefficients generally admit only rather low
path regularity, thereby mandating low order 
Finite Element discretizations in the physical domain.
We established strong pathwise solution regularity, 
and FE error bounds for the corresponding single-level \MC-FEM algorithm. 
The corresponding error vs. work bounds for the multi-level \MC algorithm
follow then in the standard way. 
We emphasize again that higher order sampling methods seem
to be obstructed by the GW-tree structure which has recently been identified in \cite{KLSS21} 
as well-suited for modelling diffuse media such as clouds, fog and aerosols.
The presently proposed MLMC-FE error analysis for Elliptic PDEs with
(log-) Besov random tree coefficients will imply corresponding 
complexity bounds in multilevel Markov Chain Monte Carlo 
sampling strategies for Bayesian Inverse Problems on log-Besov random tree priors,
as considered for example in imaging applications in
\cite{LassasBesov09,DHS12,AgaDasHel2021,HL11}.
Details on their analysis and computation will be developed elsewhere.

\subsection*{Acknowledgements}

AS was partly funded by the ETH Foundations of Date Science Initiative (ETH-FDS), and it is gratefully acknowledged.

\bibliographystyle{abbrv}
\bibliography{references}
\appendix
\section{Galton-Watson trees}
\label{appendix:trees}
We provide some basic concepts of discrete trees, give a formal definition of \emph{Galton-Watson trees}, 
and record a result on their extinction probabilities. 
The presentation follows~\cite[Section 2]{abraham2015GW-Trees}, 
with modified notation where necessary. 
\subsection{Notation and basic definitions}
\label{appendix:trees-basics}
Let $\cU:=\bigcup_{n\ge 0}\bN^n$
be the set of all finite sequences of positive integers, 
where $\varrho:=()$ denotes the empty sequence, and we use the convention $\bN^0=\{\varrho\}$. 
For any $\mf n\in\cU$, let $|\mf n|$ denote the length of $\mf n$, with $|\varrho|:=0$.
For $\mf n,\mf m\in\cU$, we denote by $\mf n\mf m$ the concatenation of two sequences, 
with the convention $\mf n\varrho=\varrho \mf n=\mf n$ for all $\mf n\in\cU$.
There exists a partial order, called the \emph{genealogical order, on $\cU$}: 
we say that $\mf m\preceq \mf n$, whenever there is a $\mf n_0\in\cU$ such that $\mf m\mf n_0 = \mf n$. 
We say that $\mf m$ is an ancestor of $\mf n$ and write $\mf m\prec \mf n$ if $\mf m\preceq \mf n$ and $\mf m\neq \mf n$. 
The \emph{set of all ancestors of $\mf n$} 
is denoted by $A_\mf n:=\{\mf m\in\cU|\, \mf m\prec \mf n\}\subset\cU$. 

\begin{defi}\cite[Section 2.1]{abraham2015GW-Trees}\label{def:trees}
	A tree $\bf t$ is a subset $\bf t\subset\cU$ that satisfies 
	\begin{itemize}
		\item $\varrho\in\bf t$,
		\item If $\mf n\in\bf t$, then $A_{\mf n}\subset\bf t$,
		\item For any $\mf n\in\bf t$, there exists $\mf K_{\mf n}(\bf t)\in\rev{\bN_0}\cup\{\infty\}$, such that for every $n\in\bN$, $\mf n n\in\bf t$ if, and only if, $1\le n\le \mf K_{\mf n}(\bf t)$.
	\end{itemize}
\end{defi}

We denote the \emph{set of all trees} by $\mfT_\infty$.
Let $|{\bf t}|\in \bN \cup\{\infty\}$ be the cardinality of the tree $\bf t\in \mfT_\infty$, 
and introduce the \emph{set of all finite trees by 
$\mfT_0:=\{{\bf t}\in\mfT|\, |{\bf t}|<\infty\}$}.
The set $\mfT_0$ is countable.
The integer $\mf K_{\mf n}(\bf t)$ represents the number of offsprings in $\bf t$ at the node $\mf n$. 
The \emph{set of all trees without infinite nodes} is a subset $\mfT_\infty$ and denoted by
\begin{equation}\label{eq:set-of-trees}
	\mfT:=\{{\bf t}\in\mfT_\infty|\,\mf K_{\mf n}(\bf t)<\infty \text{ for all $\mf n\in\bf t$.} \}
\end{equation}
For $\mf n\in\bf t$, the \emph{sub-tree $\mf S_{\mf n}(\bf t)$ of $\bf t$ above 
node $\mf n$} is defined as:
\begin{equation}\label{eq:subtrees}
	\mf S_{\mf n}({\bf t}):=\{\mf m\in\cU|\,\mf n\mf m\in\bf t\}.
\end{equation}
We also need the \emph{restriction functions $\mf r_n:\mfT\to \mfT$, $n\in\bN_0$} which are
given by 
\begin{equation}\label{eq:tree-restriction}
	\mf r_n({\bf t}):=\{\mf n\in{\bf t}|\,|\mf n|\le n\}.
\end{equation}
With these preparations, 
we are in position to define metric and associated Borel $\gs$-algebra on $\mfT$. 
\begin{defi}\cite[Section 2.1]{abraham2015GW-Trees}
	\label{defi:tree-algebra}
	Let 
	\begin{equation}\label{eq:tree-metric}
		\gd_{\mfT}:\mfT\times\mfT\to[0, 1],
		\quad 
		({\bf t}, {\bf t'})\mapsto 2^{-\sup\{n\in\bN_0|\mf r_n({\bf t})=\mf r_n({\bf t'})\}}.
	\end{equation}  
	Furthermore, define the $\gs$-algebra
	\begin{equation}\label{eq:tree-sigma-algebra}
		\cB(\mfT)
		:=
		\gs\left(
		\bigcup_{t\in\mfT} 
		\bigcup_{n\in\bN} 
		\mf r_n^{-1}(\mf r_n({\bf t}))
		\right)
		=
		\gs\left(
		\bigcup_{t\in\mfT} 
		\bigcup_{n\in\bN}  
		\{{\bf t'}|\, \gd_{\mfT}({\bf t}, {\bf t'})\le 2^{-n}\}
		\right).
	\end{equation}
\end{defi}

By~\cite[Lemma 2.1]{abraham2015GW-Trees}, $(\mfT, \gd)$ is a complete and separable metric space.
The countable set of all finite trees $\mfT_0$ is dense in $\mfT$:
for all $\bf t\in\mfT$, we have $\mf r_n({\bf t})\to {\bf t}$ as $n\to\infty$ in $(\mfT, \gd)$. 
Further, 
$\gd_{\mfT}$ is an ultra-metric (see \cite[Section 2.1]{abraham2015GW-Trees}),
hence $\mf r_n^{-1}(\mf r_n({\bf t}))$ is the set of open (and closed) 
balls with center $\bf t$ and radius \rev{$2^{-n}$} \cite[Section 2.1]{abraham2015GW-Trees} with respect to $\gd$.
By separability, 
$\cB(\mfT)$ coincides with the Borel $\gs$-algebra on $\mfT$, 
that is generated by all open sets on $(\mfT, \gd_{\mfT})$.
Given a probability space $(\gO, \cA, P)$, 
we then call any $\cA/\cB(\mfT)$-measurable mapping $\tau:\gO\to\mfT$ a \emph{$\mfT$-valued random variable}. 
This allows us to formalize \emph{Galton-Watson} trees:

\begin{defi}\label{def:GW-tree}
[Galton-Watson(GW) tree with offspring distribution $\cP$]
	A $\mfT$-valued random variable $T$ has the \emph{branching property} if, for any $n\in\bN$, 
conditionally on $\{k_\varrho(\tau)=n\}$, the sub-trees $\mf S_1(T), \dots, \mf S_n(T)$ 
are independent and distributed as the original tree $\tau$.
	Now let $\cP$ be a probability distribution on $\bN_0$, i.e., 
        a probability measure on $(\bN_0, \cB(\bN_0))$.  
	A $\mfT$-valued random variable $T$ is called a \emph{Galton-Watson(GW) tree
        with offspring distribution $\cP$} if $T$ has the branching property 
        and if $\mf K_\varrho(T)\sim \cP$.
\end{defi}

According to \cite[Equation (12)]{abraham2015GW-Trees}, 
the \emph{distribution $\bQ_T$ of a GW tree $T$, restricted to the set of finite trees,} 
is given by
\begin{equation}\label{eq:tree-measure}
	\bQ_T(T={\bf t}) =
	\prod_{\mf n\in{\bf t}}\cP(\mf K_{\mf n}({\bf t})),
	\quad 
	{\bf t}\in \mfT_0,
\end{equation}
where the product on the right hand side is finite. 
\subsection{Random wavelet trees and extinction probabilities}
\label{sec:RndWavTree}
\label{appendix:wavelet-trees}
Now we consider again the $d$-dimensional torus $\bT^d$ with wavelet basis $\mathbf \Psi$ as in~\eqref{eq:torusbasis2}.
To relate the nodes of a GW tree to the wavelet indices in $\cI_{\mathbf \Psi}=\{j\in\bN_0,\; k\in K_j,\;l\in\cL_j\}$, 
we assume that $T$ is a GW tree with offspring distribution $\cP=\textrm{Bin}(2^d, \gb)$ for some $\gb\in[0,1]$.
For a given realization $T(\go)$ and node $\mf n\in T(\go)$, we identify the length of $\mf n$ with the corresponding wavelet scale via $j:=|\mf n|\in\bN_0$. 
Further, since $\cP$ is binomial, there are at most $2^{dj}$ nodes $\mf n$ of length $j$ in $T(\go)$. 
Each of \rev{these} nodes has $j$ entries in $\{1,\dots,2^d\}$. 
We assign an integer to all $\mf n\in T(\go)$ with $|\mf n|=j$ via the bijection
\begin{equation*}
	\mf I^1_{d,j}: \{1,\dots,2^d\}^j\to \{1,\dots,2^{dj}\},
	\quad
	\mf n\mapsto 1 + \sum_{i=1}^{j}2^{d(j-i)}(\mf n_i-1).
\end{equation*}
On the other hand, we assign for any $\{ 1,\dots, 2^{jd}\}$ an index in $K_j=\{0,\dots,2^j-1\}^d$ via 
\begin{equation*}
	\mf I^2_{d,j}: \{1,\dots, 2^{jd}\} \to K_j,
	\quad
	n\mapsto \left(
	\max(n-2^{j(i-1)}, 0)~\textrm{mod}~2^j  
	\right)_{i=1,\dots,d},
\end{equation*}
which yields a one-to-one mapping 
\begin{equation}\label{eq:node-to-index-map}
	\mf I_{d,j}: \{1,\dots,2^d\}^j\to K_j,
	\quad
	\mf n\mapsto \left(\mf I^2_{d,j}\circ \mf I^1_{d,j}\right)(\mf n).
\end{equation}
Thus, each $\mf n\in T(\go)$ corresponds to a unique pair of indices $(j,k)$ via 
$\mf n\mapsto (|\mf n|, \mf I^2_{d,|\mf n|}\circ\mf I^1_{d,|\mf n|}(\mf n))$.
Collecting the pairs $(j,k)$ for all nodes in $T$ yields the 
\emph{random active index set}
\begin{equation}\label{eq:node-to-index-set}
	\mf I_T(\go):=
	\bigcup_{\mf n\in T(\go)}
	\left(|\mf n|, \mf I^2_{d,|\mf n|}\circ\mf I^1_{d,|\mf n|}(\mf n)\right)
	\subset
	\{j\in\bN_0, k\in K_j\}.
\end{equation}
Then, only wavelets with indices 
$\cI_T(\go):=\{(j,k,l)|\; (j,k)\in \mf I_T(\go),\; l\in\cL_j\}\subset \cI_{\mathbf \Psi}$ 
are ``activated'' in the series expansion of a sample of 
$b_T$ in Definition~\ref{def:randomtree-prior}.

It is then of course of interest whether the GW tree $T$ terminates after a finite number of nodes, 
in which case $\mf I_T$ is finite, or if $T$ has infinitely many nodes. 
In the latter case, $b_T$ exhibits fractal structures on $\bT^d$, 
in areas where the wavelet expansion has infinitely many terms. 
Let the \emph{extinction event of a GW tree $T$} be denoted by $\cE(T):=\{T\in\mfT_0\}$.
The \emph{extinction probability of GW trees} are quantified in the following result:
\begin{lem}\cite[Corollary 2.5/Lemma 2.6]{abraham2015GW-Trees}
	\label{lem:extinction}
	Let $T:\gO\to\mfT$ be GW tree with offspring distribution $\cP$ and let $\zeta\sim\cP$.
	\begin{enumerate}
		\item If $\cP(0)=0$, then $\rev{\bQ_T}(\cE(T))=0$, 
		\item If $\cP(0)\in (0,1)$ and $\cP(0)+\cP(1)=1$, then $\rev{\bQ_T}(\cE(T))=1$,
		\item If $\cP(0)\in (0,1)$, $\cP(0)+\cP(1)<1$, and $\bE(\zeta)\le 1$, then $\rev{\bQ_T}(\cE(T))=1$,
		 \item If $\cP(0)\in (0,1)$, $\cP(0)+\cP(1)<1$, and $\bE(\zeta)> 1$, 
                       then $\rev{\bQ_T}(\cE(T))=q\in(0,1)$. 
                       Here $q$ is the smallest root in $[0,1]$ of the equation $\bE(q^\zeta)=q$. 
	\end{enumerate} 
\end{lem} 
Lemma~\ref{lem:extinction} does not require a Binomial offspring distribution, 
but remains true for arbitrary distributions $\cP$ on $\bN_0$. 
Moreover, we conclude that a GW tree $T$ with offspring distribution $\cP=\textrm{Bin}(2^d, \gb)$ 
generates finite wavelet expansions via $\mf I_T$ $P$-a.s. if and only if $\gb\in [0,2^{-d}]$.

\subsection{\rev{Parametrization of binomial GW trees}}
\label{appendix:parametrization}
\rev{
Recall that $(\mfT, d_\mfT)$ is polish, but not normed.
The product parameter space $\bR^\bN\times \mfT$ 
of the Besov random tree prior is thus not a Banach space, 
which is in turn a key assumption for Bayesian inverse problems. 
However, it is possible to obtain suitable parametrizations for 
GW trees with binomial offspring distribution, by exploiting the 
connection to the continuous uniform distribution on $(0,1)$. 
}

\rev{
Let $T$ be a GW tree with with offspring distribution $\cP=\textrm{Bin}(2^d, \gb)$ for $\gb\in[0,1]$, 
let $\gO_U:=[0,1]^\bN$, and denote by $\cP_U$ the univariate uniform distribution on $(0,1)$ with density $f_U(u)=\indi_{\{ u\in(0,1)\}}(u)$.
We equip $\gO_U$ with the countable product-$\gs$-algebra and product measure, given by 
\begin{equation}  
	\cA_U:=\bigotimes_{n\in\bN} \cB((0,1)),\quad \text{and} \quad
	\bQ_U:=\bigotimes_{n\in\bN} \cP_U, 
\end{equation}
respectively, consider the probability space $(\gO_U, \cA_U, \bQ_U)$, and a sequence $U:=(U_{j,k}, j\in\bN,\; k\in K_j)$ of i.i.d. uniform random variables defined on $(\gO_U, \cA_U, \bQ_U)$. 
}

\rev{
Let $j\in\bN$, and consider the wavelet index $(j,k,l)\in \cI_\Psi$. 
By~\eqref{eq:node-to-index-map}, there is a unique node $\mf n$ of length $|\mf n|=j$ associated to $(j,k)$ for $k\in K_j$  via $(j,k) = (|\mf n|, \mf I^2_{d,|\mf n|}\circ\mf I^1_{d,|\mf n|}(\mf n))=(j, \mf I^2_{d,j}\circ\mf I^1_{d,j}(\mf n))$.
Recall that $A_{\mf n}$  is the set of ancestors of $\mf n$ and note that for each $i\le j-1$, there exists a unique ancestor $\mf m_i\in A_\mf n$ of $\mf n$ with length $|\mf m_i|=i$. We further introduce the notation $\mf m_j:=\mf n$ for convenience. 
By independence of the elements in the sequence $U$ we obtain that
\begin{align*}
	\bQ_U(\mf n\in T) 
	&= \bQ_U(\mf m_{|\mf n|-1}\in T,\; U_{j,k}\le \gb) \\
	&= \bQ_U(\mf m_{|\mf n|-1}\in T,\; U_{j, \mf I^2_{d,j}\circ\mf I^1_{d,j}(\mf n)}\le \gb) \\
	&=  \bQ_U(\mf m_{|\mf n|-2}\in T,\; 
	U_{j-1, \mf I^2_{d,j-1}\circ\mf I^1_{d,j-1}(\mf m_{|\mf n|-1})}\le \gb,\;
	U_{j, \mf I^2_{d,j}\circ\mf I^1_{d,j}(\mf n)}\le \gb) \\
	&=  \bQ_U(U_{1, \mf I^2_{d,1}\circ\mf I^1_{d,1}(\mf m_1)}\le \gb 
	,\dots,
	U_{j, \mf I^2_{d,j}\circ\mf I^1_{d,j}(\mf n)}\le \gb) \\
	&= \prod_{i=1}^{j} \cP_U(U_{i, \mf I^2_{d,i}\circ\mf I^1_{d,i}(\mf m_i)}\le \gb) \\
	&= \gb^{j}.
\end{align*} 
Hence, we may parametrize $T$, i.e., the random index set $\cI_T$, by the equivalence 
\begin{equation}\label{eq:uniform-equiv}
	(j,k, l)\in \cI_T(\go)\quad\Longleftrightarrow\quad 
	U_{1, \mf I^2_{d,1}\circ\mf I^1_{d,1}(\mf m_1)}, \dots, 
	U_{j-1, \mf I^2_{d,j-1}\circ\mf I^1_{d,j-1}(\mf m_{j-1})},\,
	U_{j, \mf I^2_{d,j}\circ\mf I^1_{d,j}(\mf n)}\le \gb
\end{equation}
for each $(j,k,l)\in \cI_\Psi$ with $j\ge 1$.
Equation~\eqref{eq:uniform-equiv} hence shows that the \emph{parameter-to-prior map} of a $B_s^p$-random variable with wavelet density $\gb$ as in~\eqref{eq:randomtreeprior} is given by   
\begin{equation}
	(0,1)^\bN \times \bR^N \to B_s^p, \quad
	U\times X \mapsto 
	\sum_{(j,k,l)\in \cI_\Psi} \eta_j X_{j,k}^l\psi_{j,k}^l \indi_{\{(j,k,l)\in \cI_T(\go)\}}.
\end{equation}
In particular, the parametrization with respect to $U\in (0,1)^\bN$ is \emph{discontinuous}, 
obstructing higher-order quadrature methods in the parameter domain, such as Quasi-\MC. 
}

\section{Finite element approximation}
\label{appendix:fem}
This appendix collects details on the (standard) 
implementation of the pathwise FE discretization 
from Section~\ref{sec:approximation}. 
In particular, 
we analyze the quadrature error arising during matrix assembly, 
describe an assembly routine based on tensorization for bilinear finite elements on $\bT^2$, 
and comment on the efficient evaluation and sampling cost of $a_N$.
\subsection{Finite element quadrature error}
\label{appendix:quadrature}
Let $h>0$ be the FE meshwidth 
and let $\{v_1,\dots,v_{n_h}\}$ be a suitable basis of $V_h$. 
As $u_{h,N}=\sum_{i=1}^{n_h} \ul u_i v_i$ for a coefficient vector $\ul u$, 
problem~\eqref{eq:ellipticpdediscrete} is equivalent to the linear system of equations
\begin{equation}\label{eq:discreteLSE}
	\mathbf A\ul u_h = \ul F.
\end{equation}
For any $i,j\in\{1,\dots,n_h\}$, the entries of $\mathbf A$ and $\ul F$ are given by 
\begin{equation}\label{eq:entriesLSE}
	\mathbf A_{i,j} :=\int_\cD a_{N}(\go)\nabla v_i\cdot \nabla v_j dx,
	\quad
	\text{and}
	\quad 
	\ul F_i:=  \dualpair{V'}{V}{f}{v_i},
\end{equation}
and have in general to be evaluated by numerical quadrature. 
Thereby, we commit a variational crime in the assembling of $\mathbf A$ and $\ul F$. 
As $a_n$ resp. $a$ is of low regularity, we have to make sure to choose an appropriate quadrature method, 
that does not spoil the convergence rate of the finite element approximation. 
It turns out that the midpoint rule is sufficient for ($d$-)linear elements, as we show in the remainder of this subsection. 
We restrict ourselves to the quadrature error analysis for the stiffness matrix $\mathbf A$ for brevity, 
the corresponding analysis for the load vector $\ul F$ is carried out analogously. 
We denote for any simplex/parallelotope $K\in\cK_h$ its midpoint or barycenter by $x_K^m\in K$. 
Furthermore, we define the piecewise constant approximation $\ol a_N$ of $a_n$ given by 
\begin{align*}
	\ol a_N(\go, x):= a_N(\go,x_K^m),
	\quad x\in K,\; K\in\cK_h.
\end{align*}
As we consider piecewise ($d$-)linear finite elements, approximating $\mathbf A_{ij}$ in~\eqref{eq:entriesLSE} by midpoint quadrature on each $K$ is equivalent to solving the following discrete problem: Find $\ol u_{N,h}(\go)\in V_h$ such that for all $v\in V$ it holds
\begin{equation*}
	\int_\cD \ol a_N(\go)\nabla \ol u_{N,h}(\go)\cdot\nabla v_h dx = \dualpair{V'}{V}{f}{v_h}.
\end{equation*}
There exists a.s. a unique solution $\ol u_{N,h}(\go)$ and the quadrature error is bounded in the following.

\begin{lem}\label{lem:quadrature}
	Let the assumptions of Theorem~\ref{thm:H1error} hold.
	For any $f\in H$, sufficiently small $\gk>0$ in~\eqref{eq:p-exponential} and any $r\in (0,s-\frac{d}{p})\cap(0,1]$, 
        there are constants $\ol q\in(1,\infty)$ and $C>0$ such that for any $N\in\bN$ and $h\in\mfH$ there holds 
	\begin{equation*}
		\|u_{N,h}-\ol u_{N,h}\|_{L^q(\gO; V)}+h^{-r}\|u_{N,h}-\ol u_{N,h}\|_{L^q(\gO; H)}\le C h^{r}
		\quad
		\begin{cases}
			&\text{for $q\in[1,\ol q)$ if $p=1$, and} \\
			&\text{for any $q\in [1,\infty)$ if $p>1$}.
		\end{cases}
	\end{equation*}
	Furthermore, if $s-\frac{d}{p}>1$, the statement holds for \rev{$r=1$}.
\end{lem}

\begin{proof}
	There exists a.s. a unique solution $\ol u_{N,h}(\go)$ and the distance to $u_{N,h}(\go)$ is readily bounded 
        with Proposition~\ref{prop:perturbation} by
	\begin{equation}
		\|u_{N,h}(\go)-\ol u_{N,h}(\go)\|_V 
		\le 
		\frac{\|f\|_{V'}}{a_N(\go)\ol a_N(\go)}\|a_N(\go)-\ol a_N(\go)\|_{L^\infty(\cD)}.
	\end{equation}
	Theorems~\ref{thm:randomtreereg} and Proposition~\ref{prop:truncation} show that $a_N\in L^{3q}(\gO; \cC^t(\cD))$ for 
        all $t\in(0,s-\frac{d}{p})$ and $q\ge \frac{1}{3}$. 
	For $p>1$, we may again choose any $q\in[\frac{1}{3},\infty)$, for $p=1$ we have $q\in[\frac{1}{3},\ol q)$, where $\ol q>1$ for sufficiently small $\gk>0$.
	This yields for $q\in[1,\ol q)$ with Hölder's inequality 
	\begin{align*}
		\|u_{N,h}-\ol u_{N,h}\|_{L^q(\gO; V)}
		&\le C \|a_{N,-}^2\|_{L^{3q/2}(\gO)}
		\|a_N-\ol a_N\|_{L^{3q}(\gO; L^\infty(\cD))} \\
		&\le C \|a_{N,-}\|_{L^{3q}(\gO)}^2
		\|a_N\|_{L^{3q}(\gO; \cC^t(\cD))} h^{\min(t,1)} \\
		&\le Ch^{r}.
	\end{align*}
	To prove the error with respect to $H$, we recall the duality argument from Theorem~\ref{thm:L2error}: Let $\ol e_{N,h}:=u_{N,h}-\ol u_{N,h}$ and consider for fixed $\go\in\gO$ the dual problem to find $\ol \varphi(\go)\in V$ such that for all $v\in V$ it holds
	\begin{equation*}
		\int_\cD a_N(\go)\nabla \ol \varphi(\go)\cdot\nabla v dx = \dualpair{V'}{V}{\ol e_{N,h}(\go)}{v},
	\end{equation*}
	Analogously to the proof of Theorem~\ref{thm:L2error}, we derive the pathwise estimate 
	\begin{equation*}
		\|\ol e_{N,h}(\go)\|_H^2
		\le \|a_N(\go)\|_{L^\infty(\cD)} \|\ol e_{N,h}(\go)\|_V \|(I-P_h)\ol \varphi(\go)\|_V.
	\end{equation*}		
	Provided sufficiently integrability if $p=1$, 
        we find that $\|u_{N,h}-\ol u_{N,h}\|_{L^q(\gO; H)}\le Ch^{2r}$.
\end{proof}
Note that $r$ in Lemma~\ref{lem:quadrature} is identical to $r$ in Theorems~\ref{thm:H1error} and~\ref{thm:L2error}.
Hence the quadrature error does not dominate the finite element convergence rate.
We further emphasize that Lemma~\ref{lem:quadrature} holds for arbitrary piecewise (multi-)linear elements, 
regardless if we use simplices or parallelotopes to discretize $\cD$.
\subsection{Bilinear finite element discretization}
\label{appendix:bilinearFE}
We focus on the rectangular domain $\cD=\bT^2=[0,1]^2$ in this subsection.
It is convenient to use a spatial discretization based on bilinear finite elements. 
Let therefore $h=1/n$ for a $n\in\bN$ and consider the nodes 
$\ol x_i:=ih$, $i\in\{0,\dots,n\}$. 
Then $\Xi_h:=\{\ol x_0,\dots, \ol x_n\}\subset [0,1]$ defines an equidistant mesh of $\bT^1$. 
A rectangular tesselation of $\cK_h$ of $\bT^2$ is then given by the 
$(n+1)^2$ grid points $\Xi_h^2:=\{(\ol x_{i_1}, \ol x_{i_2})|\, i_1,i_2\in\{0,\dots,n\} \}\subset\bT^2$. 
Let 
\begin{equation*}
	\phi_i(x):=\max\left\{0,1-\frac{|\ol x_i-h|}{h}\right\},\quad x\in\bR, \quad i\in\{0,\dots,n\}
\end{equation*} 
be the one-dimensional hat function basis at the nodes in $\Xi_h$.
Then, the space of bilinear finite elements corresponding to $\Xi_h^2$ resp. $\cK_h$ is given by
\begin{equation*}
	V_h:=\text{span}_\bR \{ \phi_{i_1}\otimes \phi_{i_2}, \quad i_1,i_2\in\{1,\dots,n-1\}\}.
\end{equation*}
The dyads in the tensor product basis coincide with the pointwise products
$$
	\phi_{i_1}\otimes \phi_{i_2}(x):=\phi_{i_1}(x_1)\phi_{i_2}(x_2),\quad x\in\bR^2.
$$
Note that $\dim(V_h)=(n-1)^2$ due to the homogeneous Dirichlet boundary conditions.
Now let $i, j\in\{1,\dots,(n-1)^2\}$ be indices such that 
$v_i=\phi_{i_1}\otimes \phi_{i_2}\in V_h$ and $v_j=\phi_{j_1}\otimes \phi_{j_2}\in V_h$.
The entries of the associated stiffness matrix $\mathbf A\in\bR^{(n-1)\times(n-1)}$ 
are given by 
\begin{equation}\label{eq:stiffness}
	\mathbf A_{i,j}
	:=
	\int_{\bT^2} a_N(\go,x) 
	\nabla (\phi_{i_1}\otimes \phi_{i_2})(x) \cdot \nabla (\phi_{j_1}\otimes \phi_{j_2})(x) dx.
\end{equation}
We approximate the entries of $\mathbf A$ by midpoint quadrature on each square in $\cK_h$, which may be realized by replacing $a$ in~\eqref{eq:stiffness} by a suitable piecewise constant interpolation at the midpoints as in Appendix~\ref{appendix:quadrature}.
Let the midpoint of each square $K=[\ol x_{i_1}, \ol x_{i_1+1}]\times [\ol x_{i_2}, \ol x_{i_2+1}]\in \cK_h$ be given by $x_K^m=x_{i_1,i_2}^m:=(\ol x_{i_1}+\frac{h}{2}, \ol x_{i_2}+\frac{h}{2})$ for $i_1,i_2\in\{0,\dots, n-1\}$ and define 
\begin{align*}
	\ol a_N(\go, x):= a_N(\go,x_{i_1,i_2}^m),
	\quad x\in [\ol x_{i_1}, \ol x_{i_1+1}]\times [\ol x_{i_2}, \ol x_{i_2+1}].
\end{align*}
With indices $i,j$ as above this yields
\begin{align*}
	\mathbf A_{i,j}
	&=
	\int_{\bT^2} a_N(\go) 
	\nabla (\phi_{i_1}\otimes \phi_{i_2})(x) \cdot \nabla (\phi_{j_1}\otimes \phi_{j_2})(x) dx 
	\\
	&\approx
	\int_{\bT^2} \ol a_N(\go) 
	\nabla (\phi_{i_1}\otimes \phi_{i_2})(x) \cdot \nabla (\phi_{j_1}\otimes \phi_{j_2})(x) dx
	\\
	&=
	\int_{\ol x_{i_1-1}}^{\ol x_{i_1+1}}\int_{\ol x_{i_2-1}}^{\ol x_{i_2+1}} 
	\ol a_N(\go) 
	\nabla (\phi_{i_1}\otimes \phi_{i_2})(x) \cdot \nabla (\phi_{j_1}\otimes \phi_{j_2})(x) dx
	\\
	&= a_N(\go,x_{i_1-1,i_2-1}^m) 
	\left[
	\int_{\ol x_{i_1-1}}^{\ol x_{i_1}}\phi_{i_1}'\phi_{j_1}'dx_1
	\int_{\ol x_{i_2-1}}^{\ol x_{i_2}}\phi_{i_2}\phi_{j_2}dx_2 
	+
	\int_{\ol x_{i_1-1}}^{\ol x_{i_1}}\phi_{i_1}\phi_{j_1}dx_1
	\int_{\ol x_{i_2-1}}^{\ol x_{i_2}}\phi_{i_2}'\phi_{j_2}'dx_2 
	\right]
	\\ &\quad +
	a_N(\go,x_{i_1,i_2-1}^m) 
	\left[
	\int_{\ol x_{i_1}}^{\ol x_{i_1+1}}\phi_{i_1}'\phi_{j_1}'dx_1
	\int_{\ol x_{i_2-1}}^{\ol x_{i_2}}\phi_{i_2}\phi_{j_2}dx_2 
	+
	\int_{\ol x_{i_1}}^{\ol x_{i_1+1}}\phi_{i_1}\phi_{j_1}dx_1
	\int_{\ol x_{i_2-1}}^{\ol x_{i_2}}\phi_{i_2}'\phi_{j_2}'dx_2 
	\right]
	\\ &\quad +
	a_N(\go,x_{i_1-1,i_2}^m) 
	\left[
	\int_{\ol x_{i_1-1}}^{\ol x_{i_1}}\phi_{i_1}'\phi_{j_1}'dx_1
	\int_{\ol x_{i_2}}^{\ol x_{i_2+1}}\phi_{i_2}\phi_{j_2}dx_2 
	+
	\int_{\ol x_{i_1-1}}^{\ol x_{i_1}}\phi_{i_1}\phi_{j_1}dx_1
	\int_{\ol x_{i_2}}^{\ol x_{i_2+1}}\phi_{i_2}'\phi_{j_2}'dx_2 
	\right]
	\\ &\quad +
	a_N(\go,x_{i_1,i_2}^m) 
	\left[
	\int_{\ol x_{i_1}}^{\ol x_{i_1+1}}\phi_{i_1}'\phi_{j_1}'dx_1
	\int_{\ol x_{i_2}}^{\ol x_{i_2+1}}\phi_{i_2}\phi_{j_2}dx_2 
	+
	\int_{\ol x_{i_1}}^{\ol x_{i_1+1}}\phi_{i_1}\phi_{j_1}dx_1
	\int_{\ol x_{i_2}}^{\ol x_{i_2+1}}\phi_{i_2}'\phi_{j_2}'dx_2 
	\right].
\end{align*}
We define the matrices $\mathbf S$ and $\mathbf M$ via 
\begin{align*}
	&\mathbf S_{i_1,j_1}
	:=\int_{\ol x_{i_1-1}}^{\ol x_{i_1}}\phi_{i_1}'\phi_{j_1}'dx_1
	=\begin{cases}
		\frac{1}{h},\quad i_1=j_1,\\
		-\frac{1}{h},\quad i_1=j_1+1,\\
		0, \quad\text{else},
	\end{cases}
	\\
	&\mathbf M_{i_1,j_1}
	:=\int_{\ol x_{i_1-1}}^{\ol x_{i_1}}\phi_{i_1}\phi_{j_1}dx_1,
	=\begin{cases}
		\frac{h}{3},\quad i_1=j_1,\\
		\frac{h}{6},\quad i_1=j_1+1,\\
		0, \quad\text{else}.
	\end{cases}
\end{align*}
for $i_1,j_1\in\{1,\dots,n-1\}$, and observe that 
\begin{equation*}
	\mathbf S^\top_{i_1,j_1}=\mathbf S_{j_1,i_1}
	=\int_{\ol x_{i_1}}^{\ol x_{i_1+1}}\phi_{i_1}'\phi_{j_1}'dx_1,
	\quad 
	\mathbf M^\top_{i_1,j_1}=\mathbf M_{j_1,i_1}
	=\int_{\ol x_{i_1}}^{\ol x_{i_1+1}}\phi_{i_1}\phi_{j_1}dx_1.
\end{equation*}
This yields
\begin{equation}\label{eq:assembly}
	\begin{split}
		\mathbf A_{i,j}
		&= 
		a_N(\go,x_{i_1-1,i_2-1}^m)  
		(\mathbf S_{i_1,j_1}\mathbf M_{i_2,j_2}+\mathbf M_{i_1,j_1}\mathbf S_{i_2,j_2}) \\
		&\quad+ a_N(\go,x_{i_1,i_2-1}^m) 
		(\mathbf S^\top_{i_1,j_1}\mathbf M_{i_2,j_2}+\mathbf M^\top_{i_1,j_1}\mathbf S_{i_2,j_2}) \\
		&\quad+ a_N(\go,x_{i_1-1,i_2}^m) 
		(\mathbf S_{i_1,j_1}\mathbf M^\top_{i_2,j_2}+\mathbf M_{i_1,j_1}\mathbf S^\top_{i_2,j_2}) \\
		&\quad+ a_N(\go,x_{i_1,i_2}^m) 
		(\mathbf S^\top_{i_1,j_1}\mathbf M^\top_{i_2,j_2}+\mathbf M^\top_{i_1,j_1}\mathbf S^\top_{i_2,j_2}),
	\end{split}
\end{equation}
and we only obtain a contribution to $\mathbf A_{i,j}$ if $|i_1-j_1|, |i_2-j_2|\le 1$. Hence, the representation in~\eqref{eq:assembly} may be used for an efficiently assembly of $\mathbf A$.

\subsection{Evaluation of $a_N$}
\label{appendix:coefficient}
To assemble $\mathbf A$ via~\eqref{eq:assembly} it still remains to evaluate $a_N(\go)=\exp(b_N(\go))$ at the quadrature points in $\Xi_h^2$, or, more generally, at the $d$-dimensional grid 
$\Xi_h^d:=\otimes_{i=1}^d \Xi_h\subset [0,1]^d$.
We recall that
\begin{equation*}
	b_{T,N}(\go)= \sum_{\substack{(j,k,l)\in \cI_T(\go) \\ j\le N}} \eta_jX_{j,k}^l(\go)\psi_{j,k}^l, \quad \go\in\gO.
\end{equation*}
The tensor-product representation of $\psi_{j,k}^l\in L^2(\bT^2)$ given in~\eqref{eq:torusbasis} and~\eqref{eq:psi_scaled} then shows
\begin{align*}
	b_{T,N}(\go, x)
	&= 
	\sum_{\substack{(j,k,l)\in \cI_T(\go) \\ j\le N}} \eta_j X_{j,k}^l(\go)
	2^{\frac{dj}{2}}
	\bigotimes_{i=1}^d\psi_{j+w,k_i,l(i)}(x) \\
	&=
	\sum_{\substack{(j,k,l)\in \cI_T(\go) \\ j\le N}} \eta_j X_{j,k}^l(\go)2^{\frac{dj}{2}}
	\prod_{i=1}^d\psi_{j+w,k_i,l(i)}(x_i), 
	\quad x\in\bT^d.
\end{align*}
We define the vectors 
$\ul {\psi_{j,k_i,l(i)}}:=2^{\frac{j}{2}}\psi_{j+w,k_i,l(i)}(\ol x)|_{\ol x\in\Xi}\in \bR^{n}$ 
and finally obtain  
\begin{equation*}
	b_{T,N}(\go, x)\big|_{x\in\Xi_h^2}
	= 
	\sum_{\substack{(j,k,l)\in \cI_T(\go) \\ j\le N}} \eta_j X_{j,k}^l(\go) \;
	\prod_{i=1}^d\ul {\psi_{j,k_i,l(i)}} \in \bR^{dn}.
\end{equation*}
Therefore, it is sufficient to evaluate the scaled and shifted functions $\psi_{j,k,l(i)}$ on the one-dimensional grid $\Xi_h$, the values of $b_N$, resp. $a_N$, at the $d$-dimensional grid $\Xi_h^d$ are then obtained by tensorization.
Evaluating $\psi_{j,k,l(i)}$ eventually requires to approximate the fractal functions $\phi$ and $\psi$ at a discrete set of points. 
This is feasible to arbitrary precision with the iterative \textit{Cascade algorithm}, see, e.g., \cite[Chapter 6.5]{daubechies1992ten}.
Using $J\in\bN$ iterations in the Cascade algorithm yields approximate values of $\phi, \psi$ at $2^{J}$ dyadic grid points, which are then interpolated to obtain piece-wise linear or constant approximation of continuous $\phi$ and $\psi$ interpolation.
The resulting error is of order $\cO(2^{-J\ga})$ if $\phi, \psi\in \rC^\ga(\bR)$ with $\ga\in(0,1]$. 
Consequently, we use $J_\ell:=\ceil{\frac{N_\ell t}{\ga}}$ on each level in the MLMC algorithm to match the midpoint quadrature error in Lemma~\ref{lem:quadrature}.
The cost of sampling $b_{T,N}$ on a uniform, dyadic grid is quantified in the following.
\begin{lem}\label{lem:cost}
	Let $h_\ell=2^{-(\ell+1)}$ for $\ell\in\bN_0$, let $\Xi_{h_\ell}^d:=\otimes_{i=1}^d \Xi_{h_\ell}\subset [0,1]^d$ for $d\in\bN$, and let $\cC_{sample}$ denote the random cost (in terms of work and memory required) of 
	sampling $b_{T,N}$ with respect to the grid $\Xi_{h_\ell}^d$.
	Then, there is a constant $C>0$, independent of $h_\ell$ and $N$, such that
		\begin{equation*}
		\bE(\cC_{\rm sample}) 
		\leq C \begin{cases}
			 h_\ell^{-d} (N+1)\quad&\text{if $\gb=1$, and}\\
			 h_\ell^{-d} \quad&\text{if $\gb\in(0,1)$.}
		\end{cases}
	\end{equation*}
\end{lem}
\begin{proof}
	Given that $\psi$ and $\phi$ are evaluated at the $2^{\ell+1}\in\bN$ grid points in $\Xi_{h_\ell}$, we need to calculate $2^d-1$ tensor products of scaled and translated vectors $\ul {\psi_{j,k_i,l(i)}}\in \bR^{2^{\ell+1}}$. 
	Recall from the MRA in Subsection~\ref{subsec:dVarMRA} that tensorization yields $2^{jd}$ one-periodic wavelet functions $\psi_{j,k}^{l}$ on each scale $j\in\bN_0$. Moreover, the support of $\psi_{j+w,k}^{l}$ has diameter bounded by $2^{-d(j+1-w)}$ in $\bT^d$ for fixed index $(j,k,l)\in\cI_{\bf \Psi}$, where $w\in\bN$ is a scaling factor that only depends on the choice of $\phi$ and $\psi$. 
	Hence, the number of grid points lying in the support of $\psi_{j,k}^{l}$ is given by 
	\begin{equation}
		|\text{supp} (\psi_{j+w,k}^l)\cap \Xi_{h_\ell}^d| \le 2^{-d(j+1-w)}2^{d(l+1)} = 2^{d(l-j+w)}.
	\end{equation}
	Now we also fix a realization $(T(\go), X(\go))$ of the random tree $T$  and the coefficients $X$.
	If $(j,k,l)\in\cI_T(\go)$, we multiply the corresponding $2^{d(l-j+w)}$ grid points with the coefficient $X_{j,k}^l(\go)$. Otherwise, if $(j,k,l)\notin\cI_T(\go)$, 
	there is no contribution to $b_{T,N}(\go)$ from this index. 
	Summing over all non-zero contributions and grid points thus requires computational cost of
	\begin{equation*}
		\cC_{\rm sample}
		\le  
		2\sum_{\substack{(j,k,l)\in \cI_T(\go) \\ j\le N}} 2^{d(l-j+w)}
		= 2\sum_{\substack{(j,k,l)\in \cI_{\bf \Psi} \\ j\le N}} 
		\indi_{\{(j,k,l)\in\cI_T(\go)\}} 2^{d(l-j+w)}.
	\end{equation*}
	Since $\cP=\textrm{Bin}(2^d, \gb)$, it readily follows that $P((j,k,l)\in\cI_{T(\go)})=\gb^j$, which in turn shows
	\begin{equation*}
		\bE(\cC_{\rm sample}) 
		= 2\sum_{j=0}^N 2^{dj} (2^{d}-1) \gb^j  2^{d(l-j+w)}
		\le \begin{cases}
			2^{dw+1}(2^d-1) h_\ell^{-d} (N+1)\quad&\text{if $\gb=1$, and}\\
			 \frac{2^{dw+1}(2^d-1)}{1-\gb} h_\ell^{-d} \quad&\text{if $\gb\in(0,1)$.}\\
		\end{cases}
	\end{equation*}
\end{proof}
\begin{rem}
	Lemma~\ref{lem:cost} shows that for given $\gb\in(0,1)$, the expected cost of sampling $b_{N,T}$ is bounded by $Ch_\ell^{-d}$ \emph{uniformly with respect to  $N\in\bN$}. Thus, the condition that a sample of $u_{N_\ell,h_\ell}$ may be realized with (expected) work $\cO(h_\ell^{-d})$ from Assumption~\ref{ass:refinement} is indeed justified. 
	On the other hand, we note that $C=C(d,\gb)\to\infty$ as $\gb\to 1$, resulting in a possibly large hidden constant within the asymptotic costs in Theorem~\ref{thm:work}. However, if we choose the error balancing $N\propto |\log(h_\ell)|$ according to~\eqref{eq:balancederror}, we still recover the only slightly worse complexity bound of $\cO(h_\ell^{-d}|\log(h_\ell)|)$ per sample in the limit $\gb=1$. 
\end{rem}
\end{document}

%% file: macros_as.tex
\theoremstyle{plain}
\newtheorem{thm}{Theorem}[section]
\theoremstyle{plain}
\newtheorem{lem}[thm]{Lemma}
\newtheorem{prop}[thm]{Proposition}
\newtheorem{cor}[thm]{Corollary}
\theoremstyle{definition}
\newtheorem{defi}[thm]{Definition}
\newtheorem{rem}[thm]{Remark}
\newtheorem{assumption}[thm]{Assumption}

\newcommand{\ga}{\alpha}
\newcommand{\gb}{\beta}
\newcommand{\gd}{\delta}
\newcommand{\eps}{\ensuremath{\varepsilon}}

\renewcommand{\gg}{\gamma}
\newcommand{\gk}{\kappa}

\newcommand{\go}{\omega}
\newcommand{\gs}{\sigma}
\newcommand{\gt}{\theta}

\newcommand{\gD}{\Delta}

\newcommand{\gO}{\Omega}

\newcommand{\cA}{\mathcal{A}}
\newcommand{\cB}{\mathcal{B}}
\newcommand{\cC}{\mathcal{C}}
\newcommand{\cD}{\mathcal{D}}
\newcommand{\cE}{\mathcal{E}}

\newcommand{\cI}{\mathcal{I}}

\newcommand{\cK}{\mathcal{K}}
\newcommand{\cL}{\mathcal{L}}

\newcommand{\cO}{\mathcal{O}}
\newcommand{\cP}{\mathcal{P}}

\newcommand{\cU}{\mathcal{U}}
\newcommand{\cV}{\mathcal{V}}

\newcommand{\cX}{\mathcal{X}}
\newcommand{\cY}{\mathcal{Y}}

\newcommand{\bE}{\mathbb{E}}

\newcommand{\bN}{\mathbb{N}}

\newcommand{\bP}{\mathbb{P}}
\newcommand{\bQ}{\mathbb{Q}}
\newcommand{\bR}{\mathbb{R}}

\newcommand{\bT}{\mathbb{T}}

\newcommand{\bZ}{\mathbb{Z}}

\newcommand{\rC}{{\rm C}}
\newcommand{\rD}{{\rm D}}

\newcommand{\rS}{{\rm S}}

\newcommand{\mfH}{\mathfrak{H}}

\newcommand{\mfT}{\mathfrak{T}}

\newcommand*{\lrscript}[5]{{\vphantom{#1}}_{#2}^{#3}{#1}_{#4}^{#5}}
\newcommand{\dualpair}[4]{\ensuremath{\lrscript{\langle}{#1}{}{}{} #3 ,#4 \rangle_{#2}}}
\newcommand{\be}{\begin {equation}}
\newcommand{\ee}{\end  {equation}}
\newcommand{\bee}{\begin {equation*}}
\newcommand{\eee}{\end {equation*}}
\newcommand{\ol}{\overline}
\newcommand{\ul}{\underline}
\newcommand{\floor}[1]{\lfloor #1 \rfloor}
\newcommand{\ceil}[1]{\lceil #1 \rceil}

\newcommand{\var}{{\rm Var}}

\newcommand{\indi}{\mathds{1}}

\DeclareMathOperator*{\essinf}{\text{ess\,inf}}
\DeclareMathOperator*{\esssup}{\text{ess\,sup}}

\newcommand{\MC}{Monte Carlo}
\newcommand{\MLMC}{multilevel \MC}
\newcommand{\MMLMC}{Multilevel \MC}

\newcommand{\UQ}{uncertainty quantification}


%% file: GWtree.tex
\centering
{\small
	\bf \begin{forest}
		for tree={
			align=center,
			edge = dashed
		}
		[{$\varrho=()$ \quad $(0,0)$}, draw, 
		[{$(1)$ \quad $(1,0)$}, draw, edge=solid
		[{$(1, 1)$ \quad $(2,0)$}, draw, edge=solid
		[{(X, X)}], [{$(1, 1, 2)$ \quad $(3,1)$}\\ $\vdots$, draw, edge=solid]
		]
		[{$(1,2)$ \quad $(2,1)$}, draw, edge=solid
		[{(X)}], [{(X)}]
		]
		]
		[{$(2)$ \quad $(1,1)$}, draw, edge=solid
		[{(X)}
		[{(X)}], [{(X)}]
		]
		[{$(2, 2)$ \quad $(2,3)$}, draw, edge=solid
		[{$(2, 2, 1)$ \quad $(3,6)$}\\ $\vdots$, draw, edge=solid], [{(X)}]
		]
		]
		]
	\end{forest}
}

%% file: besovtreeuq_R2.bbl
\begin{thebibliography}{10}

\bibitem{abraham2015GW-Trees}
R.~Abraham and J.-F. Delmas.
\newblock An introduction to {G}alton-{W}atson trees and their local limits.
\newblock {\em arXiv preprint arXiv:1506.05571}, 2015.

\bibitem{AchdouRamif}
Y.~Achdou, C.~Sabot, and N.~Tchou.
\newblock Diffusion and propagation problems in some ramified domains with a
  fractal boundary.
\newblock {\em M2AN Math. Model. Numer. Anal.}, 40(4):623--652, 2006.

\bibitem{adams2003sobolev}
R.~A. Adams and J.~J. Fournier.
\newblock {\em Sobolev Spaces}.
\newblock Elsevier, 2nd edition, 2003.

\bibitem{AgaDasHel2021}
S.~Agapiou, M.~Dashti, and T.~Helin.
\newblock Rates of contraction of posterior distributions based on
  {$p$}-exponential priors.
\newblock {\em Bernoulli}, 27(3):1616--1642, 2021.

\bibitem{AB06}
C.~D. Aliprantis and K.~Border.
\newblock {\em Infinite Dimensional Analysis: A Hitchhiker's Guide}.
\newblock Springer, 2006.

\bibitem{barth2011multi}
A.~Barth, C.~Schwab, and N.~Zollinger.
\newblock Multi-level {M}onte {C}arlo finite element method for elliptic {PDE}s
  with stochastic coefficients.
\newblock {\em Numerische Mathematik}, 119(1):123--161, 2011.

\bibitem{BS18b}
A.~Barth and A.~Stein.
\newblock A study of elliptic partial differential equations with jump
  diffusion coefficients.
\newblock {\em SIAM/ASA Journal on Uncertainty Quantification},
  6(4):1707–1743, 2018.

\bibitem{FracSurvey2019}
I.~Berre, F.~Doster, and E.~Keilegavlen.
\newblock Flow in fractured porous media: a review of conceptual models and
  discretization approaches.
\newblock {\em Transp. Porous Media}, 130(1):215--236, 2019.

\bibitem{BS08}
S.~C. Brenner and L.~R. Scott.
\newblock {\em The Mathematical Theory of Finite Element Methods}, volume~3.
\newblock Springer, 2008.

\bibitem{cliffe2011multilevel}
K.~A. Cliffe, M.~B. Giles, R.~Scheichl, and A.~L. Teckentrup.
\newblock Multilevel {M}onte {C}arlo methods and applications to elliptic
  {PDEs} with random coefficients.
\newblock {\em Computing and Visualization in Science}, 14(1):3--15, 2011.

\bibitem{DHS12}
M.~Dashti, S.~Harris, and A.~Stuart.
\newblock {Besov priors for Bayesian inverse problems}.
\newblock {\em Inverse Problems and Imaging}, 6(2):183--200, 2012.

\bibitem{daubechies1992ten}
I.~Daubechies.
\newblock {\em Ten {L}ectures on {W}avelets}.
\newblock SIAM, 1992.

\bibitem{DNSZ22}
D.~Dung, V.~Nguyen, C.~Schwab, and J.~Zech.
\newblock Analyticity and sparsity in uncertainty quantification for {PDE}s
  with {G}aussian random field inputs.
\newblock Technical Report 2022-02, Seminar for Applied Mathematics, ETH
  Z{\"u}rich, Switzerland, 2022.
\newblock (to appear in Springer LNM (2023)).

\bibitem{giles2008multilevel}
M.~B. Giles.
\newblock Multilevel {M}onte {C}arlo path simulation.
\newblock {\em Operations research}, 56(3):607--617, 2008.

\bibitem{giles2015multilevel}
M.~B. Giles.
\newblock Multilevel {M}onte {C}arlo methods.
\newblock {\em Acta Numerica}, 24:259--328, 2015.

\bibitem{grisvard2011elliptic}
P.~Grisvard.
\newblock {\em Elliptic {P}roblems in {N}onsmooth {D}omains}.
\newblock SIAM, 2011.

\bibitem{H17}
W.~Hackbusch.
\newblock {\em Elliptic Differential Equations: Theory and Numerical
  Treatment}, volume~18.
\newblock Springer, 2nd edition, 2017.

\bibitem{heinrich2001multilevel}
S.~Heinrich.
\newblock Multilevel {M}onte {C}arlo methods.
\newblock In {\em International Conference on Large-Scale Scientific
  Computing}, pages 58--67. Springer, 2001.

\bibitem{HL11}
T.~Helin and M.~Lassas.
\newblock Hierarchical models in statistical inverse problems and the
  {M}umford-{S}hah functional.
\newblock {\em Inverse Problems}, 27(1):015008, 32, 2011.

\bibitem{herrmann2019multilevel}
L.~Herrmann and C.~Schwab.
\newblock Multilevel quasi-{M}onte {C}arlo integration with product weights for
  elliptic {PDE}s with lognormal coefficients.
\newblock {\em ESAIM: Mathematical Modelling and Numerical Analysis},
  53(5):1507--1552, 2019.

\bibitem{hoang2013sparse}
V.~H. Hoang and C.~Schwab.
\newblock Sparse tensor {G}alerkin discretization of parametric and random
  parabolic {PDE}s---analytic regularity and generalized polynomial chaos
  approximation.
\newblock {\em SIAM Journal on Mathematical Analysis}, 45(5):3050--3083, 2013.

\bibitem{HosNig17}
B.~Hosseini and N.~Nigam.
\newblock Well-posed {B}ayesian inverse problems: priors with exponential
  tails.
\newblock {\em SIAM/ASA J. Uncertain. Quantif.}, 5(1):436--465, 2017.

\bibitem{KLSS21}
H.~Kekkonen, M.~Lassas, E.~Saksman, and S.~Siltanen.
\newblock Random tree {Besov} priors -- towards fractal imaging.
\newblock {\em arXiv preprint arXiv:2103.00574}, 2021.

\bibitem{LassasBesov09}
E.~Saksman, M.~Lassas, and S.~Siltanen.
\newblock {Discretization-invariant Bayesian inversion and Besov space priors}.
\newblock {\em Inverse Problems and Imaging}, 3(1):87–122, 2009.

\bibitem{TriebelTOFS2}
H.~Triebel.
\newblock {\em Theory of Function Spaces II}.
\newblock Modern Birkhäuser Classics. Birkhäuser, 2nd edition, 2000.

\bibitem{TriebelFctSpcDom}
H.~Triebel.
\newblock {\em Function {S}paces and {W}avelets on {D}omains}, volume~7 of {\em
  EMS Tracts in Mathematics}.
\newblock European Mathematical Society (EMS), Z\"{u}rich, 2008.

\bibitem{TriebelTOFS4}
H.~Triebel.
\newblock {\em Theory of Function Spaces IV}, volume 107 of {\em Monographs in
  Mathematics}.
\newblock Birkhäuser, 2020.

\bibitem{zech2020convergence}
J.~Zech and C.~Schwab.
\newblock Convergence rates of high dimensional {S}molyak quadrature.
\newblock {\em ESAIM: Mathematical Modelling and Numerical Analysis},
  54(4):1259--1307, 2020.

\end{thebibliography}
